\documentclass[numbers,webpdf,imanum]{ima-authoring-template}
\setcitestyle{authoryear,round}
\PassOptionsToPackage{dvipsnames}{xcolor}
\usepackage{bbm}
\usepackage{booktabs}
\usepackage{tabularx}
\usepackage{systeme}
\usepackage{soul,todonotes}
\usepackage{pgfplots}
\pgfplotsset{compat=1.18}
\usepackage{listings}
\usepackage{nth}
\usepackage{courier}
\usepackage{nicefrac}
\usepackage{mathrsfs}
\usepackage{mathtools}
\usepackage{algpseudocode}
\usepackage{wrapfig}
\usepackage{ragged2e}
\usepackage{eurosym}
\usepackage{indentfirst}
\usepackage{pifont}
\usepackage[loadonly]{enumitem}
\renewlist{enumerate}{enumerate}{4}
\setlist[enumerate]{label=\arabic*.}
\makeatletter
\newcommand{\myitem}[1]{%
	\item[#1]\protected@edef\@currentlabel{#1}%
}
\makeatother

\DeclareMathOperator*{\argmin}{argmin}
\DeclareMathOperator{\VCdim}{VCdim}

\DeclareMathOperator{\Pdim}{Pdim}
\DeclareMathOperator{\format}{format}
\DeclareMathOperator{\ReQU}{ReQU}
\renewcommand{\thesection}{\arabic{section}}
\newcommand{\ind}{\mathbbm{1}}
\newcommand{\sgn}{\operatorname{sign}}
\newcommand{\R}{\mathbb{R}}
\newcommand{\N}{\mathbb{N}}

\newcommand{\Func}{\mathbb{F}}
\newcommand{\F}{\mathcal{F}}
\newcommand{\size}{\mathcal{S}}
\newcommand{\partition}{\mathcal{P}}
\newcommand{\bigO}{\mathcal{O}}
\newcommand{\Exp}{\mathbb{E}}
\newcommand{\bW}{\mathbf{W}}
\newcommand{\bb}{\mathbf{b}}
\newcommand{\bx}{\mathbf{x}}
\newcommand{\bm}{\mathbf{m}}
\DeclareSymbolFont{vectorGreek}{OML}{cmm}{b}{it}
\DeclareMathSymbol{\bnu}{\mathord}{vectorGreek}{"17}
\DeclareMathSymbol{\btheta}{\mathord}{vectorGreek}{"12}
\newcommand{\paOm}{\partial \Omega}

\theoremstyle{thmstyletwo}
\newtheorem{theorem}{Theorem}[section]

\newtheorem{example}[theorem]{Example}

\newtheorem{proposition}[theorem]{Proposition}
\newtheorem{corollary}[theorem]{Corollary}
\newtheorem{remark}[theorem]{Remark}
\newtheorem{definition}[theorem]{Definition}
\newtheorem{lemma}[theorem]{Lemma}

\numberwithin{equation}{section}

\begin{document}
	
	\DOI{xxx}
	\copyrightyear{202X}
	\vol{00}
	\pubyear{202X}
	\access{Advance Access Publication Date: Day Month Year}
	\appnotes{Paper}
	%\copyrightstatement{Published by Oxford University Press on behalf of the Institute of Mathematics and its Applications. All rights reserved.}
	\firstpage{1}
	
	%\subtitle{Subject Section}
	
	\title[$H^2(\Omega)$ Error Bounds for Boundary-Adapted PINNs]{$H^2(\Omega)$ A Priori Error Bounds for Boundary-Adapted PINNs Solving Elliptic Dirichlet Problems}
	
	\author{Nathanael Tepakbong*
		\address{\orgdiv{Department of Mathematics}, \orgname{City University of Hong Kong}, \orgaddress{\street{83 Tat Chee Avenue, Kowloon Tong}, \country{Hong Kong}}}}
	\author{Jun Fan
		\address{\orgdiv{Department of Mathematics}, \orgname{Hong Kong Baptist University}, \orgaddress{\street{224 Waterloo Road, Kowloon Tong}, \country{Hong Kong}}}}
	\author{Xiang Zhou
		\address{\orgdiv{Department of Mathematics}, \orgname{City University of Hong Kong}, \orgaddress{\street{83 Tat Chee Avenue, Kowloon Tong}, \country{Hong Kong}}}}
	\author{Ding-Xuan Zhou
		\address{\orgdiv{School of Mathematics and Statistics}, \orgname{The University of Sydney}, \orgaddress{\street{Camperdown}, \postcode{2050}, \state{New South Wales}, \country{Australia}}}}
	
	\authormark{N. Tepakbong et al.}
	
	\corresp[*]{Corresponding author: \href{email:ntepakbo-c@my.cityu.edu.hk}{ntepakbo-c@my.cityu.edu.hk}}
	
	\received{Date}{0}{Year}
	\revised{Date}{0}{Year}
	\accepted{Date}{0}{Year}
	
	%\editor{Associate Editor: Name}
	
	\abstract{
		\noindent We study the numerical resolution of uniformly elliptic Dirichlet boundary value problems (BVPs) by minimization of the PDE residual over a hypothesis space of Neural Networks made to satisfy exactly the boundary conditions through multiplication by a user-defined \emph{distance-to-boundary approximation} $\rho$. While an $L^2(\partial\Omega)$ boundary penalty cannot generally guarantee $H^2(\Omega)$ error control, we show that exact enforcement can, provided $\rho$ is chosen appropriately. More precisely, we identify a sharp sufficient condition on $\rho$ preserving the regularity needed for $H^2(\Omega)$ error control, provide practical constructions of such $\rho$ for general smooth domains $\Omega$, and refer to the resulting networks as \textit{boundary-adapted}. For boundary-adapted Rectified Quadratic Unit (ReQU) and hyperbolic tangent (tanh) networks, we then derive explicit high-probability $H^2(\Omega)$ convergence rates by combining approximation-theoretic and statistical-learning arguments, making explicit the dependence on $\rho$. Numerical experiments support our theoretical results, showing that appropriate choices of $\rho$ improve accuracy and convergence, while poorly chosen distance functions can substantially degrade the solution. Our proof also yields new VC-dimension bounds for hypothesis spaces of higher-order derivatives of ReQU and tanh networks, together with new approximation bounds for shallow ReQU networks in higher-order Sobolev norms, all of which are of important independent interest.
	}
	
	
	%\begin{MSCcodes}
	%	68Q32, 68T07, 65N12, 41A30, 41A28.
	%\end{MSCcodes}
	
	\maketitle
	
	\section{Introduction}
	
	Boundary Value Problems (BVPs) of Dirichlet type arise frequently across a plethora of scientific and industrial applications. In their most general formulation, they consist in looking for a solution ${u^\star:\Omega\to\R}$ of a partial differential equation (PDE) of the form
	\begin{equation}
		\label{eqn:general_dirichlet_problem}
		\begin{cases}
			\mathscr{L}u= f &\text{in }\Omega,\\
			u =g  &\text{on }\partial\Omega,
		\end{cases}
	\end{equation}
	where $\Omega\subset\R^d$ is a domain, $\mathscr{L}$ is a differential operator, and $f:\Omega\to\R$ and $g:\bar\Omega\to\R$ are known \textit{forcing term} and \textit{boundary data}. Due to the ubiquity of such problems, which in practice do not admit analytical solutions, much literature has been dedicated to the numerical approximation of solutions to \eqref{eqn:general_dirichlet_problem}. %\citep{evans2022partial}.
	
	A family of numerical methods which have traditionally proved extremely successful in tackling such problems fall under the umbrella of \textit{finite element methods} (FEM), and other related algorithms such as \textit{finite difference} or \textit{finite volume} methods \citep{brenner2007mathematical}. Such approaches typically work by splitting the domain $\Omega$ into finitely many mesh elements, and approximating the solution with simpler trial functions on each element while ensuring that the PDE is satisfied locally. This turns the task of solving \eqref{eqn:general_dirichlet_problem} into a linear algebra task, which makes such schemes computationally tractable and amenable to theoretical analysis. Despite their essential ubiquity in most industrial applications for problems in dimensions $d\le 3$, these methods still suffer from a number of well-known issues. On top of poor scalability to higher-dimensional PDEs, these methods suffer from high sensitivity to geometry, mesh design, and a dire lack of flexibility.

	\subsection{Physics-informed neural networks}
	
	Physics-Informed Neural Networks (PINNs) \citep{raissi2019physics} and related variants \citep{yu2018deep} have recently emerged as a promising mesh-free paradigm for approximating PDE solutions by embedding the equations directly into the neural network training process. In this approach, a neural network is trained to minimize the PDE residual evaluated at discrete collocation points within the domain, together with terms enforcing boundary and initial conditions. This residual-based formulation enables PINNs to leverage dedicated training algorithms such as stochastic gradient descent or automatic differentiation, and in turn provides a path towards flexible and efficient solving of complex scientific problems, specifically for high-dimensional or irregular domains \citep{cuomo2022scientific}.
	
	A critical aspect of PINN training is the enforcement of boundary conditions, which can significantly influence solution accuracy and physical meaningfulness \citep{Chen2020-CMS}. Common strategies include ``soft'' enforcement, where boundary error is penalized in the loss function, and ``hard" enforcement, which constructs the neural network output to satisfy the boundary conditions exactly through multiplicative or additive ansatzes \citep{sukumar2022exact}. While easier to implement, soft boundary-enforcement methods often suffer from loss imbalance, leading to poor boundary adherence unless penalty weights have gone through an expensive hyperparameter tuning process \citep{krishnapriyan2021characterizing}. Hard enforcement of the boundary conditions avoid these issues entirely, making them preferable when available. Other basic problems however naturally arise when trying to implement them in practice: how to construct, for an arbitrary domain $\Omega$, a neural network ansatz which satisfies the boundary conditions exactly? How is the expressivity of the resulting \textit{boundary-enforced} neural network affected compared to the original architecture? Can we establish rigorous a priori error guarantees for such architectures, providing some guidance on the boundary-enforcement procedure along the way?
	
	\subsection{Our contributions}
	
	We study ``boundary-enforced PINNs" for solving second-order BVPs of the form \eqref{eqn:general_dirichlet_problem}, where the operator $\mathscr{L}$ is assumed uniformly elliptic. Specifically, we consider functions of the form $x\mapsto g(x)+\rho(x)\cdot u_{\btheta}(x)$, where $\rho$ vanishes on $\partial\Omega$, $g$ is defined on $\bar\Omega$, and $u_{\btheta}$ is a \textit{fully connected} (or \textit{feedforward}) neural network (FNN) parameterized by $\btheta$. Our main contributions can be summarized as follows:
	
	\begin{itemize}
		\item By a classical estimate from elliptic PDE theory \eqref{eqn:strong_convexity_lower_bound}, we observe that adding an $L^2(\partial \Omega)$ penalty to the PDE residual, as is commonly done in PINN training, is \textit{inconsistent} with $H^2(\Omega)$ a priori error estimates. Instead, we show in Proposition~\ref{prop:quotient_regularity} that a practical way to guarantee such $H^2(\Omega)$ a priori estimates for general PINNs is to multiply the neural network ansatz $u_{\btheta}$ by a distance approximation $\rho$ \textit{normalized to the first order}, as per Definition~\ref{def:order_m_normalization}. We refer to PINNs constructed this way as \emph{boundary-adapted PINNs}.
		
		\item In Theorem~\ref{thm:approx_requ}, we derive new $W^{2,\infty}(\Omega)$ approximation error bounds for shallow neural networks using the rectified quadratic unit (ReQU) activation. This extends in a non-trivial way previous $L^\infty(\Omega)$ approximation error bounds for shallow ReLU and ReLU$^k$ networks obtained using the theory of variation spaces and harmonic analysis on spheres \citep{bach2017breaking, yang2025optimal}.
		
		\item We prove in Theorem~\ref{thm:vcdim_bounds} new bounds on the VC dimension of higher order derivatives of ReQU and tanh fully connected neural networks. These bounds enable us to apply an oracle inequality from \citep{lei2025solving}, which in turn allows error control without imposing restrictive bounds on the magnitude of the network weights. This makes our theoretical results more relevant for practical PINN implementations.
		
		\item Based on the above results, we derive in Theorem~\ref{thm:main} a complete a priori error analysis for boundary-adapted PINN solutions of second-order elliptic PDEs using either ReQU or tanh activations. Our analysis incorporates both approximation and statistical errors and gives explicit high-probability $H^2(\Omega)$ convergence rates in terms of the number of collocation points. To the best of our knowledge, this is the first rigorous error bound for PINNs that applies to the popular $\tanh$ activation, and explicitly accounts for the enforcement of boundary conditions.
	\end{itemize}
	
	Our theoretical findings are illustrated with numerical experiments highlighting the critical role of the choice of $\rho$ when solving Dirichlet BVPs\footnote{The implementation details and source code are available at \url{https://github.com/n2pakbong/boundary-adapted-pinns}.}. Our results confirm that poorly chosen distance-approximation ansatzes and soft penalty approaches yield significantly worse solutions than boundary-adapted PINNs, in agreement with empirical observations in \citep{berrone2023enforcing}.
	
	\subsection{Related works}

	While a number of works, such as \citep{shin2020convergence, lu2021learning, jiao2022rate, shin2023error}, have provided some theoretical guarantees for PINN solutions of elliptic problems, our understanding of these algorithms remains far from complete. Indeed, besides the challenging issue of training, for which only few results are known \citep{rathore2024challenges, luo2024two, jiao2025drm}, the effect of explicitly enforcing the boundary conditions on their approximation and statistical error analysis has remained largely under-explored.  
	
	Indeed, the analysis in \citep{lu2021machine} gives $H^2$ rates for a Schrödinger equation on the unit cube, under the unrealistic assumption that the network trial functions have exactly zero boundary values. Likewise, the PDEs studied in \citep{lei2025solving} are posed on the sphere, where there are no boundary conditions to impose, and leverage that to obtain rates in strong $H^s$ norms. In contrasts, works taking the $L^2(\paOm)$ penalty into account, such as \citep{jiao2022rate,xu2025refined}, obtain explicit rates in the weaker $H^{1/2}(\Omega)$ or $L^2(\Omega)$ norms. Other approaches strengthen the boundary norm, yielding improved bounds while still falling short of $H^{2}(\Omega)$ error bounds \citep{wu2023convergence, jiang2024generalization, bonito2025convergence}. We are thus the first to address, in this work, the issue of obtaining $H^2(\Omega)$ error bounds for networks made to satisfy exactly the Dirichlet boundary conditions in a fully constructive manner.
	
	%[TODO: tableau à refaire. papier/problème ou edp/activations/méthode d"imposition des conditions au bord/vitesse de convergence. plus citer \citep{lei2025solving} dans le tableau]
	
	The advantage of exact boundary enforcement for residual minimization was established in \citep{muller2022exactboundary}: for the Poisson problem, they gave an $H^2$ residual estimate and an $H^2$ best-approximation bound for ``boundary-enforced" trial classes. The analysis in \citep{zeinhofer2025unified} extends these estimates to several linear PDEs, including diffusion equations, and recovers $H^2$ control when the boundary values are exact. These estimates however leave open the question of quantifying the approximation and statistical error of the resulting boundary-enforced network classes for arbitrary smooth domains $\Omega$. These are the questions we resolve in this work with explicit $H^2(\Omega)$ a priori error bounds in terms of the number of collocation points. Other related works include \citep{mishra2023estimates}, which bounds solution errors in terms of measured training errors and quadrature under a PDE stability assumption, and \citep{bonito2025convergence} which gives optimal $H^1$ recovery from point values of the forcing and boundary data, using discretizations of an $H^{-1}(\Omega)$ residual and an $H^{1/2}(\paOm)$ boundary norm. We compare these estimates with ours in Table~\ref{tab:elliptic_pinn_comparison}.
	
	\begin{table}[htbp]
		\centering
		\small
		\resizebox{\textwidth}{!}{%
			\begin{tabularx}{\textwidth}{@{}>{\raggedright\arraybackslash}p{0.17\textwidth}>{\raggedright\arraybackslash}X>{\raggedright\arraybackslash}X>{\raggedright\arraybackslash}X>{\raggedright\arraybackslash}p{0.27\textwidth}@{}}
				\toprule
				Paper & PDE and domain & Architecture / activation & Boundary enforcement method & A priori error estimate \\
				\midrule
				\citep{muller2022exactboundary} & Poisson, bounded $C^{1,1}$ or convex domain & Any $H^2$ trial class & Exact enforcement via $\rho u_{\btheta}$ ansatz & $H^2$ error controlled by $L^2$ PDE residual, no sampling rate \\
				\citep{zeinhofer2025unified} & Diffusion, bounded $C^{1,1}$ domain & Any $H^2$ trial class & Exact satisfaction assumed & $H^2$ error controlled by $L^2$ PDE residual, no sampling rate \\
				\citep{lu2021machine} & Schr\"odinger, unit cube, $u^\star\in H^r$ & Sparse deep ReLU$^3$ network & Exact satisfaction assumed & $H^2$: $n^{-(r-2)/(d+2r-4)}$ \\
				\citep{mishra2023estimates} & Stable PDEs on bounded domains & Smooth FNNs & $L^p$ boundary residual & Solution error bounded by training and quadrature errors \\
				\citep{jiao2022rate} & Linear elliptic, $C^3$ domain and solution & deep ReLU$^3$ & $L^2$ boundary penalty & $H^{1/2}$: $n^{-1/(2d+4+\delta)}$, $\delta>0$ (mean) \\
				\citep{xu2025refined} & Linear elliptic, $C^r$ domain and solution & Sparse deep ReLU$^3$ & $L^2$ boundary penalty & $H^{1/2}$: $n^{-(r-2)/(d+2r-4)}$ \\
				\citep{bonito2025convergence} & Poisson, unit cube & Activation not specified &  $H^{1/2}$ boundary penalty & $H^1$ recovery from sampled forcing and boundary data \\
				\citep{lei2025solving} & Order-$s$ stable PDE on $\mathbb S^{d-1}$, $W^{r,\infty}$ target & ReLU-ReLU$^k$ (CNN + FNN) & N/A & $H^s$: $n^{-(r-s)/(2(r-s+d-1))}$ \\
				\textbf{This work} & Linear elliptic, $C^{r+1,\kappa}$ domain and solution & Shallow ReQU & Exact enforcement via $g+\rho u_{\btheta}$ ansatz & $H^2$: $n^{-\gamma/(d+2\gamma)}$ \\
				\textbf{This work} & Same setting & Two-hidden-layer tanh & Same construction & $H^2$: $n^{-(r-2)/(2(2d+r-1))}$ \\
				\bottomrule
			\end{tabularx}%
		}
		\caption{A priori error estimates and boundary enforcement methods of related works. Bounds are up to logarithmic factors. $n$ is the collocation sample size, and $\gamma:=\min\{r-2,(d+1)/2\}$.}
		\label{tab:elliptic_pinn_comparison}
	\end{table}

	\subsection{Notations}
	
	For a finite-dimensional vector $\xi\in\R^m$, we denote by $\|\xi\|_p=(\sum_{i=1}^m|\xi_i|^p)^{1/p}$ its $\ell^p$ norm for $1\le p<\infty$, and by $\|\xi\|_\infty=\max_i|\xi_i|$ its maximum norm. We denote the nonnegative integers by $\N$. For multi-indices $\alpha,\beta\in\N^d$, we let $\partial^\alpha=\partial_1^{\alpha_1}\cdots\partial_d^{\alpha_d}$ denote the $\alpha$-th order differential operator, and by $\binom\alpha\beta=\prod_i\binom{\alpha_i}{\beta_i}$ the multi-index binomial coefficient.
	
	Sobolev spaces are denoted $W^{s,p}(\Omega)$ and carry their usual norms, with $H^s(\Omega):=W^{s,2}(\Omega)$ and $H^1_0(\Omega)$ denoting the zero-trace subspace of $H^1(\Omega)$ \citep{evans2022partial}. For integer $s\ge0$, we use $\|v\|_{W^{s,\infty}(\Omega)}=\max_{\|\alpha\|_1\le s}\|\partial^\alpha v\|_{L^\infty(\Omega)}$. For two quantities $a,b$, the notation $a\lesssim b$ means $a\le Cb$ for a positive constant $C$ independent of the parameters of interest. We write $a\asymp b$ if both $a\lesssim b$ and $b\lesssim a$ hold, and $a\sim b$ for asymptotic equivalence.

	\subsection{Statement on AI usage}
	The essentiality of the prose and mathematical content of this work is due to the authors, and the only use of AI in this work has been to generate the code used for the numerical experiments, together with literature review and proofreading a finalized version of this manuscript. We acknowledge that during the proofreading process, an AI agent (GPT 5.6 Sol) found an improved proof of Proposition~\ref{prop:quotient_regularity}, and identified a loose bound on the parameter count in our error analysis for tanh networks in Theorem~\ref{thm:main}. These findings led to minor improvements in our smoothness Assumptions~\eqref{eqn:assumption1_smooth_boundary}, \eqref{eqn:assumption3_regular_data} and \eqref{eqn:assumption4_regular_coeffs}, and in our tanh convergence rate in Theorem~\ref{thm:main} and are thus included in the present work.
	
	\section{Problem setup}
	
%	We start by formally introducing the mean escape time (MET) problem, as it will serve as a concrete example to illustrate our theory. We stress however that the results we will show apply to all second-order elliptic PDEs satisfying appropriate regularity conditions.  
%	
%	Let $d\ge2$ be an integer, and for all $x\in\R^d$, denote by $X_t^x$ the unique --- up to standard regularity conditions --- solution of the stochastic differential equation \citep{pavliotis2014stochastic}
%	\begin{equation}
%		\label{eqn:stoch_process}
%		dX_t^x = b(X_t^x) dt + \varsigma(X_t^x) dW_t,\quad X_0 = x
%	\end{equation}
%	
%	where $(W_t)$ is a standard Brownian motion in $\R^d$, while $b:\R^d\to\R^d$ and $\varsigma:\R^d\to\R^{d\times d}$ are respectively called the \textit{drift} and \textit{diffusion} coefficients of the stochastic process $X_t^x$.
%	
%	Given a bounded, open domain $\Omega\subset\R^d$ with smooth boundary $\paOm$, we define for all $x\in\Omega$ the \textit{first passage time}, or \textit{first exit time} as the time when $X_t^x $ first exits $\Omega$:
%	\[\tau_\Omega^x := \inf\{t\ge 0 : X_t^x\not\in\Omega\}.\]
%	We then define the \textit{mean escape time} (also known as mean exit time or mean first passage time) as the expected valued of $\tau_\Omega^x$:
%	\begin{equation}
%		\label{eqn:met_definition}
%		\tau(x):=\Exp\tau_\Omega^x = \Exp\left[\inf\{t\ge 0 : X_t^x\not\in\Omega\}\right].
%	\end{equation} 
%	
	Let $d\ge2$ be an integer, $\Omega\subset\R^d$ be an open, bounded and connected set, $f:\Omega\to\R$ and $g:\bar{\Omega}\to\R$. We consider the boundary value problem:
	
	\begin{equation}
		\label{eqn:bvp}
		\begin{cases}
			\mathscr{L}u= f &\text{in }\Omega,\\
			u =g  &\text{on }\partial\Omega,
		\end{cases}
	\end{equation}
	where $\mathscr{L}$ is a differential operator acting on sufficiently smooth functions $\varphi$ by
	\begin{equation}
		\label{eqn:generator}
		\mathscr{L}\varphi: x\mapsto c(x)\varphi(x) + \sum_{i=1}^d b_i(x)\frac{\partial \varphi}{\partial x_i}(x) + \sum_{i,j=1}^d \Sigma_{ij}(x)\frac{\partial^2 \varphi}{\partial x_i \partial x_j}(x),
	\end{equation} 
	where the coefficients $c$,$(b_i)_{1\le i \le d}$, and $(\Sigma_{ij})_{1\le i,j \le d}$ are all real-valued functions defined on $\Omega$.
	
	Let $r\ge 3$ be a natural number and $0<\kappa<1$ be real. We will assume that the operator $\mathscr L$ and the domain $\Omega$ satisfy the following assumptions:
	\begin{enumerate}[label=\textbf{A.\arabic*}]
		\item \ \ \label{eqn:assumption1_smooth_boundary} $\Omega$ is a bounded $C^{r+1,\kappa}$ domain, in the sense that $\paOm$ is $C^{r+1,\kappa}$.
		\item \ \ \label{eqn:assumption2_strict_elliptic} The \textit{diffusion matrix} $\Sigma$ is \textit{uniformly elliptic}: there exists $\lambda>0$ such that
		\[\xi^T\Sigma(x)\xi \ge \lambda \|\xi\|_2^2, \ \text{ for all } x\in\Omega \,\text{ and } \xi\in\R^d.\]
		\item \ \ \label{eqn:assumption3_regular_data} The boundary term $g\in C^{r+1,\kappa}(\bar{\Omega})$, and $c(x)\le 0$ for all $x\in\Omega$.
		\item \ \ \label{eqn:assumption4_regular_coeffs} The coefficients $c, (b_j)_{1\le j \le d}, (\Sigma_{ij})_{1\le i,j\le d}$ of $\mathscr{L}$ together with the forcing $f$ all belong to $C^{r-1,\kappa}(\bar{\Omega})$.
	\end{enumerate}
	
	We refer the reader to standard references such as \citep{gilbarg1977elliptic} for the definitions of $C^{r+1,\kappa}$ domains and $C^{r-1,\kappa}(\bar{\Omega})$ functions. Assumptions \eqref{eqn:assumption1_smooth_boundary} --- \eqref{eqn:assumption4_regular_coeffs} ensure the well-posedness of \eqref{eqn:bvp}.
	
	\begin{theorem}[{\citep[Theorems 6.14 and 6.19]{gilbarg1977elliptic}}]
		\label{thm:dirichlet_solution}
		If Assumptions \eqref{eqn:assumption1_smooth_boundary} --- \eqref{eqn:assumption4_regular_coeffs} hold, then the boundary value problem \eqref{eqn:bvp} admits a unique solution $u^\star$ lying in $C^{r+1,\kappa}(\bar{\Omega})$.
	\end{theorem}
	
	Since $\Omega$ is bounded, we will assume for convenience in all of the following that $\Omega$ is contained in the closed unit ball $\mathbb B^d:=\{x\in\R^d:\|x\|_2\le1\}\subset[-1,1]^d$, which can be done without loss of generality by scaling and translation of the coordinate system.
	
	For illustration's sake, let's consider a representative example of BVP of interest: let $\Omega$ be the open unit ball, and for any $x\in \R^d$, let $(X_t^x)_{t\ge 0}$ be a continuous-time stochastic process on $\R^d$ solving the stochastic differential equation (SDE) 
	\begin{equation}
		\label{eqn:OU_example}
		dX_t^x= -\theta X_t^x dt + \sqrt{2\varepsilon}dW_t, \quad X_0^x= x,
	\end{equation}
	where $(W_t)_{t\ge0}$ is a standard Brownian motion, and $\theta,\varepsilon>0$ are physical parameters. A quantity of interest across many fields such as neurobiology, molecular dynamics, or computational chemistry, is that of computing the \textit{mean exit time} (MET) of $(X_t^x)_{t\ge 0}$ from $\Omega$. That is, computing for any $x\in\Omega$ the quantity $\tau(x):=\Exp\left[\inf \{t\ge 0: X_t^x\in\partial\Omega\}\right]$. One can show by probabilistic arguments that $\tau$ solves a BVP of the form \eqref{eqn:bvp}, with all Assumptions \eqref{eqn:assumption1_smooth_boundary} --- \eqref{eqn:assumption4_regular_coeffs} satisfied, such that $\tau$ is well-defined and $C^\infty(\bar{\Omega})$-smooth \citep[Result 7.1]{pavliotis2014stochastic}.  
	
	In the $\varepsilon\ll1$ regime, which is of important practical interest, the theory of singular perturbations gives a limit-behaviour for the MET. Namely, that $\tau(x) \sim_{\varepsilon\to 0^+} C\varepsilon^{d/2} \exp \left(\frac{\theta}{2\varepsilon}\right)$ for all $x$ at least $\varepsilon|\log\varepsilon|$ away from $\paOm$ \citep{verhulst2005methods, freidlin2012random}. This sharp growth of the solution over such a thin boundary layer suggests that, despite the domain and solution being infinitely smooth, PINNs are likely to struggle in satisfying the Dirichlet boundary conditions---and thus in solving the PDE---if they are not exactly enforced in their architecture.
	
	\begin{figure}[h!]
		\centering
		\includegraphics[width=0.45\textwidth]{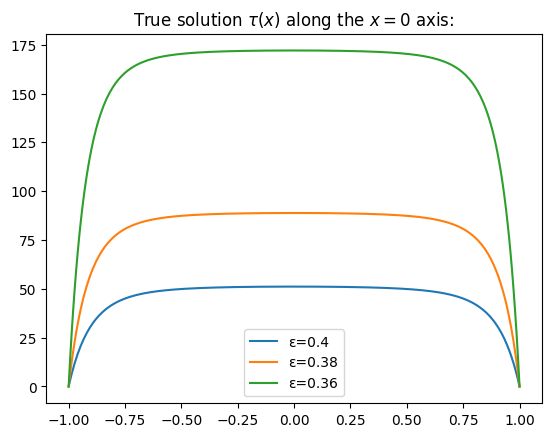}
		\caption{Projection on the $\{(x,y):x=0\}$ axis of the MET for the SDE \eqref{eqn:OU_example} in the 2D unit ball, with $\theta=1$ and varying $\varepsilon>0$. The closed-form expression for this MET has been derived in \citep{kersting2023mean}.}% As we can see, there is sharp growth for solutions away from the boundary, and the solution profile rises exponentially fast in $1/\varepsilon$.}
		\label{fig:boundary_layer}
	\end{figure}
	
	\subsection{Physics-informed neural networks and boundary conditions}
	
	An approach to solve boundary value problems (BVPs) such as \eqref{eqn:bvp} consists in using Physics-Informed Neural Networks (PINNs). Popularized in \citep{raissi2019physics}, the idea is to use Neural Networks as solution ansatz to the BVP, and minimize the PDE residual by using gradient-based algorithms.	More specifically, for a PDE such as \eqref{eqn:bvp}, the residual is given by
	\begin{equation}
		\label{eqn:standard_residual}
		\mathcal{R}(u) = \frac{1}{|\Omega|}\int_\Omega (\mathscr{L}u (x) - f(x))^2\ dx + \frac{1}{\sigma(\paOm)}\int_{\paOm}(u(x) - g(x))^2\ dx,
	\end{equation}
	where $|\cdot|$ and $\sigma(\cdot)$ respectively denote the Lebesgue measure in $\R^d$ and the $d-1$ dimensional Hausdorff measure. The PINN approach in this setting consists in respectively sampling $n$ and $m$ independent realizations of the uniform distribution respectively on $\Omega$ and $\paOm$, that is, $x_1,\ldots,x_n\in\Omega$, and $ x_{n+1},\ldots,x_{n+m}\in\paOm$ and minimizing the empirical residual:
	\begin{equation}
		\label{eqn:empirical_residual}
		\hat \theta\in\argmin_{\theta\in\Theta} \frac1n\sum_{i=1}^n (\mathscr{L}u_\theta(x_i) - f(x_i))^2 + \frac1m \sum_{i=1}^m (u_\theta(x_{n+i}) - g(x_{n+i}))^2,
	\end{equation}
	where $u_\theta$ is a neural network parametrized by $\theta\in\Theta$.  
	
	\subsubsection{Inadequacy of standard PINN loss for $H^2(\Omega)$ error bounds}
	
	Classical learning theory gives us tools to control the excess residual $\mathcal{R}(\hat{u}) - \mathcal{R}(u^\star)$ (which, in this context, is referred to as \textit{excess risk}), where $\hat u \equiv u_{\hat{\theta}}$ is a minimizer of the empirical residual \eqref{eqn:empirical_residual} over a set $\mathcal U\subset H^2(\Omega)$, by decomposing the residual into an approximation error term and a statistical error term \citep{koltchinskii2000rademacher}. By itself, however, residual control does not imply that $\hat u$ is close to the target $u^\star$ in any function norm. To obtain such a conclusion, one needs an estimate relating $\mathcal{R}(u) - \mathcal{R}(u^\star)$ to a norm of $u-u^\star$ for all $u\in\mathcal U$.  
	
	As the mean escape time discussion above illustrates, not every norm is equally informative regarding the quality of the approximate solution. In Figure~\ref{fig:boundary_layer}, we see that the solution $u^\star$ is nearly constant throughout most of $\Omega$ and varies sharply only in a thin boundary layer near $\partial\Omega$. An approximation $\hat{u}$ may therefore agree well with $u^\star$ in the bulk while still fitting poorly in the narrow region where the sharp variation of the solution is concentrated. This therefore shows that an $L^2(\Omega)$ error bound is too weak for our purposes, since such a norm may remain small even when the boundary layer is poorly captured. The issue is therefore not merely to obtain an error bound, but to obtain one in a norm strong enough to detect sharp variations of the solution, even over small regions. The strongest norm that is a viable candidate in our setting is the $H^2(\Omega)$ norm, and it is thus the one we will focus on in what follows.
	
	Our target norm now clear, our next question is thus whether minimizing the residual actually forces the PINN approximation $\hat{u}$ toward $u^\star$ in this stronger norm. That is, whether we can guarantee an estimate of the form $\|u - u^\star\|_{H^2(\Omega)}^2 \lesssim \mathcal{R}(u) - \mathcal{R}(u^\star)$ for all $u\in\mathcal U$. Likewise, because the natural way to control the approximation error $\inf_{u \in \mathcal U}\mathcal{R}(u) - \mathcal{R}(u^\star)$ is through the expressivity of our neural network space, the reverse estimate $\mathcal{R}(u) - \mathcal{R}(u^\star) \lesssim \|u - u^\star\|_{H^2(\Omega)}^2$ is equally important. We are therefore naturally led to ask whether the residual $\mathcal{R}$ satisfies a strong-convexity estimate of the form
	\begin{equation}
		\label{eqn:strong_convexity_general}
		\|{u} - u^\star\|_{H^2(\Omega)}^2 \lesssim \mathcal{R}({u}) - \mathcal{R}(u^\star) \lesssim \|{u} - u^\star\|_{H^2(\Omega)}^2,
	\end{equation}
	uniformly over $\mathcal U$. Such estimates are closely tied to elliptic a priori bounds and are non-trivial in general. In our setting, the upper bound in \eqref{eqn:strong_convexity_general} is not hard to obtain: indeed, since $u^\star$ solves the BVP \eqref{eqn:bvp}, we have that $\mathcal{R}(u^\star)=0$, and thus get for any $u\in H^2(\Omega)$:
	\begin{align*}
		\mathcal{R}(u) - \mathcal{R}(u^\star) &=  \frac{1}{|\Omega|}\int_\Omega (\mathscr{L}u (x) - f(x))^2\ dx + \frac{1}{\sigma(\paOm)}\int_{\paOm} (u(x) - g(x))^2\  dx\\
		&= \frac{1}{|\Omega|}\|\mathscr{L}(u-u^\star)\|_{L^2(\Omega)}^2 + \frac{1}{\sigma(\paOm)} \|u-u^\star\|_{L^2(\paOm)}^2\\
		&\lesssim \|u-u^\star\|_{H^2(\Omega)}^2 + \|u-u^\star\|_{L^2(\paOm)}^2 \lesssim \|u-u^\star\|_{H^2(\Omega)}^2,
	\end{align*}
	
	where we used the regularity of the coefficients of $\mathscr{L}$ and the Sobolev trace theorem to get the last two inequalities. Obtaining the matching lower bound in \eqref{eqn:strong_convexity_general}, however, is a much more delicate matter, which is addressed by the following classical theorem:
	
	\begin{theorem}[Adapted from Theorem 15.2 in \citep{agmon1959estimates}]
	\label{thm:nirenberg_estimates}
	Let $\mathscr{L}$ be a uniformly elliptic operator of order $2$ whose coefficients are in $C(\bar\Omega)$, and suppose that $\paOm$ is of class $C^2$. If $u$ is a solution of the boundary value problem
	\begin{equation}
		\label{eqn:abstract_bvp}
		\begin{cases}
			\mathscr{L} u &= f \quad \text{ in } \Omega\\
			u &= g \quad \text{ on } \paOm
		\end{cases}
	\end{equation}
	such that $\|u\|_{H^2(\Omega)}$, $\|f\|_{L^2(\Omega)}$, and $\|g\|_{H^{3/2}(\paOm)}$ are all finite, then we have the estimate
	\[\|u\|_{H^2(\Omega)} \lesssim \|f\|_{L^2(\Omega)} + \|g\|_{H^{3/2}(\paOm)} + \|u\|_{L^2(\Omega)}. \]
	Furthermore, if the BVP \eqref{eqn:abstract_bvp} has a unique solution in $H^2(\Omega)$, the additional $\|u\|_{L^2(\Omega)}$ term can be omitted, leading to the estimate:
	\begin{equation*}
		\|u\|_{H^2(\Omega)} \lesssim \|f\|_{L^2(\Omega)} + \|g\|_{H^{3/2}(\paOm)}.
	\end{equation*}
	\end{theorem}
	 
	From Theorem \ref{thm:nirenberg_estimates}, if our Assumptions \eqref{eqn:assumption1_smooth_boundary}---\eqref{eqn:assumption4_regular_coeffs} hold, we get the estimate
	\begin{equation}
		\label{eqn:strong_convexity_lower_bound}
		\| u - u^\star\|_{H^2(\Omega)} \lesssim \|\mathscr{L}( u - u^\star)\|_{L^2(\Omega)} + \| u - u^\star\|_{H^{3/2}(\paOm)}, \ \ \text{ for all } u\in H^2(\Omega).
	\end{equation}
	Estimate \eqref{eqn:strong_convexity_lower_bound} shows that an $H^2(\Omega)$ error bound cannot in general be recovered from a residual that only penalizes the boundary trace in $L^2(\partial\Omega)$. One could in principle replace this term by an $H^s(\partial\Omega)$ penalty with $s \ge 3/2$, but this is difficult to implement in practice, as the accurate computation of higher-order (including fractional) Sobolev norms on arbitrary domains is challenging. Moreover, incorporating such norms directly into neural network training can significantly affect the efficiency of algorithms based on automatic differentiation. We however note that a methodology to compute fractional Sobolev norms on manifolds based on eigenvalues of the Laplace-Beltrami operator has recently been proposed in \citep{zhou2026simultaneous}, which the authors have successfully applied to train physics-informed convolutional neural networks. We likewise note that it is still possible to get a priori error estimates using an $L^2(\paOm)$ penalty, but this typically results in estimates in weaker norms, which may fail to capture the full complexity of the PDE solution \citep{jiao2022rate, bonito2025convergence}.
	
	\begin{example}
		\label{prop:soft_counterexample}
		On the unit disk $\Omega=\{x\in\R^2:\|x\|_2<1\}$, consider the equation $\Delta u=-1$ with zero boundary condition, whose solution is $u^\star:x\mapsto (1-\|x\|_2^2)/4$. In polar coordinates $(r,\vartheta)$, define for $k\ge 3$ $u_k(r,\vartheta):=u^\star(r,\vartheta)+k^{-1}r^k\sin(k\vartheta)$, which also satisfies $\Delta u_k=-1$. On the boundary, $u_k-u^\star=k^{-1}\sin(k\vartheta)$, and the standard PINN residual \eqref{eqn:standard_residual} satisfies $\mathcal R(u_k)=1/(2k^2)$, which goes to $0$ as $k\to\infty$. However, a direct computation gives $\|u_k-u^\star\|_{H^2(\Omega)}^2\ge\pi(k-1)/2$ which diverges as $k\to\infty$. Thus, neither boundary penalty or PDE residual control can prevent the $H^2(\Omega)$ error from diverging.
	\end{example}
	
	These difficulties suggest opting for a different strategy: instead of penalizing boundary error, we enforce satisfaction of the boundary conditions directly in the ansatz, for instance by multiplying the network output by a smooth approximation of the distance-to-boundary function. This conforming approach goes back to \citep{lagaris1998artificial, lagaris2000neural} and was further developed in \citep{berg2018unified, sukumar2022exact}. This is the approach we will consider in this work.
	
	\subsubsection{``Hard-enforcement" of boundary conditions in PINN architectures}
		
	In the physics-informed machine learning literature, it is common to take the PDE-solving ansatz as a feedforward neural network (FNN) with a sufficiently smooth activation function \citep{yu2018deep,raissi2019physics}. This leads to hypothesis spaces of the following form:
	
	\begin{definition}[Fully connected neural networks]
		\label{def:fcnn_hypothesis_space}
		Let $\sigma:\R\to\R$ be an activation function, $L\in\N_{>0}$, $\mathbf m=(m_1,\ldots,m_L)\in\N_{>0}^L$, $m_0=d$, $B>0$, and $M\ge0$. We denote by $\Sigma^\sigma_{\mathbf m_{1:L}}(B)$ the class of FNNs defined on $\R^d$ by:
		\begin{equation}
			\label{eqn:fcnn_general}
			\begin{split}
				\Sigma^\sigma_{\mathbf m_{1:L}}(B):=\bigl\{&x\mapsto c\cdot h_L(x)+c_0:
				h_0(x)=x,\ h_i(x)=\sigma(A_i h_{i-1}(x)+b_i), A_i\in\R^{m_i\times m_{i-1}},\\
				& b_i\in\R^{m_i}, c\in\R^{m_L},\ c_0\in\R, \|A_i\|_{\infty,\infty}\le B, \|b_i\|_\infty\le B, i=1,\ldots,L\bigr\},
			\end{split}
		\end{equation}
		where $\sigma$ is applied element-wise, and $\|b\|_\infty$, $\|A\|_{\infty,\infty}$ both denote the entry-wise maximum absolute value of respectively $b\in\R^{m_i}$ and $A\in\R^{m_i\times m_{i-1}}$. If we further require $\max\{\|c\|_\infty,|c_0|\}\le M$, we will denote the class by $\Sigma^\sigma_{\mathbf m_{1:L},M}(B)$, following notation from \citep{he2023expressivity,siegel2024sharp}. Lastly, we may set $B=\infty$ or $M=\infty$ to signify that the corresponding parameters are unrestricted.
	\end{definition}
	
	Although the hypothesis space $\Sigma^\sigma_{\mathbf{m}_{1:L},M}(B)$ generally enjoys a lot of desirable properties regarding expressiveness and complexity, the functions in this space do not, in general, satisfy the Dirichlet boundary conditions \eqref{eqn:bvp}. As we just discussed, this implies that a control of the a priori $H^2(\Omega)$ error with this hypothesis space is not possible. Existing analyses have bypassed this issue either through the unrealistic assumption of exact satisfaction of boundary conditions \citep{lu2021machine}, or considering PDEs on boundary-less submanifolds of $\R^d$, such as the sphere \citep{lei2025solving}, which is too restrictive. To make up for these limitations, we propose to analyze the hypothesis space of FNNs which exactly satisfy the boundary conditions after having their outputs multiplied by \textit{smooth distance approximations}:
	
	\begin{definition}[Smooth distance approximation]
		\label{def:smooth_distance}
		Given $\rho : \bar{\Omega}\to\R$ and $r\in\N^+$, we say that $\rho$ is an $r$-smooth approximate distance function if (i) $\rho\in C^r(\bar{\Omega})$, (ii) $\rho(x)=0$ for all $x\in\paOm$, and (iii) $\rho(x)>0$ for all $x\in\Omega$. %We will also call such a function $\rho$ an $r$-smooth distance approximation.
	\end{definition}
	
	Given such a distance approximation, we can then define the hypothesis space of \textit{boundary-enforced} neural networks:
	
	\begin{definition}[Boundary-enforced FNNs]
		\label{def:boundary_pinns_hyp_space}
		Let $r\ge3$, $\rho:\bar{\Omega}\to\R$ be a $r$-smooth distance approximation, $\sigma : \R\to\R$ be an activation function, $L\in\N^+$, $\mathbf{m}:=\{m_i\}_{1\le i \le L}\subset \N^+ $, $g:\bar{\Omega}\to\R$ be the boundary data \eqref{eqn:bvp}, and $Q,B,M>0$. We denote by $\mathcal F(\rho, L, \mathbf{m},\sigma, Q, B, M)$ the hypothesis space of functions
		\begin{equation}
			\label{eqn:boundary_fcnn}
			\F(\rho, L, \mathbf{m},\sigma, Q, B, M) = \left\{g + \rho h: h\in  \Sigma^\sigma_{\mathbf{m}_{1:L}, M}(B)\text{ and } \max_{\|\alpha\|_1 \leq 2} \|\partial^\alpha h\|_{L^\infty(\Omega)} \le Q\right\}.
		\end{equation}
		As in Definition \ref{def:fcnn_hypothesis_space}, we may let any of $Q,B$ and $M$ be equal to $\infty$ to signify that the corresponding parameters need not be bounded.
	\end{definition}
	
	As we will see later, we in fact do not need to impose restrictions on the FNN weights' norms to carry our main error analysis. Our approximation theorems will however still be stated with sufficient parameter weight bounds for the sake of comparison with similar results in the literature. One element we will however need is a uniform bound $Q$ on the $W^{2,\infty}(\Omega)$ norm of the underlying FNNs $h$ in \eqref{eqn:boundary_fcnn}, whose practical implementation could be challenging. Fortunately, for the activation functions considered in this work, there is a straightforward way to guarantee such control, as we will discuss in Remark~\ref{rem:practical_q_bound} further down.
	
	Assuming the activation function $\sigma$ is sufficiently regular, it follows immediately that all ``boundary-enforced" PINNs satisfy the Dirichlet boundary condition. Consequently, the strong convexity estimate \eqref{eqn:strong_convexity_general} holds, which is necessary for our error analysis to carry through. However, multiplying by $\rho$ introduces additional challenges for our error analysis: to control the approximation error of the hypothesis space \eqref{eqn:boundary_fcnn}, we now must control the Sobolev norm of $h-(u^\star-g)/\rho$, where $h$ lies in the FNN space $\Sigma^\sigma_{\mathbf{m}_{1:L},M}(B)$\begingroup
	\renewcommand{\thefootnote}{\fnsymbol{footnote}}%
	\footnotemark[4]%
	\footnotetext[4]{Indeed, this is because the Sobolev-Leibniz rule gives $\|g+\rho h-u^\star\|_{W^{2,\infty}(\Omega)}\le4\|\rho\|_{W^{2,\infty}(\Omega)}\|h-(u^\star-g)/\rho\|_{W^{2,\infty}(\Omega)}$.}\endgroup. Since Definition~\ref{def:smooth_distance} alone does not govern the interaction between $\rho$ and $u^\star-g$, it is unclear whether their quotient is well-defined on $\paOm$, let alone whether it possesses any Sobolev regularity. As it turns out, a simple sufficient condition on $\rho$ to ensure smoothness of this quotient exists, as we now introduce.
	
	\subsubsection{Ensuring regularity of the approximation target}
	
	We now introduce the better behaved class of smooth \textit{normalized} distance approximations \citep{shapiro2007semi, sukumar2022exact}. For a function $u$ defined in a neighborhood of $\paOm$, $q\in\paOm$, $\bnu(q)$ the associated inward unit normal, and any $k\in \N$, we write $\partial_{\bnu}^k u(q) := \left.\frac{d^k}{dt^k} u\bigl(q+t\,\bnu(q)\bigr)\right|_{t=0}$, whenever the derivative exists.
	
	\begin{definition}[Smooth normalized distance approximation]
		\label{def:order_m_normalization}
		Let $r\in\N^+$ and $m\in\{1,\dots,r\}$. A function $\rho:\bar\Omega\to\R$ is called an $r$-smooth approximate distance function normalized to order $m$ if it is an $r$-smooth distance approximation in the sense of Definition~\ref{def:smooth_distance} which furthermore satisfies:
		\[
		\partial_{\bnu} \rho(q)>0,\text{ and }\ \partial_{\bnu}^k \rho(q)=0\ \text{ for all } k=2,\dots,m \text{ and } q\in\paOm.
		\]
	\end{definition}
	
	The following proposition guarantees that if $\rho$ is a smooth \textit{normalized} distance approximation, then the quotient $(u^\star-g)/\rho$ extends to a $C^r$-smooth function up to the boundary\footnote{This assumption is in fact \textit{necessary} for the quotient to extend continuously to $\bar\Omega$ whenever $\partial_\bnu[u^\star-g](x_0)\ne0$ at a point $x_0\in\paOm$. Indeed, if $\partial_\bnu\rho(x_0)=0$, then $[u^\star-g](x_0+t\bnu)=\partial_\bnu[u^\star-g](x_0)t+o(t)$, while $\rho(x_0+t\bnu)=o(t)$, and the quotient is unbounded near $x_0$.}:
	
	\begin{proposition}
		\label{prop:quotient_regularity}
		If Assumptions \eqref{eqn:assumption1_smooth_boundary}---\eqref{eqn:assumption4_regular_coeffs} hold such that $u^\star$ is well-defined as the unique solution of \eqref{eqn:bvp}, and if $\rho:\bar{\Omega}\to\mathbb{R}$ is an $(r+1)$-smooth distance approximation normalized to the first order as per Definition \ref{def:order_m_normalization}, then $(u^\star - g)/\rho \in C^r(\bar{\Omega})$.
	\end{proposition}
	
	While Proposition~\ref{prop:quotient_regularity}, whose proof is given in Appendix~\ref{sec:quotient_regularity}, gives us a sufficient criterion for appropriate choices of boundary-enforcing ansatzes, a computationally tractable construction for such nicely behaved functions is also needed. Fortunately, for an extensive range of geometries $\Omega$, efficient algorithms are available to construct such distance approximations \citep{shapiro2007semi,sukumar2022exact}.
	
	\begin{remark}
		\label{rem:shapiro-practical_extension}
		When $\Omega=\{x:\phi(x)>0\}$ with $\phi\in C^{r+1}(\bar\Omega)$ and $\nabla\phi\ne0$ on $\paOm$, the choice $\rho=\phi$ already satisfies the conditions of Proposition~\ref{prop:quotient_regularity}. If $\phi\in C^{r+2}(\bar\Omega)$, the first-order normalized construction $\rho(x)=\phi(x)/\sqrt{\phi(x)^2+\|\nabla\phi(x)\|_2^2}$ is another admissible choice. When such a defining function is not directly available, one may build an implicit representation using level-set, R-function, or mean-value-potential techniques \citep{osher2004level,shapiro2007semi,sukumar2022exact}.
	\end{remark}
	
	\subsubsection{Boundary-adapted PINNs}
	
	In light of Proposition~\ref{prop:quotient_regularity}, we will refer to functions of the form \eqref{eqn:boundary_fcnn}, where $\rho$ is an $(r+1)$-smooth distance approximation normalized to the first order, as ``\textit{boundary-adapted} FNNs":
	
	\begin{definition}[Boundary-adapted FNNs]
		\label{def:boundary_adapted_fnns}
		Let $r\ge3$, $\sigma : \R\to\R$ be an activation function, $L\in\N^+$, $\mathbf{m}:=\{m_i\}_{1\le i \le L}\subset \N^+ $, and $Q,B,M>0$. If $\rho:\bar{\Omega}\to\R$ is an $(r+1)$-smooth distance approximation normalized to the first order in the sense of Definition~\ref{def:order_m_normalization}, then we refer to the hypothesis space $\F(\rho, L, \mathbf{m},\sigma, Q, B, M)$ as defined in Definition~\ref{def:boundary_pinns_hyp_space} as a space of boundary-adapted FNNs (or boundary-adapted PINNs).
	\end{definition}
	
	The advantage of boundary-adapted PINNs lies in the fact that they satisfy the Dirichlet boundary conditions exactly while preserving the smoothness of the approximation target $(u^\star-g)/\rho$. In particular, the $C^r(\bar{\Omega})$ regularity of $(u^\star-g)/\rho$ implies that it can be extended to a $W^{r,\infty}(\R^d)$ function, which will be helpful for theoretical analysis. The below result can be found, e.g. in \citep[Theorem 7.25]{gilbarg1977elliptic}:
	
	\begin{lemma}[Extension lemma]
		\label{lemma:extension_lemma}
		Under the assumptions of Proposition~\ref{prop:quotient_regularity}, there exist a compactly supported $v\in W^{r,\infty}(\R^d)$ and a constant $C_{\mathrm{ext}}>0$ depending only on $r,d$ and $\Omega$, such that $v=(u^\star-g)/\rho$ on $\Omega$, and
		\begin{equation}
			\label{eqn:met_extended_estimate}
			\|v\|_{W^{r,\infty}(\R^d)}\le C_{\mathrm{ext}}\|(u^\star-g)/\rho\|_{W^{r,\infty}(\Omega)}.
		\end{equation}
	\end{lemma}
	
	In this work, we will consider FNNs using either the Rectified Quadratic Unit (ReQU), or the hyperbolic tangent (tanh) as activation functions. These choices ensure that the FNN hypothesis spaces defined in \eqref{eqn:boundary_fcnn} are contained in the Sobolev space $W^{2,\infty}(\Omega)\subset H^2(\Omega)$, and represent two of the most widely used activation functions in scientific machine learning. Concretely, this means we consider either $\sigma(x)=\max\{0,x\}^2$ or $\sigma(x)=\tanh(x)$, and will provide an error analysis for each choice.
	
	\begin{remark}[Practical enforcement of Sobolev norm bound]
		\label{rem:practical_q_bound}
		For the ReQU and tanh activations, the bound $Q$ in Definitions~\ref{def:boundary_pinns_hyp_space} and \ref{def:boundary_adapted_fnns} can be enforced indirectly by penalizing the trainable parameters. Indeed, if $J_\lambda(\theta)=J(\theta)+\lambda\|\theta\|_2^2$ with $\lambda>0$ and $J(\theta)\ge 0$, then every fixed sublevel set of $J_\lambda$ yields a uniform bound on $\|\theta\|_2$. For a fixed architecture and bounded domain $\Omega$, this immediately gives a uniform $W^{2,\infty}(\Omega)$ bound on the network, and therefore also on the boundary-adapted ansatz \eqref{eqn:boundary_fcnn}, since $\rho$ is fixed and smooth, as can be checked by directly examining the expression of the neural networks' derivatives. One may also refer to \citep{doumeche2025pinn} for further discussion on the theoretical benefits of such regularization in the context of PINNs.
	\end{remark}
	
	As our boundary-adapted PINNs \eqref{eqn:boundary_fcnn} satisfy the boundary conditions exactly, the corresponding error term vanishes. From here onwards, we thus solely work with the interior residual, defined for all $u\in H^2(\Omega)$ as
	
	\begin{equation}
		\label{eqn:residual_true}
		\mathcal{R}(u) = \frac{1}{|\Omega|}\int_\Omega (\mathscr{L}u(x) - f(x))^2 dx.
	\end{equation}
	
	And likewise, its empirical counterpart is defined as
	
	\begin{equation}
		\label{eqn:residual_empirical}
		\mathcal{R}_n(u) = \frac1n \sum_{i=1}^n (\mathscr{L}u(x_i) - f(x_i))^2,
	\end{equation}
	
	where the $x_1,\ldots,x_n$ are independent samples drawn uniformly on $\Omega$. For $u$ in the space $\F$ of boundary-adapted PINNs, the error $u-u^\star$ belongs to $H^2(\Omega)\cap H^1_0(\Omega)$, so the strong-convexity estimate \eqref{eqn:strong_convexity_general} holds for this interior residual.
	
	\section{Main results}
	
	\subsection{Approximation error bound}
	
	We first give a lemma which bounds the Sobolev distance between $\rho h$ and $u^\star-g$ in terms of the distance between $h$ and $(u^\star-g)/\rho$.
	
	\begin{lemma}
		\label{lemma:product_estimate}
		Let $r\in\N$, $h,v,\rho\in W^{r,\infty}(\Omega)$, and assume that $v/\rho\in W^{r,\infty}(\Omega)$. Then
		\begin{equation}
			\label{eqn:product_estimate}
			\|\rho h-v\|_{W^{r,\infty}(\Omega)}\le C_{\mathrm{mult}}\|h-v/\rho\|_{W^{r,\infty}(\Omega)},
		\end{equation}
		where $C_{\mathrm{mult}}=2^r\|\rho\|_{W^{r,\infty}(\Omega)}$.
	\end{lemma}
	
	\begin{proof}
		Let $z=h-v/\rho$. Leibniz's rule gives $\|\partial^\alpha(\rho z)\|_\infty\le\sum_{\beta\le\alpha}\binom\alpha\beta\|\partial^\beta\rho\|_\infty\|\partial^{\alpha-\beta}z\|_\infty$. Since $\sum_{\beta\le\alpha}\binom\alpha\beta=2^{\|\alpha\|_1}\le2^r$ for $\|\alpha\|_1\le r$, taking the maximum proves the estimate.
	\end{proof}
	
	We now introduce an upper bound for the norm of the operator $\mathscr{L}:H^2(\Omega)\to L^2(\Omega)$ in terms of its coefficients, which will be necessary to control the approximation error:
	
	\begin{lemma}
		\label{lemma:approximation_error_estimate}
		For any $u\in H^2(\Omega)$, we have the estimate $\|\mathscr{L}u\|_{L^2(\Omega)}^2 \le C_{\mathscr{L}}\|u\|_{H^2(\Omega)}^2$, where 
		\[
		C_{\mathscr L}=3\max\left\{
		\|c\|_{L^\infty(\Omega)}^2,
		d\max_{1\le i\le d}\|b_i\|_{L^\infty(\Omega)}^2,
		2d^2\max_{1\le i,j\le d}\|\Sigma_{ij}\|_{L^\infty(\Omega)}^2
		\right\}.
		\]
	\end{lemma}
	
	\begin{proof}
		By Cauchy-Schwarz inequality, followed by Hölder inequality, we get:
		\begin{align*}
			\|\mathscr Lu\|_{L^2(\Omega)}^2
			\le3\|cu\|_{L^2(\Omega)}^2
			+3d\sum_{i=1}^d\|b_i\partial_i u\|_{L^2(\Omega)}^2
			+3d^2\sum_{i,j=1}^d\|\Sigma_{ij}\partial_{ij}u\|_{L^2(\Omega)}^2
			\le C_{\mathscr L}\|u\|_{H^2(\Omega)}^2.
		\end{align*}
	\end{proof}
	
	We can now state the approximation error bounds of smooth functions using respectively the ReQU and tanh activations. We begin by the former, for which we give a proof in Appendix~\ref{app:approximation_error}:
	
	\begin{theorem}
		\label{thm:approx_requ}
		Let $r\ge 3,N\ge 2$ be two integers, and let $u\in W^{r,\infty}(\R^d)$. Take the hypothesis space of ReQU FNNs as $ \F_{NN}:=\Sigma^\sigma_{\mathbf{m}_{1:L},M}(B)$
		with $\sigma:x\mapsto \max\{0,x\}^2$, $L = 1$, $\|\mathbf{m}\|_\infty = N^d$, $B=B_0$, and $M = M_N(u)$, where $M_N(u) \le
			C_{d,r}\,\bigl(1+M_\ast(N)\bigr)\,\|u\|_{W^{r,\infty}(\R^d)}$ for
			\[
			M_\ast(N)=\begin{cases}
				N^{(d+5)/2-r},&r<(d+5)/2,\\
				1,&r\ge(d+5)/2,
			\end{cases}
			\]
		and where $B_0>0$ and $C_{d,r}>0$ depend only on $d$ and $r$. Then there exists a constant $M_{\ReQU}>0$, also dependent on $r$ and $d$ only, such that
		\begin{equation}
			\label{eqn:approx_requ_1}
			\inf_{f_{NN}\in\F_{NN}}
			\|f_{NN} - u\|_{W^{2,\infty}(\mathbb B^d)}
			\le
			\frac{M_{\ReQU}}{N^{\gamma(r)}}
			\,\|u\|_{W^{r,\infty}(\R^d)},
		\end{equation}
		where $\gamma(r):=\min\{r-2,(d+1)/2\}$.
	\end{theorem}
	
	Theorem~\ref{thm:approx_requ} extends to the $W^{2,\infty}(\Omega)$ norm the $L^\infty(\Omega)$ shallow ReLU$^k$ approximation results of \citep{yang2025optimal}. In particular, our bound exhibits the same \textit{saturation phenomenon}, where the exponent in \eqref{eqn:approx_requ_1} does not improve once $r>(d+5)/2$. In fact, virtually all available approximation bounds for shallow ReLU$^k$ networks exhibit this phenomenon, and to the best of our knowledge, the question of whether this saturation is intrinsic to shallow ReLU$^k$ networks or is instead an artifact due to the limitations of variation space-based arguments remains open \citep{mao2026approximation}. This suggests that shallow ReQU networks are not best suited to solve PDEs with highly smooth solutions in low dimension. Although the full saturation-free approximation result could be recovered by increasing depth \citep{guhring2021approximation, he2023expressivity}, this is not advisable as it would come at the cost of significantly more unstable training \citep{wang2024piratenets}.
	
	 We also have an analogous approximation error bound for FNNs with tanh activation function, due to \citep{deryck2021approximation}.
	
	\begin{theorem}[Adapted from \citep{deryck2021approximation}, Theorem 5.1]
		\label{thm:approx_tanh}
		Let $r\ge3$, $N>5d^2$ be integers, and let $u\in W^{r,\infty}(\R^d)$. Take the hypothesis space of tanh FNNs as $\F_{NN}:=\Sigma^\sigma_{\mathbf m_{1:L},M}(B)$ with $\sigma:x\mapsto\tanh(x)$, $L = 2$, $m_1=\lceil A_{d,r}N\rceil$, $m_2=\lceil A_{d,r}N^d\rceil$, and $B = M = C_{d,r,u} N^{d(d+r^2+4)/2} $, where $A_{d,r}, C_{d,r,u} >0$ are independent of $N$. Then there exist $M_{\tanh}>0$, dependent on $r$ and $d$ only, and $\beta\ge \exp(1)$ dependent on $d,r$ and $u$ only such that
		\begin{equation}
			\label{eqn:approx_tanh_1}
			\inf_{f_{NN}\in\mathcal F_{NN}} \|f_{NN} - u\|_{W^{2,\infty}([-1,1]^d)} \le  \log^2\left(\beta N^{r+d+2} \right) \frac{M_{\tanh}}{N^{(r-2)}}\|u\|_{W^{r,\infty}(\R^d)}.
		\end{equation}
	\end{theorem}
	
	From Theorems \ref{thm:approx_requ} and \ref{thm:approx_tanh}, together with the lemmas preceding them, we deduce as a corollary the following approximation error bound for our boundary-adapted Neural Networks hypothesis space.
	
	\begin{corollary}
		\label{cor:approx_error}
		Let $r\ge 3$ be an integer. Assume that Assumptions \eqref{eqn:assumption1_smooth_boundary}---\eqref{eqn:assumption4_regular_coeffs} hold, and that $\rho$ satisfies Definition~\ref{def:boundary_adapted_fnns}. Let $v\in W^{r,\infty}(\R^d)$ be the extension of $(u^\star-g)/\rho$ from Lemma~\ref{lemma:extension_lemma}. Finally, recall the hypothesis space of boundary-adapted FNNs $\mathcal F = \mathcal{F}(\rho, L, \mathbf{m}, \sigma, Q, B, M)$, and suppose that
		\begin{equation}
			\label{eqn:qbound_requ}
			Q\ge1+2\|v\|_{W^{r,\infty}(\R^d)}.
		\end{equation}
		There exists $N_0=N_0(d,r,v)$ such that the following holds for every integer $N\ge N_0$:
		\begin{itemize}
			\item If $\sigma$ is the ReQU activation, and the parameters $L, \mathbf{m}, M=M_N(v)$ and $B$ are chosen as in Theorem \ref{thm:approx_requ}, then the approximation error of the hypothesis space $\F$ satisfies the inequality:
			\begin{equation}
				\label{eqn:approx_requ_2}
				\inf_{h\in\F}\mathcal R(h) \le
				C_{\ReQU} \,\|(u^\star-g)/\rho\|_{W^{r,\infty}(\Omega)}^2\,\frac{1}{N^{2\gamma(r)}},
			\end{equation}
			where $\gamma(r)$ is defined as in Theorem~\ref{thm:approx_requ}, and $C_{\ReQU} :=  \binom{d+2}{2} C_{\mathscr{L}} C_{\mathrm{ext}}^2 C_{\mathrm{mult}}^2 M_{\ReQU}^2$.
			\item If $\sigma$ is the tanh activation, and the parameters $L, \mathbf{m}, M$ and $B$ are chosen as in Theorem \ref{thm:approx_tanh}, then the approximation error of the hypothesis space $\F$ satisfies the inequality:
			\begin{equation}
				\label{eqn:approx_tanh_2}
				\inf_{h\in\F}\mathcal R(h)\le C_{\tanh} \|(u^\star - g)/\rho\|^2_{W^{r,\infty}(\Omega)}\frac{\log^4\left(\beta N^{r+d+2}\right)}{N^{2(r- 2)}},
			\end{equation}
			where $\beta$ is defined as in Theorem \ref{thm:approx_tanh}, and $C_{\tanh} := \binom{d+2}{2} C_{\mathscr{L}}C_{\mathrm{ext}}^2C_{\mathrm{mult}}^2 M_{\tanh}^2$.
		\end{itemize}
	\end{corollary}
	
	\begin{proof}[Proof of Corollary \ref{cor:approx_error}]
		We only give the proof for the ReQU activation, as the one for tanh is identical. Let $v$ be the extension in Lemma~\ref{lemma:extension_lemma}, and choose the approximating network $f_{NN}$ from Theorem~\ref{thm:approx_requ}. Since $r\ge3$, its $W^{2,\infty}$ error tends to zero as $N$ grows. For sufficiently large $N$, that error is at most $\|v\|_{W^{r,\infty}(\R^d)}$, and therefore
		\[
		\|f_{NN}\|_{W^{2,\infty}(\Omega)}
		\le\|v\|_{W^{2,\infty}(\Omega)}+\|f_{NN}-v\|_{W^{2,\infty}(\Omega)}
		\le2\|v\|_{W^{r,\infty}(\R^d)}<Q.
		\]
		This implies that $h_N:=g+\rho f_{NN}$ belongs to $\F$. Applying Lemma~\ref{lemma:product_estimate} with derivative order two, and writing $C_{\mathrm{mult}}=4\|\rho\|_{W^{2,\infty}(\Omega)}$, we get
		\begin{align*}
			\mathcal R(h_N) =\frac1{|\Omega|}\|\mathscr L(h_N-u^\star)\|_{L^2(\Omega)}^2
			\le\frac{C_{\mathscr L}}{|\Omega|}\|h_N-u^\star\|_{H^2(\Omega)}^2 
			\le C_{\mathscr L}C_{\mathrm{mult}}^2\binom{d+2}{2}
			\|f_{NN}-v\|_{W^{2,\infty}(\Omega)}^2,
		\end{align*}
		where we used the fact that $\|z\|_{H^2(\Omega)}^2\le |\Omega|\binom{d+2}{2}\|z\|_{W^{2,\infty}(\Omega)}^2$ in the last inequality.  
		
		By Theorem~\ref{thm:approx_requ} and the extension estimate \eqref{eqn:met_extended_estimate},
		\begin{align*}
			\inf_{h\in\F}\mathcal R(h)\!
			\le\mathcal R(h_N) 
			\le \!C_{\mathscr L}C_{\mathrm{mult}}^2\binom{d+2}{2}M_{\ReQU}^2
			N^{-2\gamma(r)}\|v\|_{W^{r,\infty}(\R^d)}^2\le\! C_{\ReQU}N^{-2\gamma(r)}
			\|(u^\star-g)/\rho\|_{W^{r,\infty}(\Omega)}^2.
		\end{align*}
		This proves the desired inequality, and, since the case $\sigma\equiv\tanh$ can be proved the exact same way, completes the proof.
	\end{proof}
	
	\subsection{Statistical error bound}
	
	In this section, we control the error incurred by approximating the PINN risk \eqref{eqn:residual_true} with a random sum \eqref{eqn:residual_empirical}. We compute new bounds on VC dimensions of hypothesis spaces of higher order derivatives of FNNs, which together with tools from the empirical process theory literature allow us to bound this error due to randomness.	
	
	\textbf{VC dimension and pseudodimension bounds for higher derivatives of FNNs.}
	
	An important element we need to obtain non-asymptotic excess risk bounds by means of oracle inequalities is an estimation of the \textit{Vapnik-Chervonenkis dimension} (VC dimension) of the hypothesis space of interest. We start by recalling its definition, together with the closely related concept of \textit{pseudo-dimension} \citep{bartlett1999neural}. 
	\begin{definition}[VC-dimension and pseudo-dimension]
		Given a function class $\mathcal{G}: \mathcal{X} \rightarrow \mathbb{R}$ and a set $X = \{x_i\}_{i=1}^m$ of $m$ points in the input space $\mathcal{X}$, let $\sgn(\mathcal{G}) = \{\sgn(g): g \in \mathcal{G}\}$ be the set of binary functions $\mathcal{X} \to \{0,1\}$ induced by $\mathcal{G}$, where we denote $\sgn(t):=\mathbf{1}_{\{t>0\}}$. If $\sgn(\mathcal{G})$ can compute all dichotomies of $X$, that is, $|\{(\sgn(g(x)))_{x\in X} : g \in \mathcal{G}\}| = 2^m$, we say that $\mathcal{G}$ shatters $X$. The Vapnik-Chervonenkis dimension (or VC-dimension) of $\mathcal{G}$ is defined as the size of the largest shattered subset of $\mathcal{X}$, and denoted by $\text{VCDim}(\mathcal{G})$.  
		
		Moreover, we say that $X$ is pseudo-shattered by $\mathcal{G}$, if there are real numbers $r_1, r_2, \dots, r_m$, such that for each $v \in \{0,1\}^m$, there exists a function $g_v \in \mathcal{G}$ with $\sgn(g_v(x_i) - r_i) = v_i$ for $1 \leq i \leq m$. The pseudo-dimension of $\mathcal{G}$, denoted by $\text{PDim}(\mathcal{G})$, is defined as the maximum cardinality of a subset $X$ of $\mathcal{X}$ that is pseudo-shattered by $\mathcal{G}$.
	\end{definition}
	
	For a hypothesis space $\F$ of functions defined on $\R^d$ and for a multi-index $\alpha\in\N^d$, we denote by $\partial^\alpha\F:=\{\partial^\alpha f: f\in\F\}$ the space of $\alpha$-th order partial derivatives, understood in the weak sense if necessary. We have the following VC dimension bounds.
	
	\begin{theorem}
		\label{thm:vcdim_bounds}
		Let $r\ge2$ be an integer, let $g,\rho\in C^r(\bar\Omega)$, and denote by $\F=\F(\rho,L,\bm,\sigma,\infty,\infty,\infty)$ the hypothesis space in \eqref{eqn:boundary_fcnn}, with unrestricted parameter magnitudes. Set $m_0=d$ and denote by $\size=\sum_{\ell=1}^L m_\ell(m_{\ell-1}+1)+m_L+1$ its number of network parameters.
		\begin{itemize}
			\item If $\sigma$ is the ReQU activation, then
			\begin{equation}
				\label{eqn:vcdim_requ}
				\sup_{\substack{\alpha\in\N^d\\\|\alpha\|_1\le2}}
				\VCdim(\partial^\alpha\F)
				=\bigO\left(\size L\left[L+\log_2(2\|\bm\|_\infty)\right]\right).
			\end{equation}
			\item If $\sigma$ is the tanh activation, then
			\begin{equation}
				\label{eqn:vcdim_tanh}
				\sup_{\substack{\alpha\in\N^d\\\|\alpha\|_1\le r}}
				\VCdim(\partial^\alpha\F)
				=\bigO\left(L^2\|\bm\|_\infty^2\size^2\right),
			\end{equation}
			with an implicit constant that may depend on $r$.
		\end{itemize}
	\end{theorem}

	The VC dimension bounds provided by Theorem~\ref{thm:vcdim_bounds} are, to the best of our knowledge, new in the literature. Bounds for hypothesis spaces of (higher-order) derivatives of ReLU or ReLU-ReLU$^k$ networks have appeared in previous works \citep{yang2023nearly,lei2025solving}, based on a seminal result of \citet{bartlett2019nearly}. These analyses however do not extend to our setup of ReQU FNNs multiplied by a smooth function. Furthermore, the tanh activation not being a piecewise polynomial, controlling VC dimensions of higher derivatives of tanh networks requires a fundamentally different approach. We give a complete proof of Theorem~\ref{thm:vcdim_bounds} in Appendix~\ref{sec:vcdim_bounds}.
	
	\begin{remark}
		Note that in the VC dimension bounds above, $\rho$ is not required to be a distance approximation, but merely a $C^r(\bar{\Omega})$ function. In particular, applying Theorem \ref{thm:vcdim_bounds} with $g\equiv0$ and $\rho\equiv1$ yields VC dimension bounds for up to second order derivatives of FNNs with ReQU activation, and for \emph{any order} derivatives of FNNs with tanh activation function.
	\end{remark}
	
	We now recall the relationship between pseudodimension and VC dimension. The below result can be found in \citep[Theorem 14.1]{bartlett1999neural}:
	\begin{lemma}
		\label{lemma:pseudodim_vcdim}
		\label{eqn:pseudodim}
		For any real-valued function class $\mathcal G$ on $\mathcal X$, define $\mathcal G_0:=\{(x,y)\mapsto h(x)-y:h\in\mathcal G\}$. Then $\Pdim(\mathcal G)=\VCdim(\mathcal G_0).$
	\end{lemma}
	
	By applying Lemma~\ref{lemma:pseudodim_vcdim} to the derivative classes, we obtain the following pseudodimension bounds, whose proof we give in Appendix~\ref{sec:vcdim_bounds}.

	\begin{lemma}
		\label{lem:pdim_from_vcdim}
		Under the hypotheses of Theorem~\ref{thm:vcdim_bounds},
		\begin{equation}
			\label{eqn:pseudodim_2}
			\sup_{\substack{\alpha\in\N^d\\\|\alpha\|_1\le2}}
			\Pdim(\partial^\alpha\F)
			=\begin{cases}
				\bigO\left(\size L[L+\log_2(2\|\bm\|_\infty)]\right),&\sigma=\ReQU,\\
				\bigO\left(L^2\|\bm\|_\infty^2\size^2\right),&\sigma=\tanh.
			\end{cases}
		\end{equation}
	\end{lemma}

	\subsubsection{An oracle inequality for the statistical error}
	
	The importance of these complexity measure bounds is highlighted in the following oracle inequality, which directly translates bounds on VC and pseudo-dimension into high-probability error bounds on the statistical error.
	
	\begin{theorem}
		\label{thm:zhou_oracle}
		Assume that Assumptions \eqref{eqn:assumption1_smooth_boundary}---\eqref{eqn:assumption4_regular_coeffs} hold, so that $u^\star$ is the unique solution of \eqref{eqn:bvp}. Let $\rho\in C^2(\bar\Omega)$ vanish on $\partial\Omega$, and denote by $\F=\F(\rho,L,\bm,\sigma,Q,B,M)$ the hypothesis space in \eqref{eqn:boundary_fcnn}. Define
		\begin{equation}
			\label{eqn:qdef}
			Q_0:=1+\|g\|_{W^{2,\infty}(\Omega)}
			+4Q\|\rho\|_{W^{2,\infty}(\Omega)}
			+\|u^\star\|_{W^{2,\infty}(\Omega)}.
		\end{equation}
		There exists a constant $C>0$, depending on the fixed PDE data, $g,\rho,Q$ and $\Omega$ only, such that for all $t\ge1$, $n\ge\Pdim_\F\vee 8e^2Q_0^2/(C_{11}+3)$, and $u_n\in\argmin_{u\in\F}\mathcal R_n(u)$, we have with probability at least $1-\exp(-t)$
		\begin{equation}
			\label{eqn:oracle_inequality}
			\mathcal R(u_n)\le C\left(
			\inf_{u\in\F}\mathcal R(u)
			+\frac{\Pdim_\F}{n}\log n + \frac tn
			+\frac{\exp(-t)}{\log(n/\log n)}
			\right),
		\end{equation}
		where $\Pdim_\F:=\max_{\|\alpha\|_1\le2}\Pdim(\partial^\alpha\F)$, and $C_{11}>0$ is chosen so that $\|z\|_{H^2(\Omega)}^2\le C_{11}\|\mathscr Lz\|_{L^2(\Omega)}^2$ for all $z\in H^2(\Omega)\cap H^1_0(\Omega)$, as follows from \eqref{eqn:strong_convexity_lower_bound}.
	\end{theorem}
	
	\begin{proof}
		The result follows from following the localized Rademacher-complexity analysis and peeling argument used in the proof of \citep[Theorem 4]{lei2025solving}, which is the analogue of Theorem~\ref{thm:zhou_oracle} in our setting. Indeed, the loss function we use is a special case of theirs, and one can check that their argument is agnostic to the exact form of the neural network hypothesis space.
	\end{proof}
	
	Theorems~\ref{thm:zhou_oracle}, \ref{thm:approx_requ} and \ref{thm:approx_tanh} combined together give us an explicit upper bound for the excess risk, from which we can establish our a priori error bounds.

	\subsection{Main result: $H^2(\Omega)$ a priori error estimates for boundary-adapted PINN}
	
	We are now ready to state and prove our main theorem, which is a non-asymptotic a priori error estimate for solutions of elliptic Dirichlet problems learned with boundary-adapted PINNs.
	
	\begin{theorem}
		\label{thm:main}
		Let $r\ge3$, and assume that Assumptions \eqref{eqn:assumption1_smooth_boundary}---\eqref{eqn:assumption4_regular_coeffs} hold, so that $u^\star$ is well-defined as the unique solution of \eqref{eqn:bvp}. Let $\rho$ satisfy Definition~\ref{def:boundary_adapted_fnns}, $v$ be the extension of $(u^\star - g)/\rho$ from Lemma~\ref{lemma:extension_lemma}, and fix $Q=1+2\|v\|_{W^{r,\infty}(\R^d)}$. Denote by $\F_n=\F(\rho,L,\bm_n,\sigma,Q,\infty,\infty)$ the corresponding hypothesis space, with unrestricted parameter magnitudes. Let $u_n\in\argmin_{u\in\F_n}\mathcal R_n(u)$. The following holds:
		\begin{itemize}
			\item If $\sigma$ is the ReQU activation, choose $L=1$ and $m_{1,n}=N_n^d$, where $N_n=\lceil n^{1/(d+2\gamma(r))}\rceil$ and $\gamma(r)=\min\{r-2,(d+1)/2\}$. Then for all $n\ge N_{\ReQU}$, with probability at least $1-\exp(-n^{d/(d+2\gamma(r))})$,
			\[
			\begin{cases}
				\mathcal R(u_n) - \mathcal{R}(u^\star)&\le c_1 n^{-2\gamma(r)/(d+2\gamma(r))}\log^2 n,\\
				\|u_n-u^\star\|_{H^2(\Omega)}&\le c_2 n^{-\gamma(r)/(d+2\gamma(r))}\log n,
			\end{cases}
			\]
			where $N_{\ReQU},c_1,c_2>0$ depend on $\mathscr L,u^\star,g,\rho,r,d$ and $\Omega$ only.
			\item If $\sigma$ is the tanh activation, choose $L=2$, $m_{1,n}=\lceil A_{d,r}N_n\rceil$ and $m_{2,n}=\lceil A_{d,r}N_n^d\rceil$, where $N_n=\lceil n^{1/[2(2d+r-1)]}\rceil$ and $A_{d,r}$ is from Theorem~\ref{thm:approx_tanh}. Then for all $n\ge N_{\tanh}$, with probability at least $1-\exp(-n^{(2d+1)/(2d+r-1)})$,
			\[
			\begin{cases}
				\mathcal R(u_n)-\mathcal R(u^\star)&\le\kappa_1 n^{-(r-2)/(2d+r-1)}\log^4 n,\\
				\|u_n-u^\star\|_{H^2(\Omega)}&\le\kappa_2 n^{-(r-2)/[2(2d+r-1)]}\log^2 n,
			\end{cases}
			\]
			where $N_{\tanh},\kappa_1,\kappa_2>0$ likewise depend on $\mathscr L,u^\star,g,\rho,r,d$ and $\Omega$ only.
		\end{itemize}
	\end{theorem}

	\begin{proof}
		We give a detailed proof for the ReQU case, and sketch at the end how the argument needs to be modified for tanh. Let $a>0$ be an exponent to be chosen later and set $N=\lceil n^a\rceil$. Consider the class $\F_n$ with $L=1$, $m_1=N^d$ and $Q=1+2\|v\|_{W^{r,\infty}(\R^d)}$. By applying the ReQU version of Corollary~\ref{cor:approx_error}, we find that
		\[
		\inf_{u\in\F_n}\mathcal R(u)
		\le C_{\ReQU}\|(u^\star-g)/\rho\|_{W^{r,\infty}(\Omega)}^2 N^{-2\gamma(r)}
		\le C_1 n^{-2a\gamma(r)}.
		\]
		On the other hand, the number of parameters in this class is given by $\size=(d+2)m_1+1=\bigO(N^d)$. Hence, by Lemma~\ref{lem:pdim_from_vcdim}, $\Pdim_{\F_n}\le C_2N^d\log N\le C_3n^{ad}\log n$. For $ad<1$, the sample-size condition in Theorem~\ref{thm:zhou_oracle} therefore holds for all sufficiently large $n$, so we can apply it and get that with probability at least $1-\exp(-t)$,
		\[
		\mathcal R(u_n)\le C\left(
		n^{-2a\gamma(r)}+n^{ad-1}\log^2 n
		+\frac tn+\frac{\exp(-t)}{\log(n/\log n)}
		\right).
		\]
		We thus pick $a=1/(d+2\gamma(r))$ to balance $2a\gamma(r)=1-ad$, and $t=n^{d/(d+2\gamma(r))}$. The last two terms are bounded by the same algebraic rate for large $n$, yielding the stated risk bound.
		
		Finally, we deduce the desired $H^2$ a priori error bound by applying our strong convexity estimate \eqref{eqn:strong_convexity_lower_bound}. Since $u_n-u^\star\in H^2(\Omega)\cap H^1_0(\Omega)$ and the risk is normalized by $|\Omega|$,
		\[
		\|u_n-u^\star\|_{H^2(\Omega)}^2
		\le C_{11}|\Omega|\mathcal R(u_n).
		\]
		Taking square roots concludes the ReQU proof, with $c_2=\sqrt{C_{11}|\Omega|c_1}$.
		
		For tanh, the widths are $m_1=\bigO(N)$ and $m_2=\bigO(N^d)$, so $\size=\bigO(N^{d+1})$ and \eqref{eqn:pseudodim_2} gives $\Pdim_{\F_n}=\bigO(N^{4d+2})$. With $N=\lceil n^a\rceil$, the approximation and sampling terms are bounded by $Cn^{-2a(r-2)}\log^4 n$ and $Cn^{a(4d+2)-1}\log n$, respectively. Balancing the powers gives $a=1/[2(2d+r-1)]$. Choosing $t=n^{(2d+1)/(2d+r-1)}$ and applying the same stability estimate yields the claimed bounds for tanh.
	\end{proof}

	Theorem~\ref{thm:main} yields non-asymptotic high-probability $H^2(\Omega)$ error bounds for boundary-adapted PINN estimators of the solution to the Dirichlet problem \eqref{eqn:bvp}. In the tanh case, the rate improves with the regularity $r$ and becomes close to $n^{-1/2}$, up to logarithmic factors, when $r\gg d$. Our tanh bound is qualitatively similar to fast-rate results for deep ReLU$^3$ and ReLU-ReLU$^k$ such as \citep{lu2021machine, jiao2022rate, lei2025solving}, but the exponent is slightly degraded due to the comparatively larger VC dimension of tanh networks, as shown in Theorem~\ref{thm:vcdim_bounds}. For the shallow ReQU case, the exponent we get for $r<(d+5)/2$ matches the rate of \citep{lu2021machine}, which was shown to be minimax optimal. Unfortunately, the gain from additional smoothness saturates for large $r$, and we can't reach the $n^{-1/2}$ rate. This saturation comes from our shallow ReQU approximation bound in Theorem~\ref{thm:approx_requ}. Increasing depth could overcome it, at the cost of making training more unstable.
	
	\section{Numerical experiments}
	\label{sec:numerical_experiments}
	
	We now illustrate our theoretical results by numerically solving two Dirichlet problems of interest using our boundary-adapted PINN ansatz and comparing it with other related neural network-based PDE solving methods\footnote{The implementation details and source code are available at \url{https://github.com/n2pakbong/boundary-adapted-pinns}.}. The problems we will consider are the mean exit time of an Ornstein-Uhlenbeck process from a ball, and a committor problem under periodic boundary conditions, which both can be expressed as in \eqref{eqn:bvp}.
	
	In all experiments we approximate $u^\star$ by neural network ansatz functions $f_{\btheta}: \Omega \to \mathbb{R}$, which are then trained using different PINN formulations. We consider the boundary-adapted PINN as defined in this paper, the standard strong-form PINN with an explicit boundary penalty term as introduced in \citep{raissi2019physics}, and a variational PINN (VPINN) based on a weak formulation with Nitsche-type boundary terms \citep{ming2021deep}. In addition, we introduce two modified variants of the boundary-adapted architecture, obtained by multiplying the neural network by $\rho^{3/2}$ and $\rho^{1/2}$ instead of $\rho$, where $\rho$ denotes a distance approximation satisfying the boundary-adapted condition of Definition~\ref{def:boundary_adapted_fnns}. We refer to these as ``over-enforced'' and ``under-enforced'' PINNs. With the addition of $g$ for nonzero boundary data, these variants still satisfy the Dirichlet boundary condition and are included to assess the importance of the ``distance-like" behaviour of $\rho$ near $\paOm$. Before reporting the numerical results, we briefly describe these architectures and the associated training methodology.
	
	\subsection{Training losses and evaluation metrics}
	
	Within each example, all methods share the same fully connected feedforward network $N_{\btheta}$, and differ only in how the boundary condition is enforced and in the resulting training losses.
	
	\noindent\textbf{PDE residual, boundary, and data-fidelity losses.}
	For all formulations that use a strong-form residual, we define the empirical PDE loss
	\begin{equation}
		\label{eqn:Jpde}
		\hat J_{\mathrm{pde}}(\btheta)=\frac1n\sum_{i=1}^n(\mathscr Lf_{\btheta}(x_i)-f(x_i))^2.
	\end{equation}
	The Dirichlet boundary condition $u^\star=g$ is enforced in the standard PINN by the empirical $L^2(\partial\Omega)$ boundary penalty
	\begin{equation}
		\label{eqn:Jbc-simple}
		\hat J_{\mathrm{bc}}(\btheta)=\frac{1}{n_{\mathrm{bc}}}\sum_{k=1}^{n_{\mathrm{bc}}}(f_{\btheta}(x_k^{\mathrm{bc}})-g(x_k^{\mathrm{bc}}))^2.
	\end{equation}
	To tame the typically badly-behaved loss-landscape of the objective \eqref{eqn:Jpde} and avoid convergence to spurious local minima, we follow \citep{tepakbong2026taming} and enhance the residual loss with a data-fidelity loss based on Monte Carlo approximations of $u^\star$ at interior points\footnote{For each data point $x_j^{\mathrm{data}}$, we approximate $u^\star(x_j^{\mathrm{data}})$ by Monte Carlo: we generate $N_{\mathrm{paths}}$ independent Euler-Maruyama trajectories of the corresponding stochastic process started at $x_j^{\mathrm{data}}$. For the OU problem, we estimate the mean exit time as the empirical mean of their first exit times. For the committor problem, among trajectories that reach $A$ or $B$, we estimate the probability of reaching $B$ first as the proportion of trajectories that do so.}. For a small number $n_{\mathrm{data}}$ of uniformly sampled points $\{x_j^{\mathrm{data}}\}_{j=1}^{n_{\mathrm{data}}}\subset\Omega$ and precomputed estimates $\hat u^\star(x_j^{\mathrm{data}})$, we set
	\begin{equation}
		\label{eqn:Jdata}
		\hat J_{\mathrm{data}}(\btheta)=\frac1{n_{\mathrm{data}}}\sum_{j=1}^{n_{\mathrm{data}}}(f_{\btheta}(x_j^{\mathrm{data}})-\hat u^\star(x_j^{\mathrm{data}}))^2.
	\end{equation}
	
	\noindent\textbf{Boundary-adapted, over-enforced, and under-enforced PINNs.}
	The boundary-adapted architecture incorporates the Dirichlet boundary condition into the ansatz via the smooth distance approximation $\rho$, which vanishes on $\partial\Omega$ and is strictly positive in $\Omega$.  
	Writing $N_{\btheta}$ for a standard feedforward network, we define
	\begin{equation}
		f_{\btheta}^{\mathrm{adapted}} := g+ \rho\,N_{\btheta},\quad
		f_{\btheta}^{\mathrm{over}} := g + \rho^{3/2} N_{\btheta},\quad
		f_{\btheta}^{\mathrm{under}} := g + \rho^{1/2} N_{\btheta},
	\end{equation}
	which we refer to as the boundary-adapted, over-enforced, and under-enforced PINNs, respectively.  
	In all three cases the boundary condition is satisfied by construction, so only the PDE residual and data-fidelity losses are used, obtained from \eqref{eqn:Jpde}–\eqref{eqn:Jdata} by replacing $f_{\btheta}$ with the corresponding $f_{\btheta}^{\mathrm{arch}}$.
	
	\noindent\textbf{Standard PINN and variational PINN.}
	The standard PINN uses the strong-form PDE residual \eqref{eqn:Jpde}, the boundary loss \eqref{eqn:Jbc-simple}, and the data-fidelity loss \eqref{eqn:Jdata}.  
	
	For the VPINN we instead adopt a deep Nitsche formulation \citep{ming2021deep}. We fix in advance a finite collection of smooth test functions $\{\phi_k\}_{k=1}^{n_{\mathrm{test}}} \subset C^\infty(\bar\Omega)$. Let $a_{\mathrm{N}}(\cdot,\cdot)$ and $l(\cdot)$ denote the Nitsche bilinear and linear forms obtained from $\mathscr{L}$ by integration by parts together with Nitsche-type boundary terms. In particular, the last term in $a_{\mathrm{N}}$ corresponds to the weak imposition of the Dirichlet condition. The exact solution $u$ then satisfies
	\begin{equation*}
		a_{\mathrm{N}}(u^\star,v) = l(v)\quad\text{for all } v\in C^\infty(\bar\Omega).
	\end{equation*}
	Given a neural approximation $f_{\btheta}$, we evaluate, for each test function $\phi_k$,
	\begin{equation*}
		R_k(f_{\btheta}) := a_{\mathrm{N}}(f_{\btheta},\phi_k) - l(\phi_k),
		\quad k=1,\dots,n_{\mathrm{test}},
	\end{equation*}
	where all domain and boundary integrals in $a_{\mathrm{N}}$ and $l$ are approximated by Monte Carlo quadrature. We decompose $R_k(f_{\btheta})$ into its interior-plus-flux contribution and its Nitsche boundary-penalty contribution and define the associated empirical losses as the average squared residuals of these two parts, denoted (with a mild abuse of notation) by $\hat{J}_{\mathrm{pde}}(\btheta)$ and $\hat{J}_{\mathrm{bc}}(\btheta)$, respectively. The VPINN also uses the same data-fidelity loss \eqref{eqn:Jdata}, but does not use the pointwise boundary loss \eqref{eqn:Jbc-simple}.
	
	\noindent\textbf{Adaptive weighting.}
	Rather than fixing the loss weights by hand, we learn them adaptively following the uncertainty-based multi-task weighting proposed in \citep{kendall2018multi}, which we have found to consistently improve the performance of all our models. For a given architecture, let $\{\hat{J}_i(\btheta)\}_{i \in \mathcal{I}}$ denote its active loss components (PDE, boundary, data). We introduce one trainable positive parameter $\lambda_i$ per component and define
	\begin{equation}
		\label{eqn:J-adaptive}
		\hat{J}(\btheta,\{\lambda_i\}_{i\in\mathcal{I}})
		:= \sum_{i\in\mathcal{I}} \frac{1}{2\lambda_i^2}\,\hat{J}_i(\btheta)
		\;+\; \sum_{i\in\mathcal{I}} \log \lambda_i.
	\end{equation}
	In particular, the standard PINN and the VPINN use three adaptive weights $\lambda_{\mathrm{pde}},\lambda_{\mathrm{bc}},\lambda_{\mathrm{data}}$ (for the VPINN, $\lambda_{\mathrm{bc}}$ is associated with the Nitsche boundary term), while the boundary-adapted, over-enforced, and under-enforced PINNs use two weights $\lambda_{\mathrm{pde}},\lambda_{\mathrm{data}}$. All network and weighting parameters are learned jointly by minimizing the objective \eqref{eqn:J-adaptive}.
	
	\noindent\textbf{Error metrics.}
	In addition to the training losses $\hat{J}_i$, we evaluate each trained model using relative $L^2$, $H^1$, and $H^2$ errors with respect to the solution $u^\star$. For a given approximation $f_{\btheta}$ and $X \in \{L^2,H^1,H^2\}$, we define
	\begin{equation}
		\label{eqn:rel-errors}
		X_{\mathrm{rel}}(f_{\btheta}):=\frac{\|f_{\btheta}-u^\star\|_{X(\Omega)}}{\|u^\star\|_{X(\Omega)}}.
	\end{equation}
	
	\subsection{Mean exit time of an Ornstein-Uhlenbeck process in a ball}
	
	The first problem we consider is the computation of the mean exit time (MET) of a two-dimensional Ornstein-Uhlenbeck (OU) process from $\Omega:=\{x\in\R^2:\|x\|_2<R\}$. For a system fluctuating near a stable state, the mean exit time measures how long on average that state persists before noise drives an escape, as in a transition out of a potential well. Computing this waiting time is a fundamental problem in metastability, with far reaching applications \citep{freidlin2012random}.
	
	Formally, consider the process $(X_t^x)_{t\ge0}$ initiated at $x\in\Omega$ and solving the SDE $dX_t^x=-\theta X_t^x\,dt+\varsigma\,dW_t$, with $X_0^x=x$, where we set $R=2$, $\theta=3$ and $\varsigma^2=3$ in our numerical experiments. The MET of $(X_t^x)_{t\ge0}$ from $\Omega$ is defined as $\tau:x\mapsto\Exp\left[\inf\{t\ge0: X_t^x \not\in \Omega\}\right]$. One can show \citep[Result 7.1]{pavliotis2014stochastic} that $u^\star\equiv\tau$ solves the Dirichlet problem:
	\begin{equation}
		\label{eqn:ou_bvp}
		\begin{cases}
			\mathscr Lu^\star=-1&\text{in }\Omega,\\
			u^\star=0&\text{on }\partial\Omega,
		\end{cases}
	\end{equation}
	
	where $\mathscr L\phi(x)=-\theta x\cdot\nabla \phi(x)+(\varsigma^2/2)\Delta \phi(x)$ is the \textit{infinitesimal generator} of $(X_t^x)_{t\ge 0}$ \citep{pavliotis2014stochastic}.
	
	All PINN variants in this example share the same fully connected feedforward network with $\tanh$ activations, two hidden layers, and $30$ neurons per hidden layer. For the boundary-adapted, over-enforced, and under-enforced architectures we use the smooth distance approximation $\rho(x) = \frac{R^2 - \|x\|^2}{2R}$, which one can verify satisfies the conditions in Definition \ref{def:order_m_normalization}. The strong-form residual and boundary terms are approximated using $n_{\mathrm{pde}} = 8192$ interior collocation points and $n_{\mathrm{bc}} = 2048$ boundary collocation points.  
	
	For the variational PINN we additionally fix $n_{\mathrm{test}} = 50$ randomly generated test functions of the form $
	\phi_k(x) = \exp\!\left(-\frac{\|x - c_k\|^2}{2 s_{\text{test}}^2}\right)$,
	with centers $c_k$ sampled uniformly in $\Omega$, and fixed length scale $s_{\text{test}} = R/5$. Lastly, the data-fidelity loss uses $n_{\mathrm{data}} = 256$ interior data points with Monte Carlo estimates of $\tau(x_j^{\mathrm{data}})$.  
	
	All models are trained for 10000 epochs of Adam on the same decaying schedule. The reported error metrics are computed with respect to the exact mean escape time $\tau$, whose exact expression has been derived in \citep{kersting2023mean}.

	\begin{figure}[htbp]
		\centering
		\begin{minipage}[t]{0.32\textwidth}
			\centering
			\includegraphics[width=\linewidth]{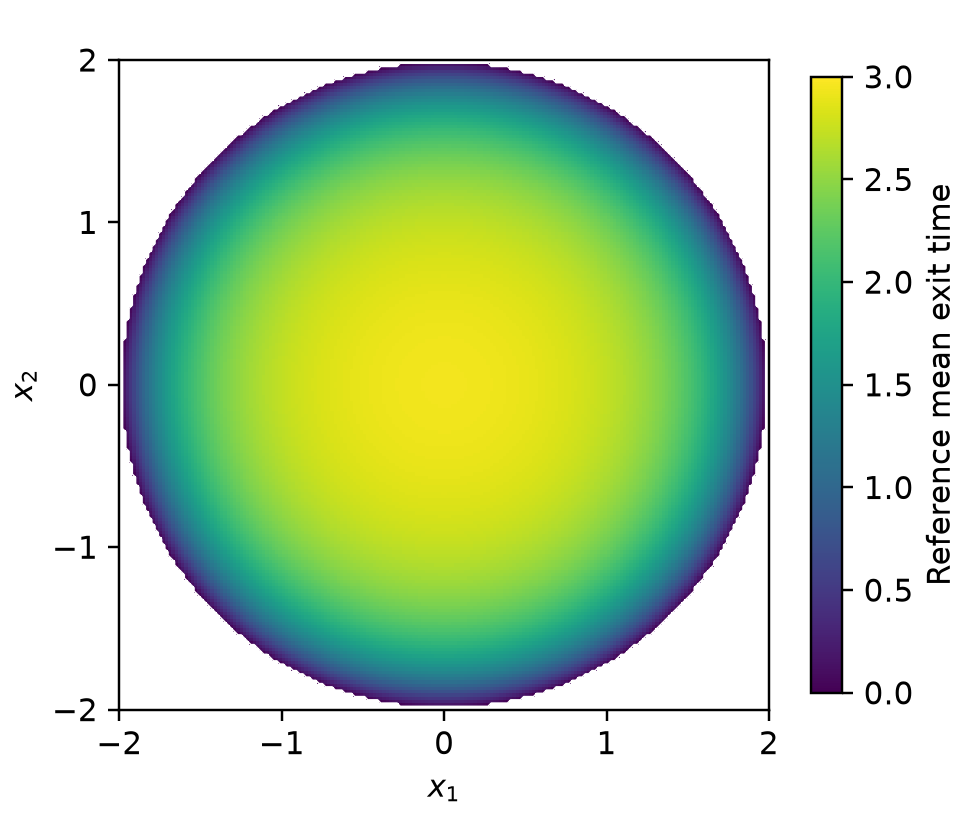}
			\par\small Reference
		\end{minipage}\hfill
		\begin{minipage}[t]{0.32\textwidth}
			\centering
			\includegraphics[width=\linewidth]{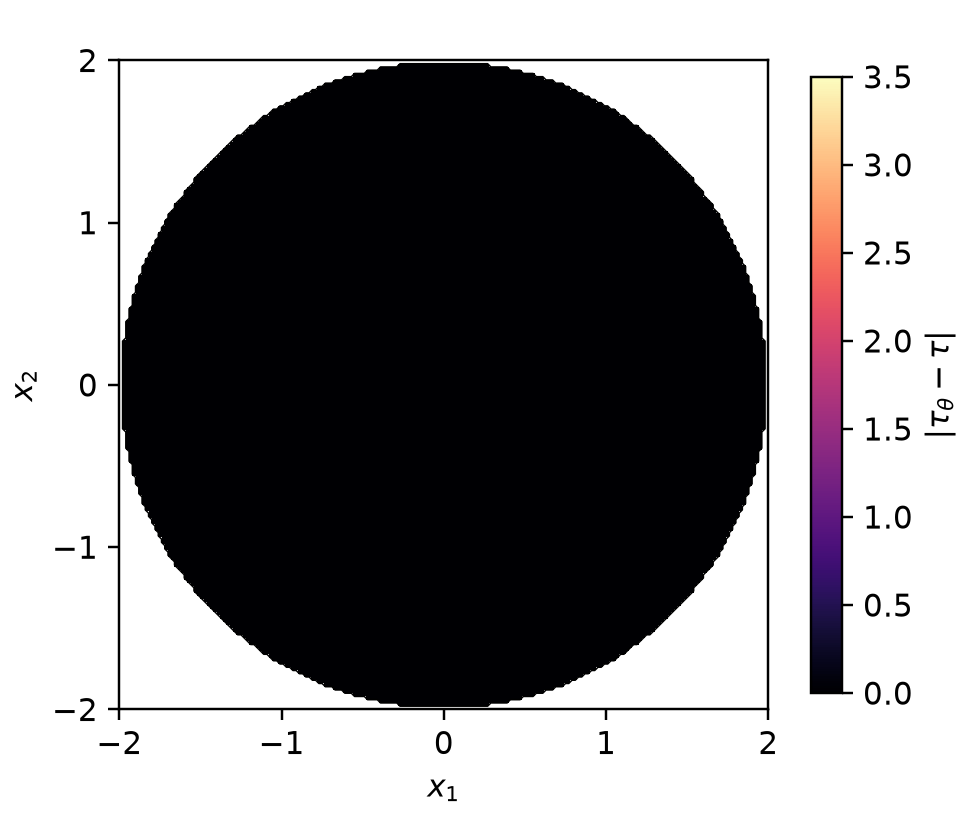}
			\par\small Adapted PINN error
		\end{minipage}\hfill
		\begin{minipage}[t]{0.32\textwidth}
			\centering
			\includegraphics[width=\linewidth]{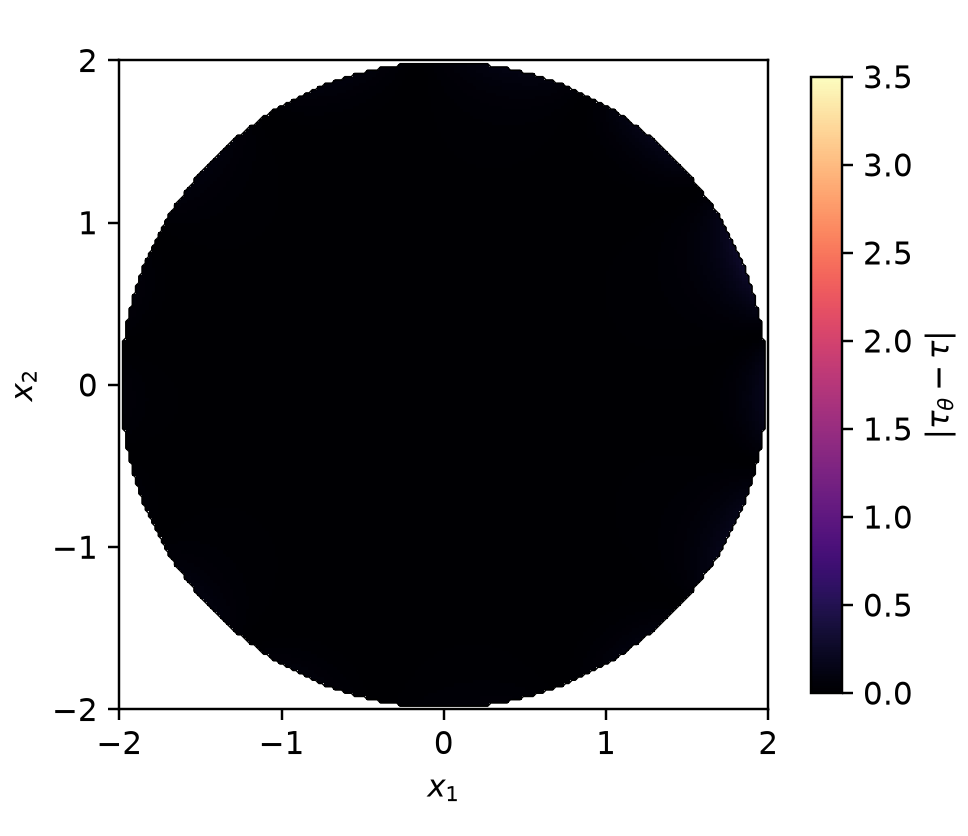}
			\par\small Standard PINN error
		\end{minipage}
		
		\begin{minipage}[t]{0.32\textwidth}
			\centering
			\includegraphics[width=\linewidth]{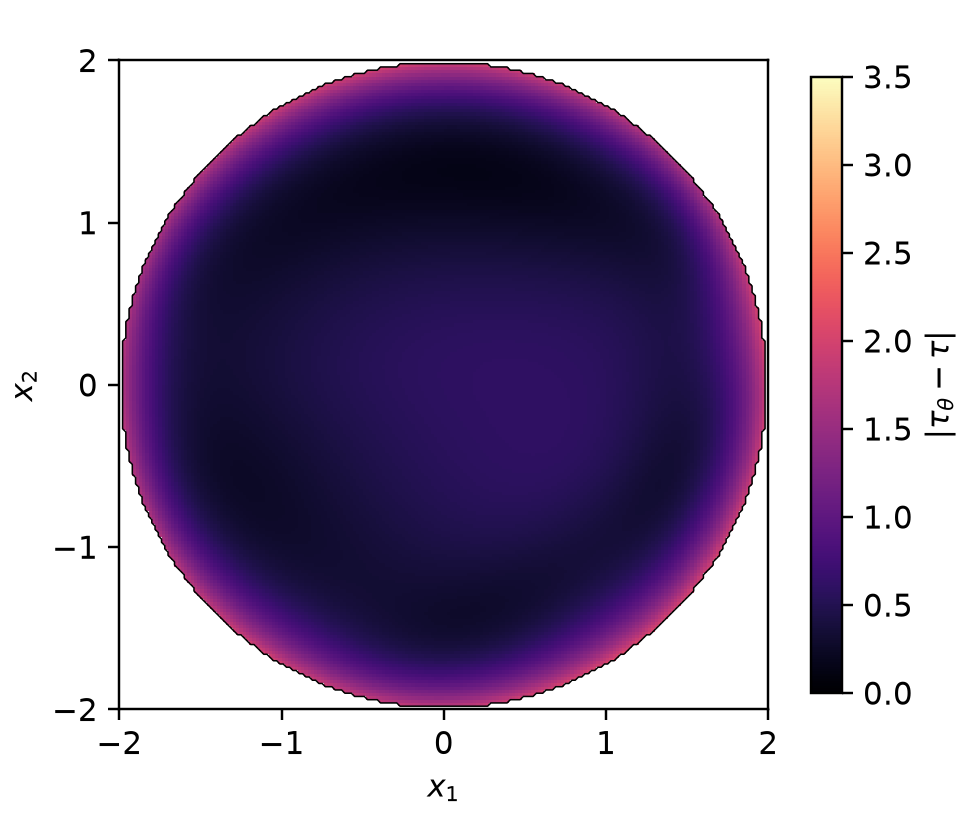}
			\par\small Variational error
		\end{minipage}\hfill
		\begin{minipage}[t]{0.32\textwidth}
			\centering
			\includegraphics[width=\linewidth]{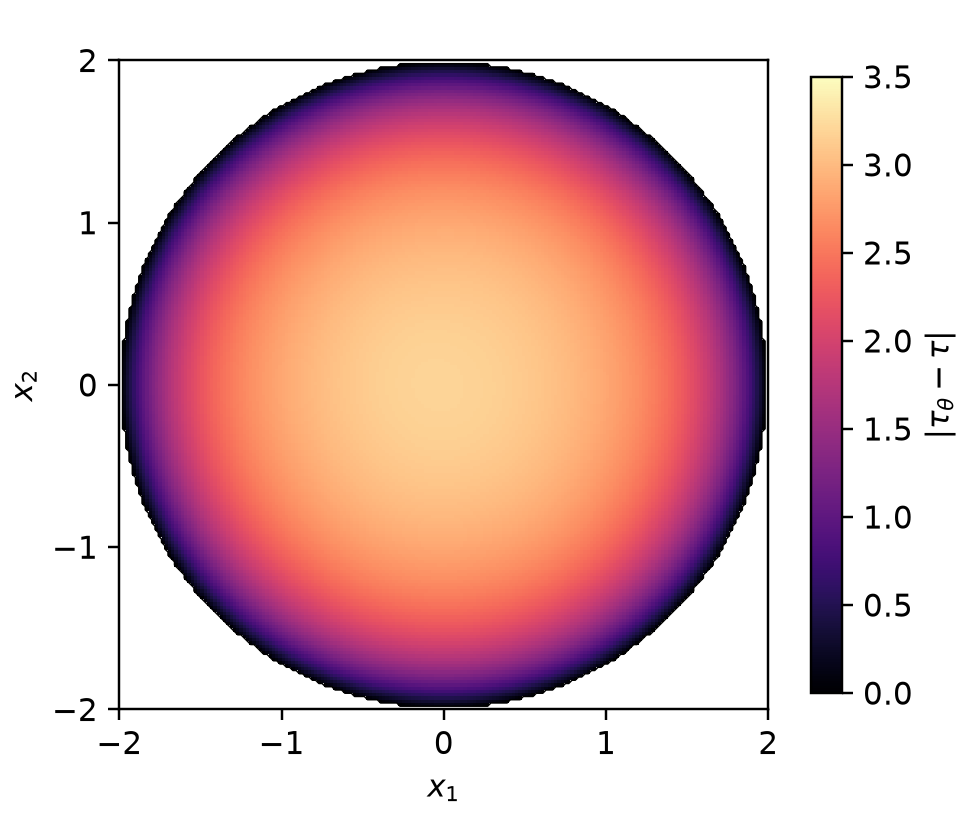}
			\par\small Over-enforced PINN error
		\end{minipage}\hfill
		\begin{minipage}[t]{0.32\textwidth}
			\centering
			\includegraphics[width=\linewidth]{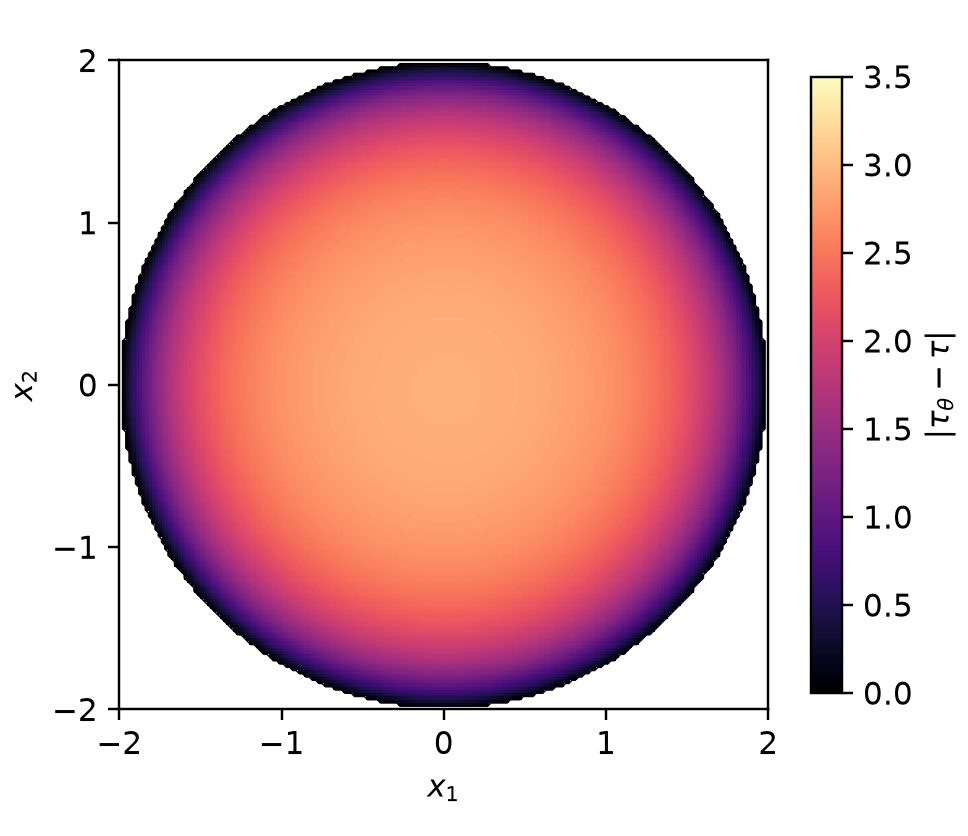}
			\par\small Under-enforced PINN error
		\end{minipage}

		\caption{Reference mean exit time and pointwise absolute errors for the OU problem. Relative errors are reported in Table~\ref{tab:ou2d_errors}.}
		\label{fig:ou2d_true}
		\label{fig:ou2d_methods_abc}
		\label{fig:ou2d_methods_de}
	\end{figure}

	\begin{table}[htbp]
		\centering
		\resizebox{\textwidth}{!}{%
			\begin{tabular}{lccc}
				\toprule
				Method & \shortstack{$L^2_{\rm rel}$ $(\|\tau\|_{L^2}=1.155\times10^1)$}
				& \shortstack{$H^1_{\rm rel}$ $(\|\tau\|_{H^1}=2.266\times10^1)$}
				& \shortstack{$H^2_{\rm rel}$ $(\|\tau\|_{H^2}=1.841\times10^2)$}\\
				\midrule
				Standard & $4.887\times10^{-2}\pm3.872\times10^{-2}$ & $4.994\times10^{-2}\pm3.926\times10^{-2}$ & $5.041\times10^{-2}\pm3.964\times10^{-2}$\\
				Variational & $1.169\times10^{-1}\pm5.697\times10^{-2}$ & $6.491\times10^{-1}\pm6.106\times10^{-1}$ & $3.269\pm4.211$\\
				Over-enforced & $1.013\pm5.774\times10^{-4}$ & $1.006\pm0.000$ & $1.001\pm5.774\times10^{-4}$\\
				Under-enforced & $1.009\pm1.732\times10^{-3}$ & $1.005\pm5.774\times10^{-4}$ & $1.000\pm0.000$\\
				\textbf{Adapted (Ours)} & $\mathbf{2.006\times10^{-3}\pm3.014\times10^{-4}}$ & $\mathbf{1.998\times10^{-3}\pm2.981\times10^{-4}}$ & $\mathbf{1.990\times10^{-3}\pm2.946\times10^{-4}}$\\
				\bottomrule
			\end{tabular}%
		}
		\caption{Error metrics (mean $\pm$ standard deviation over four runs) for different PINN variants solving the Ornstein-Uhlenbeck MET problem.}
		\label{tab:ou2d_errors}
	\end{table}
	
	\subsection{Periodic committor}
	
	The second PDE we consider is the \textit{committor problem}. In molecular dynamics, a committor measures the probability that a configuration reaches a product state $B$ before returning to a reactant state $A$. It is used to identify transition regions and compute reaction rates \citep{e2006towards}. For an overdamped Langevin process on $\mathbb T^2=[0,1)^2$, the forward committor $q(x)$ is this probability when starting from $x$. With $\Omega=\mathbb T^2\setminus(A\cup B)$, it solves the BVP:
	\begin{equation}
		\label{eqn:committor_bvp}
		\begin{cases}
			\varepsilon\Delta q-\nabla V\cdot\nabla q=0&\text{in }\Omega,\\
			q=0&\text{on }\partial A,\\
			q=1&\text{on }\partial B,
		\end{cases}
	\end{equation}
	with periodic identification of the unit cell.
	
	For $j\in\{A,B\}$, we let $c_A=(1/4,0)$, $c_B=(3/4,0)$, and define $\omega_j(x):=4\sum_{i=1}^2\sin^2(\pi(x_i-c_{j,i}))-R^2$, with $R=1/2$. We set the absorbing sets as $A=\{\omega_A\le0\}$ and $B=\{\omega_B\le0\}$, the potential as $V(x)=1-\cos(4\pi(x_1-1/4))+0.6[1-\cos(2\pi x_2)]$, and $\varepsilon=0.4$.
	
	All PINN variants in this example share the same fully connected feedforward network with tanh activations, two hidden layers, and $64$ neurons per hidden layer. To enforce periodicity, its inputs are $(\sin(2\pi x_1),\sin(2\pi x_2),\cos(2\pi x_1),\cos(2\pi x_2))$. For the boundary-adapted, over-enforced and under-enforced architectures $g+\rho^\alpha N_{\btheta} $, we combine the two boundary factors using the R-function construction of \citep{sukumar2022exact} and set:
	\[
	g:=\frac{d_A}{d_A+d_B},
	\text{ and }
	\rho:=\frac{d_Ad_B}{\sqrt{d_A^2+d_B^2}},
	\]
	where $d_j=\omega_j/\sqrt{\omega_j^2+\|\nabla\omega_j\|_2^2}$. These functions are smooth and periodic on $\bar\Omega$. The function $g$ has the required values on $\partial A$ and $\partial B$, and one can check that $\rho$ satisfies Definition~\ref{def:boundary_adapted_fnns} on both boundary components. All variants use the same Monte Carlo data, and the error metrics are computed with respect to a reference finite-difference approximation of the committor.
	
	\begin{table}[htbp]
		\centering
		\resizebox{\textwidth}{!}{%
			\begin{tabular}{lccc}
				\toprule
				Method & $L^2_{\rm rel}$ ($\|q\|_{L^2}=0.6369$)
				& $H^1_{\rm rel}$ ($\|q\|_{H^1}=3.1702$)
				& $H^2_{\rm rel}$ ($\|q\|_{H^2}=41.4861$)\\
				\midrule
				Standard & $6.426\times10^{-1}\pm1.978\times10^{-3}$ & $9.878\times10^{-1}\pm1.399\times10^{-4}$ & $9.995\times10^{-1}\pm2.661\times10^{-4}$\\
				Variational & $4.393\times10^{-2}\pm8.145\times10^{-3}$ & $1.902\times10^{-1}\pm3.953\times10^{-2}$ & $5.427\times10^{-1}\pm2.795\times10^{-1}$\\
				Over-enforced & $1.287\times10^{-1}\pm3.083\times10^{-2}$ & $2.277\times10^{-1}\pm4.541\times10^{-2}$ & $2.588\times10^{-1}\pm3.441\times10^{-2}$\\
				Under-enforced & $3.792\times10^{-1}\pm1.396\times10^{-2}$ & $7.322\times10^{-1}\pm1.630\times10^{-2}$ & $9.639\times10^{-1}\pm6.832\times10^{-3}$\\
				\textbf{Adapted (Ours)} & $\mathbf{8.587\times10^{-3}\pm2.150\times10^{-3}}$ & $\mathbf{1.538\times10^{-2}\pm2.206\times10^{-3}}$ & $\mathbf{1.999\times10^{-2}\pm1.477\times10^{-3}}$\\
				\bottomrule
			\end{tabular}%
		}
		\caption{Error metrics (mean $\pm$ sample standard deviation across four runs) for different PINN variants solving the periodic committor problem.}
		\label{tab:committor_errors}
	\end{table}

	%Table~\ref{tab:committor_errors} reports the comparison over four runs. The boundary-adapted ansatz has mean relative $H^2$ error $1.999\times10^{-2}$, compared with $9.995\times10^{-1}$ for the standard PINN and $5.427\times10^{-1}$ for the variational PINN.
	
	\begin{figure}[htbp]
		\centering
		\begin{minipage}[t]{0.32\textwidth}
			\centering
			\IfFileExists{./numerical_experiments/committor/periodic_two_well/true_solution.png}{%
				\includegraphics[width=\linewidth]{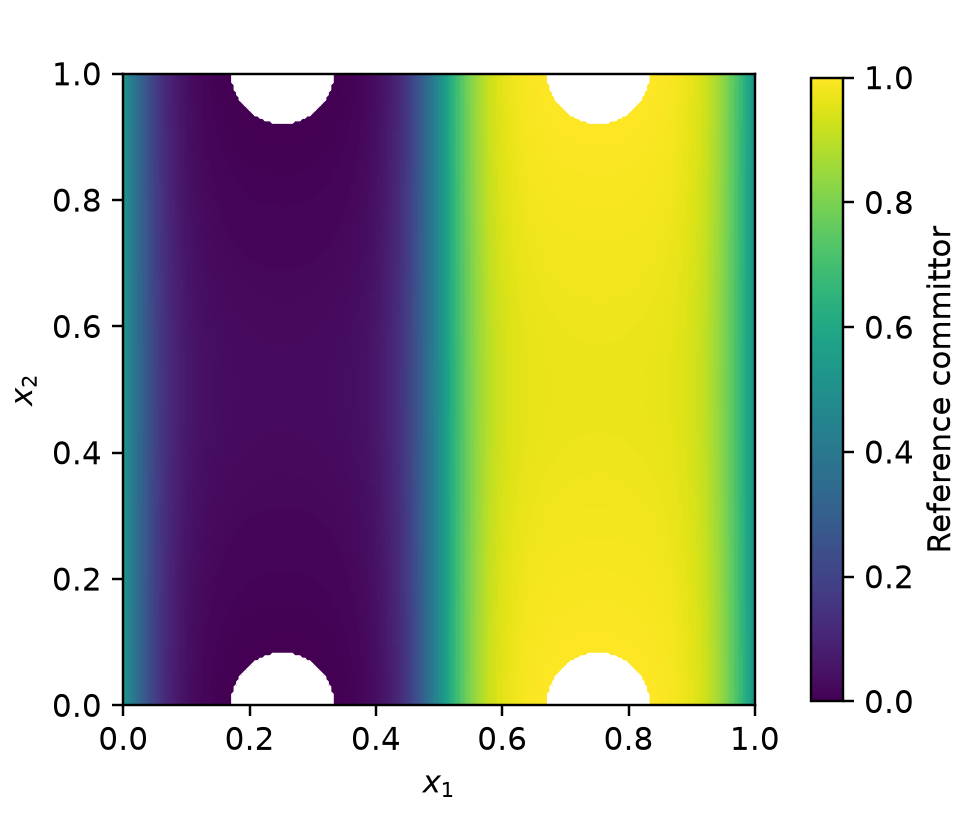}%
			}{%
				\fbox{\parbox[c][0.75\linewidth][c]{\dimexpr\linewidth-2\fboxsep-2\fboxrule\relax}{\centering\small Image pending}}%
			}
			\par\small Reference
		\end{minipage}\hfill
		\begin{minipage}[t]{0.32\textwidth}
			\centering
			\IfFileExists{./numerical_experiments/committor/periodic_two_well/adapted_err.png}{%
				\includegraphics[width=\linewidth]{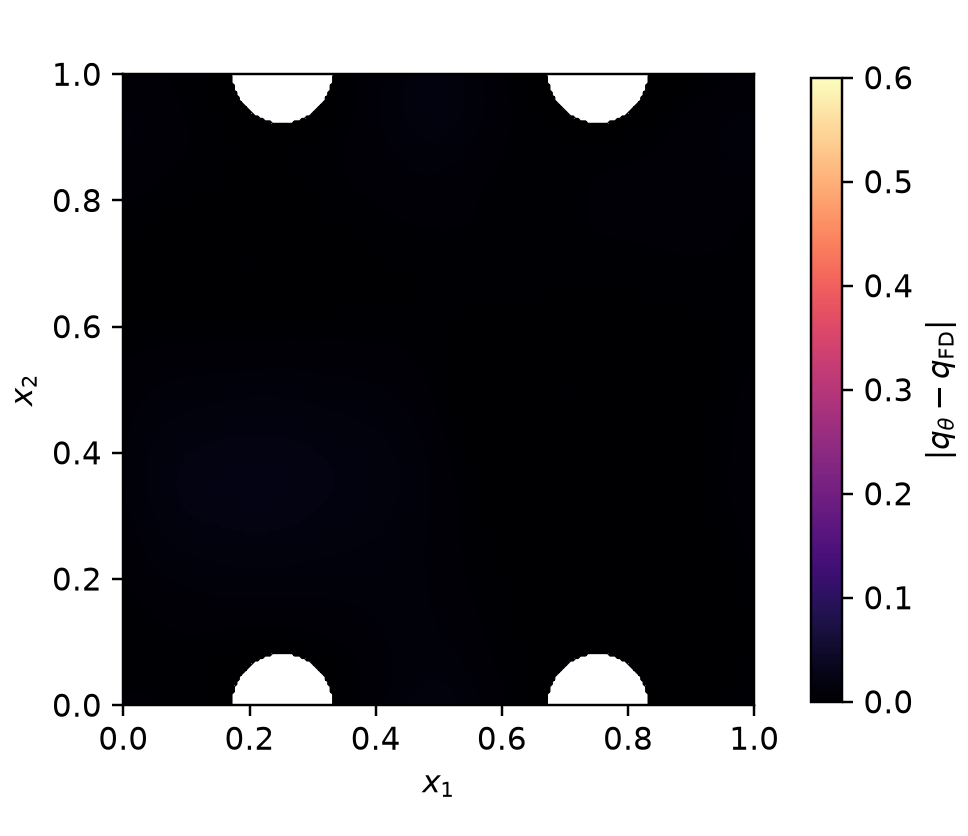}%
			}{%
				\fbox{\parbox[c][0.75\linewidth][c]{\dimexpr\linewidth-2\fboxsep-2\fboxrule\relax}{\centering\small Image pending}}%
			}
			\par\small Adapted PINN error
		\end{minipage}\hfill
		\begin{minipage}[t]{0.32\textwidth}
			\centering
			\IfFileExists{./numerical_experiments/committor/periodic_two_well/simple_err.png}{%
				\includegraphics[width=\linewidth]{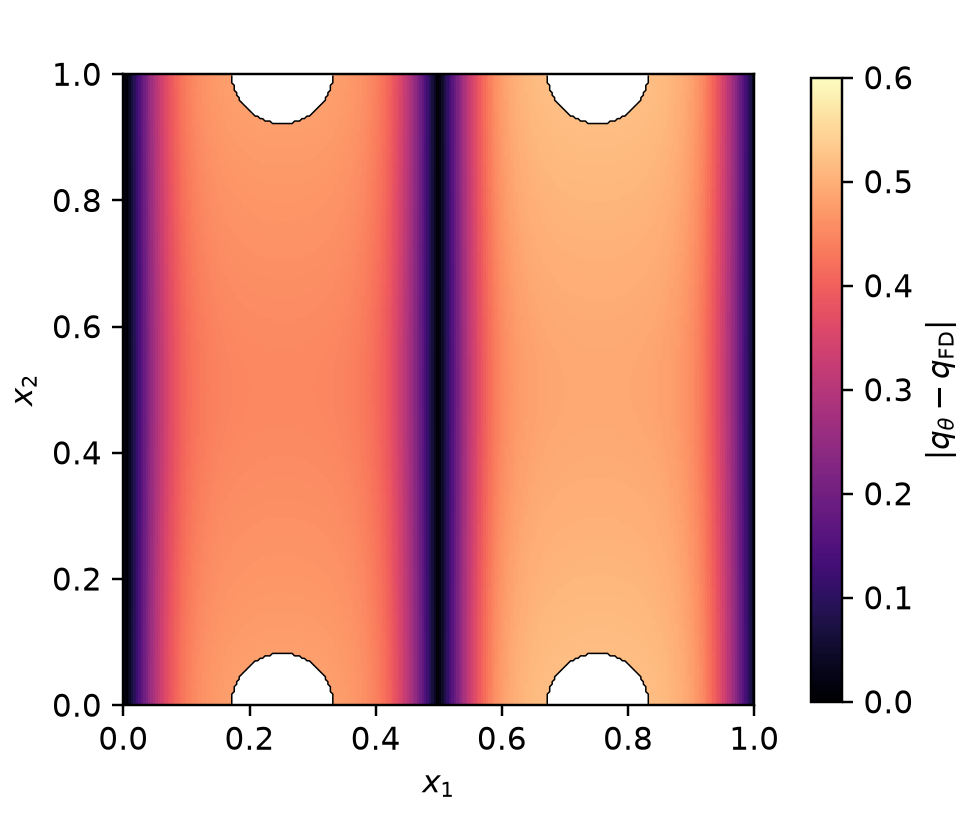}%
			}{%
				\fbox{\parbox[c][0.75\linewidth][c]{\dimexpr\linewidth-2\fboxsep-2\fboxrule\relax}{\centering\small Image pending}}%
			}
			\par\small Standard PINN error
		\end{minipage}
		
		\begin{minipage}[t]{0.32\textwidth}
			\centering
			\IfFileExists{./numerical_experiments/committor/periodic_two_well/variational_err.png}{%
				\includegraphics[width=\linewidth]{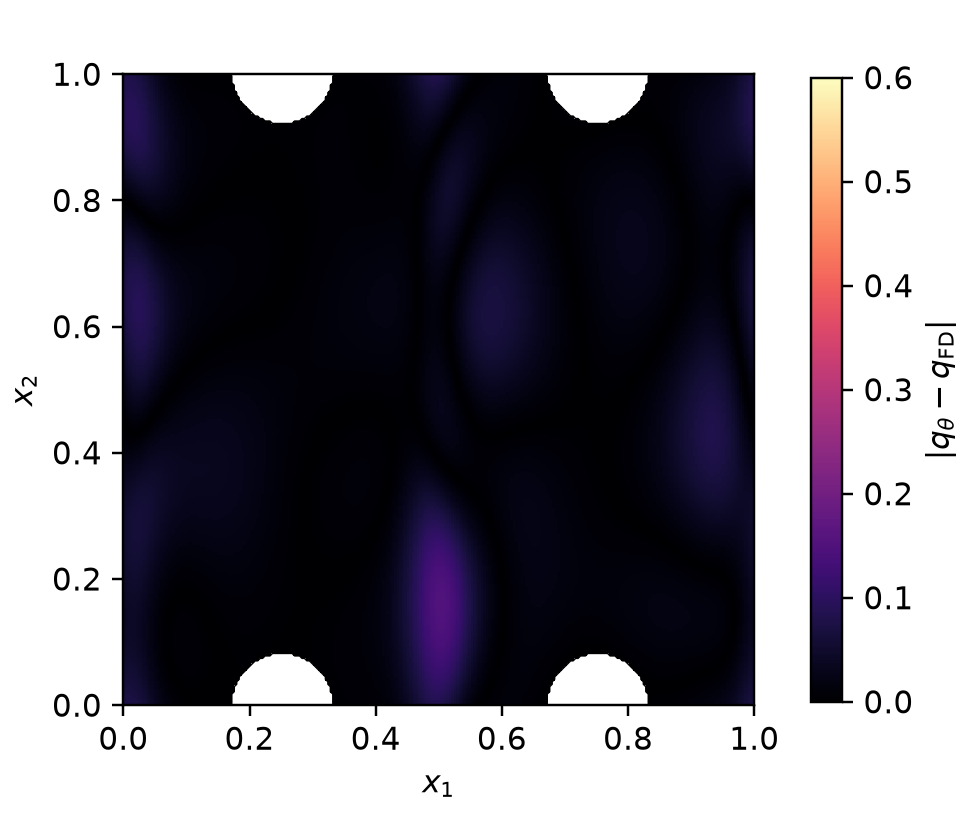}%
			}{%
				\fbox{\parbox[c][0.75\linewidth][c]{\dimexpr\linewidth-2\fboxsep-2\fboxrule\relax}{\centering\small Image pending}}%
			}
			\par\small Variational PINN error
		\end{minipage}\hfill
		\begin{minipage}[t]{0.32\textwidth}
			\centering
			\IfFileExists{./numerical_experiments/committor/periodic_two_well/over_err.png}{%
				\includegraphics[width=\linewidth]{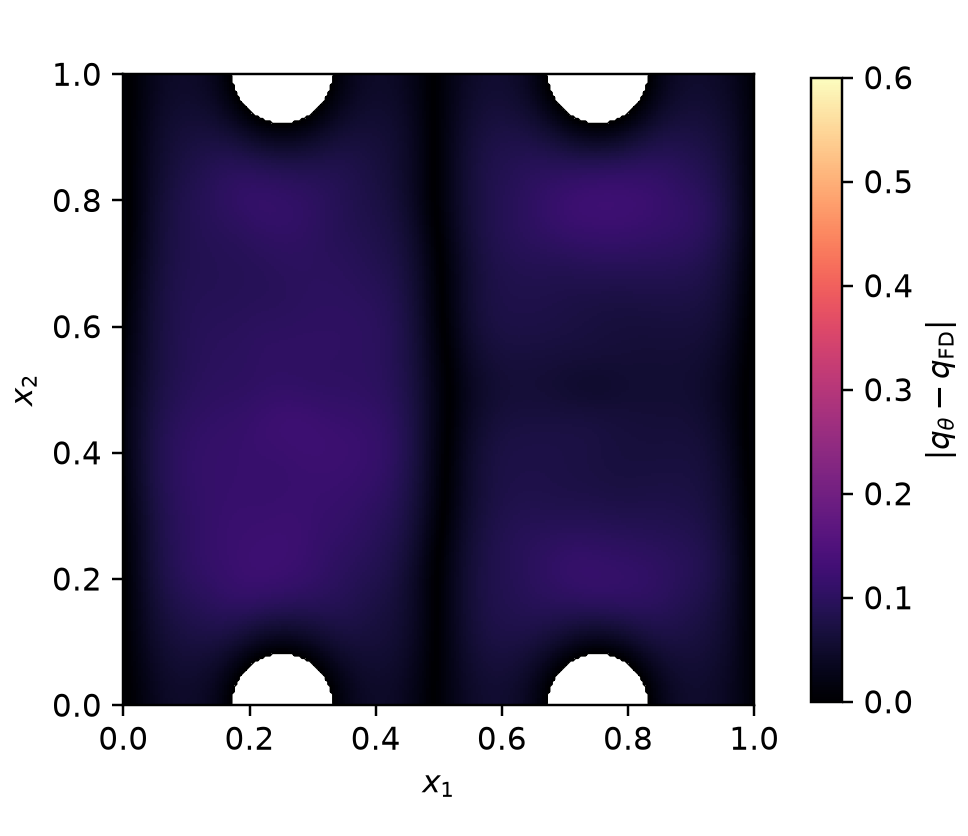}%
			}{%
				\fbox{\parbox[c][0.75\linewidth][c]{\dimexpr\linewidth-2\fboxsep-2\fboxrule\relax}{\centering\small Image pending}}%
			}
			\par\small Over-enforced PINN error
		\end{minipage}\hfill
		\begin{minipage}[t]{0.32\textwidth}
			\centering
			\IfFileExists{./numerical_experiments/committor/periodic_two_well/under_err.png}{%
				\includegraphics[width=\linewidth]{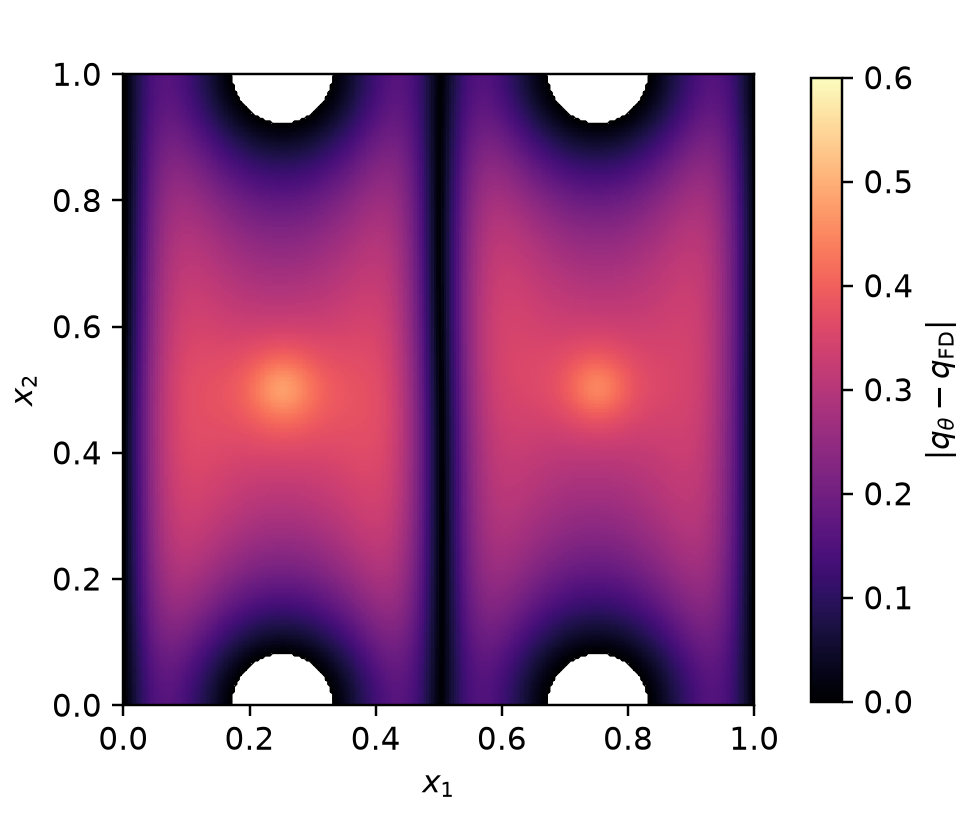}%
			}{%
				\fbox{\parbox[c][0.75\linewidth][c]{\dimexpr\linewidth-2\fboxsep-2\fboxrule\relax}{\centering\small Image pending}}%
			}
			\par\small Under-enforced PINN error
		\end{minipage}
		\caption{Reference committor and pointwise absolute errors for the periodic committor problem. Relative errors are reported in Table~\ref{tab:committor_errors}.}
		\label{fig:committor_methods_abc}
	\end{figure}
	
	Overall, the numerical experiments confirm that the adapted PINN formulation is the most accurate and robust across all test cases, consistently achieving the lowest relative errors and the smallest variability over independent runs. By contrast, the baseline and unadapted variants either yield significantly larger errors or exhibit pronounced run-to-run fluctuations, indicating a lack of stability. Most notably, the over- and under-enforced PINNs frequently fail \emph{catastrophically} to approximate the PDE solution, in agreement with Proposition~\ref{prop:quotient_regularity}, which shows that inappropriate choices of the multiplication function destroy the smoothness of the effective approximation target and thereby invalidate the convergence guarantees to the true solution.
	
	\section{Conclusion and future work}
	
	We have shown that residual-based PINNs for linear elliptic problems with Dirichlet boundary conditions admit $H^2(\Omega)$ a priori error bounds when the smooth distance approximation $\rho$ used in the boundary-enforced architecture satisfies the boundary-adapted condition of Definition~\ref{def:boundary_adapted_fnns}. Such smooth distance approximations can be constructed by a straightforward algorithm \citep{sukumar2022exact}, making the approach well suited for practitioners. For this class of architectures, we have established explicit high-probability $H^2(\Omega)$ a priori error bounds for tanh and $\ReQU$ activations, making explicit the dependence on the number of collocation points and the problem parameters, while rigorously tracking the role of $\rho$ throughout the analysis. Along the way, we have also derived new $W^{2,\infty}(\Omega)$ error bounds for shallow ReQU networks, together with new bounds on the VC dimensions of higher-order derivatives of ReQU and tanh networks. Our numerical experiments on committor and mean exit time problems validate our theory and show clear and consistent improvement in the accuracy of the learned solutions when enforcing boundary conditions with smooth normalized distance functions.
  
	Several directions for future work naturally emerge from our analysis. One such direction is the study of elliptic problems with mixed boundary conditions, including narrow escape problems and problems with absorbing traps \citep{schuss2007narrow,holcman2015stochastic}. Although exact enforcement of mixed boundary conditions can already be built into PINN architectures \citep{sukumar2022exact}, deriving approximation-theoretic guarantees for these constructions remains open. A different but closely related direction concerns singularly perturbed problems with sharp boundary layers, such as small-noise Langevin escape problems, whose solution scales are typically exponential in $\varepsilon^{-1}$ \citep{freidlin2012random,Dupuis2019REbook}. This suggests looking for $\varepsilon$-dependent architectures which adapt to the whole boundary layer beyond satisfaction of Dirichlet conditions. Finally, extending our analysis to time-dependent problems and more general PDEs is a natural direction to consider as well.
	
	%We lastly stress that our error analysis relies on the empirical risk minimization framework, and therefore does not account for the optimization error introduced by training. Addressing that issue would require a separate analysis of PINN trainability, which remains a central challenge in the literature \citep{krishnapriyan2021characterizing, rathore2024challenges} and lies beyond the scope of the present work.
   
   %\section*{Acknowledgement}
	%N. Tepakbong was supported by the Hong Kong PhD Fellowship Scheme, which also funded a research visit to The University of Sydney where part of this work was carried out. The work by J. Fan is partially supported by the Research Grants Council of Hong Kong [Project No. HKBU12302923 and HKBU12303024], and Guangdong and Hong Kong Universities “1+1+1” Joint Research Collaboration Scheme 2025A0505000007. X. Zhou acknowledges support by General Research Funds from the Research Grants Council of the Hong Kong Special Administrative Region, China (Project No. 11308121, 11318522, 11308323, 11304525).

\appendix
\renewcommand{\theHsection}{appendix.\Alph{section}}

\section{Complement on smooth distance approximations}
\label{sec:quotient_regularity}
\subsection{Proof of Proposition~\ref{prop:quotient_regularity}}

To prove Proposition~\ref{prop:quotient_regularity}, we begin by showing that any zero-trace $C^{r+1}(\bar\Omega)$ function can be expressed, near each boundary point, as the product of a $C^r(\bar U)$ function and a $C^{r+1}(\bar U)$ function, for a suitable open set $U\subset\Omega$, where the second factor depends only on the geometry $\Omega$, and vanishes on $\paOm\cap\bar U$ with a nonzero inward normal derivative at that point. Although this result is a straightforward consequence of elementary facts in PDE analysis, we haven't been able to find a reference for it in the literature. We therefore give its proof, using the local description of smooth boundaries, as described e.g. in \citep[Section 6.2]{gilbarg1977elliptic}.

\begin{lemma}
	\label{lemma:factorization}
	Assume that Assumption~\eqref{eqn:assumption1_smooth_boundary} holds, and fix $q_0\in\paOm$. There are an open set $U\subset\Omega$ with $q_0\in\bar U$ and a function $\zeta\in C^{r+1}(\bar U)$, depending only on $\Omega$ and $q_0$, such that $\zeta>0$ in $U$, $\zeta=0$ on $\paOm\cap\bar U$, and $\partial_{\bnu}\zeta(q_0)=1$. Every $u\in C^{r+1}(\bar\Omega)$ vanishing on $\paOm$ can then be written as $u=\zeta f$ in $U$, where $f\in C^r(\bar U)$ and $f(q_0)=\partial_{\bnu}u(q_0)$.
\end{lemma}

\begin{proof}
	Translate and rotate the coordinates so that $q_0=0$ and the inward unit normal $\bnu(q_0)$ is $e_d=(0,\ldots,0,1)$. By Assumption~\eqref{eqn:assumption1_smooth_boundary} and the local description of smooth boundaries in \citep[Section 6.2]{gilbarg1977elliptic}, there is a neighborhood $V$ of zero in which any boundary point can be written as $(x',\varphi(x'))$, where $x'=(x_1,\ldots,x_{d-1})$ and $\varphi$ is $C^{r+1}$. Moreover, $\varphi(0)=0$, $\nabla\varphi(0)=0$, and the interior of $\Omega$ in $V$ consists of the points with $x_d>\varphi(x')$.
	
	Choose a small ball $D$ centered at zero in $\R^{d-1}$ and $\delta>0$ such that $U:=\{(x',x_d):x'\in D,\ 0<x_d-\varphi(x')<\delta\}\subset\Omega$ and $\bar U\subset V$, with $\varphi\in C^{r+1}(\bar D)$. Define $\zeta(x):=x_d-\varphi(x')$, the difference between the last coordinate of $x$ and that of the boundary point with the same $x'$. Then $\zeta\in C^{r+1}(\bar U)$ is positive in $U$ and zero on $\paOm\cap\bar U$. Since $\nabla\varphi(0)=0$, we also have $\partial_{\bnu}\zeta(q_0)=1$.
	
	Now let $u\in C^{r+1}(\bar\Omega)$ vanish on $\paOm$. Then $u(x',\varphi(x'))=0$ for $x'\in\bar D$. Define
	\[
	f(x):=\int_0^1\partial_d u\bigl(x',\varphi(x')+t\zeta(x)\bigr)\,dt,
	\text{ for }x\in\bar U.
	\]
	By the fundamental theorem of calculus on the last coordinate, we get $u(x)=\zeta(x)f(x)$. Note also that the segment in the integral lies in $\bar U$, and its integrand is $C^r$ in $x$, since $\partial_d u$ is $C^r$ up to $\paOm$ and $\varphi$ is $C^{r+1}$. Differentiating under the integral therefore gives $f\in C^r(\bar U)$. At $q_0$, the segment reduces to the point $q_0$, so $f(q_0)=\partial_d u(q_0)=\partial_{\bnu}u(q_0)$, completing the proof.
\end{proof}

Equipped with Lemma \ref{lemma:factorization}, we can now proceed with the proof of Proposition \ref{prop:quotient_regularity}.

\begin{proof}[Proof of Proposition~\ref{prop:quotient_regularity}]
	Under Assumptions \eqref{eqn:assumption1_smooth_boundary}---\eqref{eqn:assumption4_regular_coeffs}, we have $u^\star-g\in C^{r+1}(\bar\Omega)$ with zero boundary values. Likewise, $\rho\in C^{r+1}(\bar\Omega)$ vanishes on $\paOm$. Let $q_0\in\paOm$ and apply Lemma~\ref{lemma:factorization} to both functions in the same neighborhood $U$. This gives $f_1,f_2\in C^r(\bar U)$ such that $u^\star-g=\zeta f_1$ and $\rho=\zeta f_2$. Consequently, $(u^\star-g)/\rho=f_1/f_2$ on $U$. Since $\rho$ does not vanish in $\Omega$, the quotient is already well-defined and $C^{r+1}$ away from $\paOm$. It is thus enough to show that $f_2$ stays positive near $q_0$.
	
	To that end, the last assertion of the lemma and Definition~\ref{def:order_m_normalization} give $f_2(q_0)=\partial_{\bnu}\rho(q_0)>0$. By continuity, $f_2\ge f_2(q_0)/2$ in a smaller neighborhood of $q_0$. Hence $f_1/f_2$ is $C^r$ there, including at the boundary.
	
	The same argument applies at every point of $\paOm$: for each $q\in\paOm$, choose an open neighborhood $V_q\subset\R^d$ of $q$ small enough that the local ratio $f_1/f_2$ is defined and $C^r$ on $\bar\Omega\cap V_q$. On $\Omega\cap V_p\cap V_q$, both ratios equal $(u^\star-g)/\rho$, so their derivatives through order $r$ agree there as well. Every point of $\paOm\cap V_p\cap V_q$ is a limit of points in this interior overlap. By continuity, the two ratios and their derivatives therefore agree also at the boundary. We have thus established that $(u^\star-g)/\rho$, which is $C^{r+1}$ in $\Omega$, extends continuously to $\bar\Omega$ together with all its derivatives through order $r$, as desired.
\end{proof}

\section{Approximation error analysis}
\label{app:approximation_error}

%To estimate the approximation error of our ReQU FCNNs, we will use the ``standard machinery" first introduced in \citep{yarotsky2017error}: we partition the domain into small cubes, on which $\tau$ can be approximated by Taylor polynomials of appropriate order, and we use a partition of unity to extend each local estimate to a global one. We note that a general framework to build approximate partitions of unity and local Taylor polynomials, leading to estimates on the approximation error of Deep Neural Networks in higher order smoothness spaces, has been proposed in \citep{guhring2021approximation}, and successfully applied in \citep{deryck2021approximation} to that end. Our approach is similar, but the main difference is that we make use of the representation power of ReQU networks \citep{he2023expressivity,bo2020better} to construct exact partitions of unity, multiplication operator and local Taylor polynomials, leading to a fully explicit construction of the Neural Networks architecture and weights, with explicit bounds on their magnitude in terms of the problem's parameters.

We will now prove Theorem~\ref{thm:approx_requ} for approximation with shallow ReQU networks. The argument is based off the integral representation of infinitely wide neural networks \citep{barron1993universal}, which was refined to ReLU activation by \citet{bach2017breaking} and later extended to broader activations \citep{siegel2020approximation, siegel2022high}. In particular, our Theorem~\ref{thm:approx_requ} extends \citet{yang2025optimal}'s results to the case of higher order Sobolev smoothness.

\subsection{Preliminaries}

In this subsection we recall the main notions from harmonic analysis on the sphere that will be used in the proof of the approximation rates and we introduce the variation spaces associated with shallow ReQU networks. The main reference we rely on is \citep{dai2013approximation}. Throughout we fix an integer $d\ge 2$. We write $\mathbb S^d\subset\R^{d+1}$ for the unit sphere endowed with its standard surface measure $\tau_d$ normalized by $\tau_d(\mathbb S^d)=1$ and we write $\mathbb B^d\subset\R^d$ for the closed unit ball.

\textbf{Geometry and spectral theory on the sphere}

We start by recalling the spherical harmonic decomposition of $L^2(\mathbb S^d)$ into eigenspaces of the Laplace-Beltrami operator. We refer the reader to \citep{dai2013approximation} and references within for a thorough overview on this topic.

\begin{definition}[Spherical harmonics and spectral decomposition]
	\label{def:spherical_harmonics}
	For $n\in\N$ let $\mathcal H_n$ denote the space of real-valued homogeneous harmonic polynomials of degree $n$ on $\R^{d+1}$ and let $Y_n$ denote the set of their restrictions to the unit sphere. The elements of $Y_n$ are called spherical harmonics of degree $n$. The Laplace-Beltrami operator $\Delta$ on $\mathbb S^d$ satisfies
	\[
	\Delta Y = -\lambda_n Y, \text{ where }
	\lambda_n := n(n+d-1) \text{ and }
	Y\in Y_n,
	\]
	so $Y_n$ is the eigenspace associated with $-\lambda_n$.
	
	The spaces $(Y_n)_{n\ge 0}$ are pairwise orthogonal in $L^2(\mathbb S^d,\tau_d)$ and they span $L^2(\mathbb S^d)$. We will denote by $\mathcal P_n:L^2(\mathbb S^d)\to Y_n$ the corresponding orthogonal projector. The dimension of $Y_n$ is equal to $1$ if $n=0$ and equal to
	\[
	N(d,n)
	= \frac{2n+d-1}{n} \binom{n+d-2}{d-1},
	\]
	whenever $n\ge 1$.
\end{definition}

We now formally define the space of $C^m$-smooth functions on the sphere $\mathbb S^d$.

\begin{definition}
	\label{def:Cm_sphere}
	For any sufficiently regular $u:\mathbb S^d\to \R$ we denote by $\|u\|_{C^m(\mathbb S^d)}$ the $C^m$-norm of $u$ defined via a finite smooth atlas. More concretely, fix a finite smooth atlas $(U_\alpha,\varphi_\alpha)_{\alpha=1}^L$ on $\mathbb S^d$ and a smooth partition of unity $(\chi_\alpha)_{\alpha=1}^L$ subordinate to this atlas.
	For $u\in C^m(\mathbb S^d)$ we set
	\[
	\|u\|_{C^m(\mathbb S^d)}
	:=
	\max_{1\le \alpha\le L}
	\max_{|\beta|\le m}
	\bigl\|
	\partial^\beta\bigl((\chi_\alpha u)\circ\varphi_\alpha^{-1}\bigr)
	\bigr\|_{L^\infty(\varphi_\alpha(U_\alpha))}.
	\]
	Different choices of atlas and partition of unity lead to equivalent norms. For the above choice there exists a constant $c_m\ge 1$, depending only on $d$ and $m$, such that
	\[
	c_m^{-1}
	\sum_{k=0}^m \|\nabla^k u\|_{L^\infty(\mathbb S^d)}
	\le
	\|u\|_{C^m(\mathbb S^d)}
	\le
	c_m
	\sum_{k=0}^m \|\nabla^k u\|_{L^\infty(\mathbb S^d)},
	\]
	where $\nabla$ is the Levi–Civita connection associated with the round metric \citep{taylor1996partial}.
\end{definition}

\textbf{Convolution and spectral multipliers on $\mathbb S^d$}

We next recall the addition formula, the definition of spherical convolution and the Funk-Hecke formula.

\begin{lemma}[Addition formula and Gegenbauer polynomials]
	\label{lemma:addition_formula}
	For each $n\in\N_0$ fix an orthonormal basis $\{Y_{n,\ell}\}_{\ell=1}^{N(d,n)}$ of $Y_n$ in $L^2(\mathbb S^d,\tau_d)$. There exists a univariate polynomial $P_n:[-1,1]\to\R$ (the Gegenbauer polynomial of degree $n$) with $P_n(1)=1$ such that
	\[
	\sum_{\ell=1}^{N(d,n)} Y_{n,\ell}(u)\,Y_{n,\ell}(v)
	=
	N(d,n)\,P_n(u\cdot v), \text{ for all }
	u,v\in\mathbb S^d.
	\]
	In particular $|P_n(t)|\le 1$ for all $t\in[-1,1]$ and $n\in\N_0$ \citep[Theorem~1.2.6]{dai2013approximation}.
\end{lemma}

The addition formula shows that the projector $\mathcal P_n$ is an integral operator whose kernel has the form $K_n(u,v)=N(d,n)P_n(u\cdot v)$, that is, it depends only on the inner product $u\cdot v$. This leads naturally to the following notion of convolution on $\mathbb S^d$.

\begin{definition}
	\label{def:convolution_sphere}
	Let $c_d>0$ be the normalizing constant
	\[
	c_d^{-1}
	=
	\int_{-1}^1 (1-t^2)^{(d-2)/2}\,dt
	\]
	and define the probability measure $\varrho$ on $[-1,1]$ with density $c_d(1-t^2)^{(d-2)/2}$. For $f\in L^1(\mathbb S^d,\tau_d)$ and a measurable function $g:[-1,1]\to\R$ with $g\in L^1([-1,1],\varrho)$ we define the convolution
	\[
	(f*g)(u)
	:=
	\int_{\mathbb S^d} f(v)\,g(u\cdot v)\,d\tau_d(v),
	\text{ for }
	u\in\mathbb S^d.
	\]
\end{definition}

We next recall a corollary of the classical Funk-Hecke formula, which describes how convolution with a kernel of the form $g(u\cdot v)$ acts on spherical harmonics \citep[Theorem~2.1.3]{dai2013approximation}.

\begin{lemma}
	\label{lemma:funk-hecke}
	Let $g\in L^1([-1,1],\varrho)$ and define coefficients
	\[
	\hat g(n)
	:=
	\frac{\omega_{d-1}}{\omega_d}
	\int_{-1}^1 g(t)\,P_n(t)\,(1-t^2)^{(d-2)/2}\,dt,
	\text{ for all } n\in\N_0,
	\]
	where $P_n$ is the Gegenbauer polynomial from Lemma~\ref{lemma:addition_formula} and $\omega_d$ denotes the surface area of $\mathbb S^d$. Then for every $f\in L^2(\mathbb S^d,\tau_d)$ and every $n\in\N_0$,
	\[
	\mathcal P_n(f*g)
	=
	\hat g(n)\,\mathcal P_n f.
	\]
\end{lemma}

\textbf{Moduli of smoothness on $\mathbb S^d$.}

We now recall the notion of spherical translations and moduli of smoothness, which will allow us to quantify smoothness of functions on $\mathbb S^d$ in a way that interacts well with convolution.

\begin{definition}[Spherical translations and moduli of smoothness]
	\label{def:modulus_sphere}
	For $u\in\mathbb S^d$ let
	\[
	\mathbb S_u^\perp
	:=
	\{\, v\in\mathbb S^d : u\cdot v = 0 \,\}
	\cong \mathbb S^{d-1}.
	\]
	For $\theta\in[0,\pi]$ and $f\in L^1(\mathbb S^d,\tau_d)$ the spherical translation (or spherical mean) operator is defined for all $u\in\mathbb{S}^d$ by
	\[
	T_\theta f(u)
	:=
	\int_{\mathbb S_u^\perp} f\bigl(u\cos\theta + v\sin\theta\bigr)\,d\tau_{d-1}(v),
	\]
	where $\tau_{d-1}$ is the normalized surface measure on $\mathbb S_u^\perp$. For all $n\in\N_0$, it holds that
	\[
	\mathcal P_n(T_\theta f)
	=
	P_n(\cos\theta)\,\mathcal P_n f.
	\]
	
	Let $\alpha>0$ and $0<\theta<\pi$. We define the fractional difference operator $\Delta_\theta^\alpha$ by spectral calculus through
	\[
	\mathcal P_n(\Delta_\theta^\alpha f)
	:=
	\bigl(1-P_n(\cos\theta)\bigr)^{\alpha/2}\,\mathcal P_n f,
	\text{ for all }
	n\in\N_0.
	\]
	For $1\le p<\infty$ and $f\in L^p(\mathbb S^d)$, or for $p=\infty$ and $f\in C(\mathbb S^d)$, the $\alpha$-th order modulus of smoothness of $f$ is defined as
	\[
	\omega_\alpha(f,t)_p
	:=
	\sup_{0<\theta\le t} \|\Delta_\theta^\alpha f\|_{L^p(\mathbb S^d)},
	\text{ where } 0<t<\pi.
	\]
\end{definition}

On the Euclidean side we will work with Hölder spaces on the unit ball.

\begin{definition}[Hölder spaces on $\R^d$ and on $\mathbb B^d$]
	\label{def:holder_spaces}
	Let $m\in\N_0$ and $\kappa\in(0,1]$.
	The Hölder space $C^{m,\kappa}(\R^d)$ consists of all functions $f\in C^m(\R^d)$ such that the seminorm
	\[
	|f|_{C^{m,\kappa}(\R^d)}
	:=
	\max_{|\gamma|=m}
	\sup_{x\neq y}
	\frac{|\partial^\gamma f(x)-\partial^\gamma f(y)|}{\|x-y\|_2^\kappa}
	\]
	is finite. The norm on $C^{m,\kappa}(\R^d)$ is
	\[
	\|f\|_{C^{m,\kappa}(\R^d)}
	:=
	\max\Bigl\{
	\max_{|\gamma|\le m} \|\partial^\gamma f\|_{L^\infty(\R^d)},
	\,|f|_{C^{m,\kappa}(\R^d)}
	\Bigr\}.
	\]
	We define $C^{m,\kappa}(\mathbb B^d)$ as the space of restrictions to $\mathbb B^d$ of functions in $C^{m,\kappa}(\R^d)$, equipped with the norm
	\[
	\|f\|_{C^{m,\kappa}(\mathbb B^d)}
	:=
	\inf\bigl\{
	\|g\|_{C^{m,\kappa}(\R^d)} :
	g\in C^{m,\kappa}(\R^d),\ g=f\text{ on }\mathbb B^d
	\bigr\}.
	\]
	For $\alpha>0$ we set $a:=\lceil \alpha\rceil-1$ and $\kappa:=\alpha-a\in(0,1]$.
	We denote the unit ball of $C^{a,\kappa}(\mathbb B^d)$ by
	\[
	\mathcal H^\alpha
	:=
	\bigl\{
	f\in C^{a,\kappa}(\mathbb B^d)
	:
	\|f\|_{C^{a,\kappa}(\mathbb B^d)}\le 1
	\bigr\}.
	\]
\end{definition}

\textbf{Variation spaces for shallow ReQU networks.}

For the remaining of this section, we denote the ReQU activation as $\sigma :t \mapsto \max\{0,t\}^2$. We now introduce the associated variation spaces on $\mathbb B^d$ and $\mathbb S^d$ respectively. These spaces may be seen as infinite-width limits of one-hidden-layer ReQU networks with an $\ell_1$ constraint on their weights \citep{bach2017breaking}.

\begin{definition}[ReQU variation space in the ball]
	\label{def:ReQU_activation_variation}
	Let $\mu$ be a finite signed Radon measure on $\mathbb S^d$. For $x\in\mathbb B^d$ we consider the integral representation
	\[
	f(x)
	=
	\int_{\mathbb S^d}
	\sigma\bigl((x^\top,1)\cdot v\bigr)\,d\mu(v).
	\]
	We define the variation norm of $f$ by
	\[
	\gamma(f)
	:=
	\inf
	\Bigl\{
	\|\mu\|_{\mathrm{TV}}:
	f(x)=\textstyle\int_{\mathbb S^d} \sigma\bigl((x^\top,1)\cdot v\bigr)\,d\mu(v)
	\text{ for all }x\in\mathbb B^d
	\Bigr\},
	\]
	where the infimum runs over all finite signed Radon measures $\mu$ on $\mathbb S^d$ and $\|\mu\|_{\mathrm{TV}}$ denotes the total variation of $\mu$. For $M>0$ we define the ReQU variation space on $\mathbb B^d$ by
	\[
	\mathcal V_\sigma(M)
	:=
	\{\, f:\mathbb B^d\to \R : \gamma(f)\le M \,\}.
	\]
\end{definition}

We also introduce the corresponding spherical ReQU variation space, which will be the main ingredient in our error analysis.

\begin{definition}[Spherical ReQU variation space]
	\label{def:spherical_variation_space}
	Let $\mu$ be a finite signed Radon measure on $\mathbb S^d$. For $u\in\mathbb S^d$ we set
	\[
	g(u)
	=
	\int_{\mathbb S^d} \sigma(u\cdot v)\,d\mu(v).
	\]
	The variation norm of $g$ is
	\[
	\gamma(g)
	:=
	\inf
	\Bigl\{
	\|\mu\|_{\mathrm{TV}}:
	g(u)=\textstyle\int_{\mathbb S^d} \sigma(u\cdot v)\,d\mu(v)
	\text{ for all }u\in\mathbb S^d
	\Bigr\}.
	\]
	For $M>0$ we define the ReQU variation space on $\mathbb S^d$ by
	\[
	\mathcal G_\sigma(M)
	:=
	\{\, g:\mathbb S^d\to\R : \gamma(g)\le M \,\}.
	\]
	If $\mu$ is absolutely continuous with respect to $\tau_d$ with density $\phi\in L^1(\mathbb S^d,\tau_d)$ then
	\[
	g(u)
	=
	\int_{\mathbb S^d} \sigma(u\cdot v)\,\phi(v)\,d\tau_d(v)
	=
	(\phi*\sigma)(u),
	\]
	where the convolution is that of Definition~\ref{def:convolution_sphere}.
\end{definition}

Several formulations of variation spaces for shallow networks appear in the literature \citep{savarese2019infinite, ongie2020function, siegel2020approximation}. For our purposes, the representation above is the most convenient.

\subsection{Approximation error of ReQU variation space}

Having defined the necessary objects, we are now ready to carry out our error analysis for the variation space associated with shallow ReQU networks. We begin by introducing a transfer operator that allows us to move between functions on the unit sphere and functions on the unit ball. Throughout this subsection we keep the dimension $d\ge 2$ fixed and we work with the sphere $\mathbb S^d\subset\R^{d+1}$ and the ball $\mathbb B^d\subset\R^d$.

\textbf{Transfer between the sphere and the ball}

\begin{definition}[Transfer operator from the sphere to the ball]
	\label{def:transfer_operator_T}
	Let $\mathcal C_+:=\{u\in\mathbb S^d:u_{d+1}\ge1/\sqrt2\}$ be the upper spherical cap. For $x\in\mathbb B^d$, define $\Phi(x):=(x^\top,1)^\top/\sqrt{1+\|x\|_2^2}\in\mathcal C_+$ and $\psi(x):=1+\|x\|_2^2$. For a bounded function $g:\mathbb S^d\to\R$, we define the transfer operator $\mathcal T$ by
	\[
	(\mathcal T g)(x):=\psi(x)g(\Phi(x)),
	\text{ for }x\in\mathbb B^d.
	\]
\end{definition}

Observe that $\Phi$ is a smooth embedding of $\mathbb B^d$ into $\mathcal C_+\subset\mathbb S^d$, while the factor $\psi$ compensates for the homogeneity of the ReQU activation. As a direct consequence of the product and chain rules, $\mathcal T$ preserves smoothness:

\begin{lemma}
	\label{lemma:transfer_sobolev_T}
	Let $r\in\N_0$. There exists a constant $C_{r,d}>0$ such that for every $g\in C^r(\mathbb S^d)$, $ \|\mathcal T g\|_{W^{r,\infty}(\mathbb B^d)} \le C_{r,d}\,\|g\|_{C^r(\mathbb S^d)}$.
\end{lemma}
\begin{proof}
	Since $\psi$ and $\Phi$ extend smoothly to the closed ball, their derivatives through order $r$ are bounded. The product and chain rules therefore give $\|\psi\,g\circ\Phi\|_{W^{r,\infty}(\mathbb B^d)}\le C_{r,d}\|g\|_{C^r(\mathcal C_+)}$. The cap $\mathcal C_+$ lies in a single smooth chart of the sphere, so its local norm is controlled by $\|g\|_{C^r(\mathbb S^d)}$, proving the estimate.
\end{proof}

We now introduce two properties of the operator $\mathcal T$ which will be useful for our later analysis. The argument below largely follows that of \citep[Proposition 3.3]{yang2025optimal}.

\begin{proposition}
	\label{prop:transfer_variation_ball_sphere}
	Recall the variation spaces $\mathcal{V}_\sigma$ and $\mathcal{G}_\sigma$ respectively defined in Definition~\ref{def:ReQU_activation_variation} and \ref{def:spherical_variation_space}. The following hold:
	\begin{enumerate}[label=\textbf{(\roman*)},leftmargin=1.5em]
		\item \ \  For every $M>0$ and every $g\in\mathcal G_\sigma(M)$ the function $\mathcal T g$ belongs to $\mathcal V_\sigma(M)$ and satisfies {$\gamma(\mathcal T g)\le \gamma(g)\le M$}.
		\item \ \ Let $\alpha>0$ and set $a:=\lceil \alpha\rceil-1$ and $\kappa:=\alpha-a\in(0,1]$. There exists a constant $C_\alpha>0$ such that for every $h\in\mathcal H^\alpha$ there exists an odd function $\tilde h\in C^{a,\kappa}(\mathbb S^d)$ such that $\mathcal T \tilde h = h$ on $\mathbb B^d$, and $\|\tilde h\|_{C^{a,\kappa}(\mathbb S^d)} \le C_\alpha$.
	\end{enumerate}
\end{proposition}

\begin{proof}
	\textbf{(i):} Let $g\in\mathcal G_\sigma(M)$. By Definition~\ref{def:spherical_variation_space} there exists a finite signed Radon measure $\mu$ on $\mathbb S^d$ with $\|\mu\|_{\mathrm{TV}}\le M$ such that $g(u) = \int_{\mathbb S^d} \sigma(u\cdot v)\,d\mu(v)$ for all
	$u\in\mathbb S^d$. Now, let $x\in\mathbb B^d$. By definition, we have
	\[
	(\mathcal T g)(x)
	=
	\psi(x)\,g(\Phi(x))
	=
	\psi(x)
	\int_{\mathbb S^d} \sigma\bigl(\Phi(x)\cdot v\bigr)\,d\mu(v).
	\]
	Using the fact that $\psi(x)=1+\|x\|_2^2$ and homogeneity of $\sigma$, we get by direct computation that
	\[
	(\mathcal T g)(x)
	=
	\int_{\mathbb S^d} \sigma\bigl((x^\top,1)\cdot v\bigr)\,d\mu(v),
	\text{ for all } x\in\mathbb B^d.
	\]
	Thus $\mathcal T g$ admits a ReQU variation representation on $\mathbb B^d$ with the same measure $\mu$. By Definition~\ref{def:ReQU_activation_variation}, we thus conclude that $\gamma(\mathcal T g) \le	\|\mu\|_{\mathrm{TV}} \le M$, which proves claim (i).
	
	\textbf{(ii):} Let $h\in\mathcal H^\alpha$ and set $a:=\lceil\alpha\rceil-1$, $\kappa:=\alpha-a$. As Definition~\ref{def:holder_spaces} defines the Hölder norm through global extensions, and $\|h\|_{C^{a,\kappa}(\mathbb B^d)}\le1$, we may choose an extension $\hat h\in C^{a,\kappa}(\R^d)$ with norm at most two. We first define a function on the upper cap $\mathcal{C}_+$ by
	\[
	\tilde h_+(u):=u_{d+1}^2\hat h(u'/u_{d+1}),
	\text{ where }u'=(u_1,\ldots,u_d).
	\]
	For $x\in\mathbb B^d$, the definition of $\Phi$ gives $\Phi(x)'/\Phi(x)_{d+1}=x$ and $\Phi(x)_{d+1}^2=\psi(x)^{-1}$. Consequently, $(\mathcal T\tilde h_+)(x)=\psi(x)\Phi(x)_{d+1}^2\hat h(x)=h(x)$.
	
	We now extend $\tilde h_+$ from the upper cap to the whole sphere. Since $\hat h$ is defined on $\R^d$, the same formula is defined wherever $u_{d+1}>1/2$. Choose a smooth cutoff $\chi:\R\to[0,1]$ equal to zero for $t\le1/2$ and one for $t\ge1/\sqrt2$, and define
	\[
	H(u):=\begin{cases}
		\chi(u_{d+1})u_{d+1}^2\hat h(u'/u_{d+1}),&u_{d+1}>1/2,\\
		0,&u_{d+1}\le1/2.
	\end{cases}
	\]
	On the support of $\chi$, the denominator is bounded away from zero. The product and chain rules therefore give $H\in C^{a,\kappa}(\mathbb S^d)$ with $\|H\|_{C^{a,\kappa}}\le C_{\alpha,d}\|\hat h\|_{C^{a,\kappa}}$. Define $\tilde h(u):=H(u)-H(-u)$. This function is odd and has uniformly bounded $C^{a,\kappa}$ norm. On the upper cap $\mathcal{C}_+$, $H=\tilde h_+$ and $H(-u)=0$, so $\tilde h=\tilde h_+$ there. Hence $\mathcal T\tilde h=h$, which proves part (ii).
	\end{proof}

\textbf{Approximation by filtered spherical polynomials.}

We now introduce the polynomial kernels used to approximate functions on the sphere.

\begin{definition}[Spectral cutoff kernels]
	\label{def:spectral_cutoff_kernel}
	Let $\eta\in C^\infty([0,\infty))$ satisfy $0\le\eta\le1$, $\eta(t)=1$ for $0\le t\le 1$ and $\eta(t)=0$ for $t\ge 2$. For each integer $m\ge 1$ we define the kernel $L_m:[-1,1]\to\R$ by
	\[
	L_m(t)
	:=
	\sum_{n=0}^\infty \eta\!\left(\frac{n}{m}\right) N(d,n)\,P_n(t),
	\text{ for all } t\in[-1,1],
	\]
	where $N(d,n)$ and $P_n$ are as in Definition~\ref{def:spherical_harmonics} and Lemma~\ref{lemma:addition_formula}. Since $\eta(n/m)=0$ for $n\ge 2m$ the sum is finite. For a continuous function $f$ on $\mathbb S^d$, if we set $g_m:=f*L_m$, then $g_m$ is a spherical polynomial of degree at most $2m-1$ and depends linearly on $f$.
\end{definition}

We have the following quantitative error bounds for approximation with spherical polynomials.

\begin{lemma}
	\label{lemma:laplacian_smooth_approx}
	Let $\alpha>0$, $h\in\mathcal H^\alpha$, and let $\tilde h$ be the associated extension to the sphere from Proposition~\ref{prop:transfer_variation_ball_sphere}(ii). For every integer $m\ge1$, set $g_m:=\tilde h*L_m$. There exists $C_\alpha>0$, independent of $h,m$, such that
	\[
	\|\tilde h-g_m\|_{L^\infty(\mathbb S^d)}\le C_\alpha m^{-\alpha}.
	\]
\end{lemma}

\begin{proof}
	By Proposition~\ref{prop:transfer_variation_ball_sphere}(ii), the $C^{a,\kappa}(\mathbb S^d)$ norm of $\tilde h$ is bounded independently of $h$, where $a=\lceil\alpha\rceil-1$ and $\kappa=\alpha-a$. We then deduce from \citep[Proposition 3.3]{yang2025optimal} and \citep[Theorems 10.3.2 and 10.4.1]{dai2013approximation} that
	\[
	\|\tilde h-g_m\|_{L^\infty(\mathbb S^d)}
	\le C_{\alpha,d}m^{-\alpha}
	\]
	for all sufficiently large $m$. By \citep[Theorem 2.6.3, property (2)]{dai2013approximation}, the same bound holds for all $m\ge1$ after possibly increasing the constant, completing the proof.
\end{proof}

We next upgrade the $L^\infty$ approximation of $\tilde h$ by $g_m$ to $C^q$-approximation on the sphere.

\begin{proposition}
	\label{prop:cq_smooth_approx}
	Let $\alpha>0$ and set $a:=\lceil \alpha\rceil-1$. Let $h\in\mathcal H^\alpha$ and let $\tilde h$ be the associated extension from Proposition~\ref{prop:transfer_variation_ball_sphere}(ii). For each integer $m\ge 1$, define $g_m:=\tilde h*L_m$ as in Definition~\ref{def:spectral_cutoff_kernel}. Then, for every integer $q$ with $0\le q\le a$ there exists a constant $C_{\alpha,q}>0$, independent of $h$ and $m$, such that
	\[
	\|\tilde h-g_m\|_{C^q(\mathbb S^d)}
	\le
	C_{\alpha,q}\,m^{-(\alpha-q)}.
	\]
\end{proposition}

\begin{proof}
	The case $q=0$ is exactly Lemma~\ref{lemma:laplacian_smooth_approx}, so now assume that $q\ge1$. By that lemma, the sequence $(g_{2^km})_{k\ge 0}$ converges uniformly to $\tilde h$, hence
	\[
	\tilde h-g_m=\sum_{k\ge0}(g_{2^{k+1}m}-g_{2^km})
	\text{ in }C^0(\mathbb S^d).
	\]
	Each summand is a spherical polynomial of degree at most $2^{k+2}m-1$, and its $L^\infty$ norm is bounded as
	\[
	\|g_{2^{k+1}m}-g_{2^km}\|_{L^\infty(\mathbb S^d)} \le \|g_{2^{k+1}m}-\tilde{h}\|_{L^\infty(\mathbb S^d)} + \|\tilde h-g_{2^km}\|_{L^\infty(\mathbb S^d)} \le 2C_\alpha(2^km)^{-\alpha},
	\] 
	
	where, again, we used Lemma~\ref{lemma:laplacian_smooth_approx}. Bernstein inequality \citep[Lemma 4.2.4]{dai2013approximation} then gives
	\begin{align*}
		\|g_{2^{k+1}m}-g_{2^km}\|_{C^q(\mathbb S^d)}
		\le B_{q,d}(2^{k+2}m)^q\|g_{2^{k+1}m}-g_{2^km}\|_{L^\infty(\mathbb S^d)}\le C_{\alpha,q,d}(2^km)^{q-\alpha}.
	\end{align*}
	Since $q\le a<\alpha$, these bounds are summable, so the series converges in $C^q$ as well. Its limit equals $\tilde h-g_m$ because convergence in $C^q$ implies uniform convergence. Summing the bounds therefore gives
	\[
	\|\tilde h-g_m\|_{C^q(\mathbb S^d)}
	\le C_{\alpha,q,d}m^{q-\alpha}\sum_{k\ge0}2^{k(q-\alpha)}
	\le C_{\alpha,q,d}'m^{-(\alpha-q)}.
	\]
\end{proof}

It now remains to estimate the variation norm of the filtered polynomials $g_m$. We follow the spherical harmonic representation used in the proof of \citep[Theorem 2.1]{yang2025optimal}. When $\alpha=(d+5)/2$ is an integer, we manage to bound derivatives of order $\alpha$ directly in $L^2$, even when $\alpha$ is odd, thus allowing us to remove the logarithmic loss discussed in \citep[Remark 3.4]{yang2025optimal}.

\begin{proposition}
	\label{prop:variation_of_gm}
	Let $\alpha>0$, set $\alpha_c:=\frac{d+5}{2}$, and let $h\in\mathcal H^\alpha$ with associated extension $\tilde h$ to the sphere.
	For $m\ge1$ define $g_m:=\tilde h*L_m$.
	There exist constants $C_1,C_2>0$ (depending only on $d$ and $\alpha$) such that:
	\begin{enumerate}[label=(\roman*),leftmargin=1.5em]
		\item If $\alpha>\alpha_c$ or if $\alpha=\alpha_c$ is an integer, then
		\[
		\sup_{m\ge1}\gamma(g_m)\le C_1.
		\]
		\item If $\alpha=\alpha_c$ is not an integer, then for all $m\ge2$,
		\[
		\gamma(g_m)\le C_2\sqrt{\log m}.
		\]
		\item If $\alpha<\alpha_c$, then for all $m\ge1$,
		\[
		\gamma(g_m)\le C_2\,m^{\alpha_c-\alpha}.
		\]
	\end{enumerate}
\end{proposition}

\begin{proof}
	First note that by Proposition~\ref{prop:transfer_variation_ball_sphere}(ii) we may assume that $\tilde h$ is odd, meaning $\tilde h(-u)=-\tilde h(u)$ for all $u\in\mathbb S^d$. Since $g_m$ is obtained from $\tilde h$ by a spectral multiplier, $g_m$ is odd as well. Therefore $\mathcal P_n g_m=0$ for all even $n$. Now, for each $m\ge 1$, define a spherical polynomial $\phi_m$ by
	\[
	\mathcal P_n \phi_m
	:=
	\begin{cases}
		\hat\sigma(n)^{-1}\,\mathcal P_n g_m, & \text{ if } n \text{ is odd and } 1\le n\le 2m-1,\\
		0, & \text{ otherwise.}
	\end{cases}
	\]
	Since $g_m$ has degree at most $2m-1$, the same is true for $\phi_m$. By Lemma~\ref{lemma:funk-hecke}, convolution with $\sigma$ acts on degree-$n$ harmonics by multiplication with $\hat\sigma(n)$, hence
	\[
	\mathcal P_n(\phi_m*\sigma)=\hat\sigma(n)\mathcal P_n\phi_m=\mathcal P_n g_m
	\text{ for all } n\in\N_0.
	\]
	Therefore, we have that $g_m=\phi_m*\sigma$ for all $m\ge 1$. In particular, $g_m$ admits a representing measure that is absolutely continuous with respect to $\tau_d$, with density $\phi_m$. Thus
	\[
	\gamma(g_m)\le \|\phi_m\|_{L^1(\mathbb S^d)}\le \|\phi_m\|_{L^2(\mathbb S^d)}.
	\]
	By orthogonality of the spherical harmonic decomposition,
	\[
	\|\phi_m\|_{L^2(\mathbb S^d)}^2
	=
	\sum_{n=0}^{2m-1} \|\mathcal P_n\phi_m\|_{L^2(\mathbb S^d)}^2
	=
	\sum_{\substack{1\le n\le2m-1\\n\text{ odd}}} |\hat\sigma(n)|^{-2}\,\|\mathcal P_n g_m\|_{L^2(\mathbb S^d)}^2.
	\]
	
	On the other hand, the harmonic coefficients $\hat\sigma(n)$ of the ReQU activation vanish for all even $n\ge 4$ and otherwise satisfy $|\hat\sigma(n)|\asymp n^{-(d+5)/2}$ on the non-vanishing odd indices \citep{bach2017breaking, yang2025optimal}. There therefore exists a constant $C_d>0$ such that
	\[
	\gamma(g_m)^2
	\le
	C_d
	\sum_{n=0}^{2m-1}
	(1+n)^{d+5}\,\|\mathcal P_n g_m\|_{L^2(\mathbb S^d)}^2.
	\]
	Furthermore, note that since $\mathcal P_n g_m=\eta\!\left(\frac{n}{m}\right)\mathcal P_n\tilde h$ and $|\eta|\le 1$, we have $\|\mathcal P_n g_m\|_{L^2(\mathbb S^d)} \le \|\mathcal P_n\tilde h\|_{L^2(\mathbb S^d)}$. Therefore we have the variation norm estimate
	\[
	\gamma(g_m)^2
	\le
	C_d
	\sum_{n=0}^{2m-1}
	(1+n)^{d+5}\,\|\mathcal P_n \tilde h\|_{L^2(\mathbb S^d)}^2.
	\]
	
	We now split cases depending on the value of $\alpha$. Suppose first that $\alpha=\alpha_c=r$ is an integer. Proposition~\ref{prop:transfer_variation_ball_sphere}(ii) gives a uniform $C^{r-1,1}(\mathbb S^d)$ bound on $\tilde h$. If $r=2k$ or $r=2k+1$, Parseval's identity and integration by parts on the sphere give, respectively,
	\[
	\sum_{n\ge1}\lambda_n^r\|\mathcal P_n\tilde h\|_{L^2(\mathbb S^d)}^2
	=\begin{cases}
		\|(-\Delta)^k\tilde h\|_{L^2(\mathbb S^d)}^2,&r=2k,\\
		\|\nabla(-\Delta)^k\tilde h\|_{L^2(\mathbb S^d)}^2,&r=2k+1.
	\end{cases}
	\]
	Indeed, $-\Delta$ multiplies the degree-$n$ harmonic component by $\lambda_n=n(n+d-1)$, and $\|\nabla w\|_{L^2(\mathbb S^d)}^2=\langle w,-\Delta w\rangle_{L^2(\mathbb S^d)}$. The identities extend from smooth functions to $H^r(\mathbb S^d)$ by approximation \citep{taylor1996partial}. The derivatives on the right have order at most $r$, so their norms are bounded by the $C^{r-1,1}(\mathbb S^d)$ norm of $\tilde h$. Since $\tilde h$ is odd, its degree-zero component vanishes. Also, $\lambda_n\asymp(1+n)^2$ for $n\ge1$ and $2r=d+5$. The variation estimate above therefore gives $\gamma(g_m)^2\le C_{d,r}$, independently of $m$. This proves part (i) at the integer critical index.

	We now treat the remaining cases by a dyadic decomposition: let $K_m\in\N_0$ be such that $2^{K_m}\le 2m < 2^{K_m+1}$. Grouping the sum into dyadic shells gives the variation norm bound
	\[
	\gamma(g_m)^2
	\le
	C_d
	\sum_{k=0}^{K_m}
	2^{(d+5)k}
	\sum_{n=2^k}^{2^{k+1}-1}
	\|\mathcal P_n \tilde h\|_{L^2(\mathbb S^d)}^2.
	\]
	
	We estimate the shell $k=0$ directly. Since this shell consists only of $n=1$, Proposition~\ref{prop:transfer_variation_ball_sphere}(ii) gives a uniform $L^\infty$ bound on $\tilde h$, so $\|\tilde h\|_{L^2(\mathbb S^d)}\le \|\tilde h\|_{L^\infty(\mathbb S^d)}\le C_\alpha$, and hence $\sum_{n=0}^{1}	\|\mathcal P_n \tilde h\|_{L^2(\mathbb S^d)}^2 \le \|\tilde h\|_{L^2(\mathbb S^d)}^2 \le C_\alpha^2$.
	
	Now fix $k\ge1$. The polynomial $g_{2^{k-1}}$ has degree at most $2^k-1$, so $\mathcal P_n g_{2^{k-1}}=0$ for $n\ge2^k$. Hence the degree-$n$ components of $\tilde h-g_{2^{k-1}}$ and $\tilde h$ agree for those $n$. Orthogonality and Parseval's identity then give
	\[
	\sum_{n=2^k}^{2^{k+1}-1}\|\mathcal P_n\tilde h\|_{L^2(\mathbb S^d)}^2
	\le\|\tilde h-g_{2^{k-1}}\|_{L^2(\mathbb S^d)}^2
	\le C_{\alpha,d}2^{-2\alpha(k-1)}
	\le C'_{\alpha,d}2^{-2\alpha k}.
	\]
	The second inequality follows from Lemma~\ref{lemma:laplacian_smooth_approx} and the fact that $\|w\|_{L^2(\mathbb S^d)}\le\|w\|_{L^\infty(\mathbb S^d)}$, since $\tau_d(\mathbb S^d)=1$.
	
	Combining the shell bounds, we get
	\[
	\gamma(g_m)^2
	\le
	C_{\alpha,d}''
	\left(
	1+
	\sum_{k=1}^{K_m} 2^{(d+5-2\alpha)k}
	\right).
	\]
	Recalling that $d+5=2\alpha_c$, this becomes $ \gamma(g_m)^2 \le C_{\alpha,d}'' \left( 1+ \sum_{k=1}^{K_m} 2^{2(\alpha_c-\alpha)k}\right)$. If $\alpha>\alpha_c$, the geometric series is uniformly bounded in $m$, which proves part (i). If $\alpha=\alpha_c$ and $\alpha_c\notin\N$, then $
	\gamma(g_m)^2
	\le
	C_{\alpha,d}'''(1+K_m)
	\le
	C_{\alpha,d}^{(4)}(1+\log m)$, so
	\[
	\gamma(g_m)\le C_{\alpha,d}^{(5)}\sqrt{\log m}
	\text{ for all } m\ge 2,
	\]
	which proves part (ii). If $\alpha<\alpha_c$, then $\sum_{k=1}^{K_m} 2^{2(\alpha_c-\alpha)k} \le C_{\alpha,d}\,2^{2(\alpha_c-\alpha)K_m} \le C_{\alpha,d}'\,m^{2(\alpha_c-\alpha)}$,	and therefore
	\[
	\gamma(g_m)^2
	\le
	C_{\alpha,d}''\,m^{2(\alpha_c-\alpha)}.
	\]
	Taking square roots proves part (iii), and thus the proof is complete.
\end{proof}

\textbf{Approximation rates for the ReQU variation space in $W^{q,\infty}$}

We are now ready to state the main approximation result for the infinite-width ReQU variation space in Sobolev norms on the ball.

\begin{theorem}
	\label{thm:ReQU_variation_Wrinfty}
	Fix $d\ge2$ and $\sigma(t)=\max\{0,t\}^2$. Let $\alpha>0$ and set $a:=\lceil \alpha\rceil-1$.
	Let $q$ be a non-negative integer such that $q\le a$, and define the critical index $\alpha_c:=(d+5)/2$. Then there exist constants $c,C>0$, depending only on $d,\alpha, q$, such that for all sufficiently large $M$:

	\begin{enumerate}[label=(\alph*),align=left]
		\item[\textbf{(supercritical):}] \ \
		If $\alpha>\alpha_c$ or if $\alpha=\alpha_c$ is an integer, then there exists $M_0<\infty$ such that $\mathcal H^\alpha \subseteq \mathcal V_\sigma(M_0)$.
		
		\item[\textbf{(critical):}] \ \
		If $\alpha=\alpha_c$ is not an integer, then
		\[
		\sup_{h\in\mathcal H^\alpha}\ \inf_{f\in\mathcal V_\sigma(M)}
		\|h-f\|_{W^{q,\infty}(\mathbb B^d)}
		\le
		C\exp\!\bigl(-c\,M^2\bigr).
		\]
		
		\item[\textbf{(subcritical):}] \ \
		If $\alpha<\alpha_c$, then
		\[
		\sup_{h\in\mathcal H^\alpha}\ \inf_{f\in\mathcal V_\sigma(M)}
		\|h-f\|_{W^{q,\infty}(\mathbb B^d)}
		\le
		C\,M^{-\,\frac{\alpha - q}{\alpha_c-\alpha}}.
		\]
	\end{enumerate}
\end{theorem}

\begin{proof}
	Let $h\in\mathcal H^\alpha$ and let $\tilde h$ be the associated extension to the sphere from Proposition~\ref{prop:transfer_variation_ball_sphere}(ii). For each integer $m\ge 1$ define $g_m:=\tilde h*L_m$ as in Definition~\ref{def:spectral_cutoff_kernel}, and set $f_m:=\mathcal T g_m$. We first estimate the approximation error in $W^{q,\infty}(\mathbb B^d)$. Since $q\le a$, Proposition~\ref{prop:cq_smooth_approx} applies. Together with Lemma~\ref{lemma:transfer_sobolev_T}, this gives
	\[
	\|h-f_m\|_{W^{q,\infty}(\mathbb B^d)}
	=
	\|\mathcal T(\tilde h-g_m)\|_{W^{q,\infty}(\mathbb B^d)}
	\le
	C_{q,d}\,\|\tilde h-g_m\|_{C^q(\mathbb S^d)}
	\le
	B_{\alpha,q,d}\,m^{-(\alpha - q)}.
	\]
	
	We next control the variation norm: since $f_m=\mathcal T g_m$, Proposition~\ref{prop:transfer_variation_ball_sphere}(i) gives $\gamma(f_m)\le \gamma(g_m)$. We now split by cases, beginning with the supercritical regime: assume that $\alpha>\alpha_c$, or that $\alpha=\alpha_c$ and $\alpha_c$ is an integer. By Proposition~\ref{prop:variation_of_gm}(i), there exists $A_{\alpha,d}>0$ such that $\gamma(g_m)\le A_{\alpha,d}
	\text{ for all } m\ge 1$. For each $m\ge 1$, choose a representing measure $\mu_m$ for $g_m$ such that $\|\mu_m\|_{\mathrm{TV}}\le A_{\alpha,d}+1/m$. By Banach-Alaoglu, after passing to a subsequence that we do not relabel, there exists a finite signed Radon measure $\mu$ on $\mathbb S^d$ such that $\mu_m$ converges to $\mu$ in the weak-$\ast$ topology and
	\[
	\|\mu\|_{\mathrm{TV}}
	\le
	\liminf_{m\to\infty}\|\mu_m\|_{\mathrm{TV}}
	\le
	A_{\alpha,d}.
	\]
	By Lemma~\ref{lemma:laplacian_smooth_approx}, one has $g_m\to\tilde h$ uniformly on $\mathbb S^d$. For each fixed $u\in\mathbb S^d$, the function $v\mapsto \sigma(u\cdot v)$ is continuous and bounded, hence
	\[
	\tilde h(u)
	=
	\lim_{m\to\infty} g_m(u)
	=
	\lim_{m\to\infty}\int_{\mathbb S^d}\sigma(u\cdot v)\,d\mu_m(v)
	=
	\int_{\mathbb S^d}\sigma(u\cdot v)\,d\mu(v).
	\]
	Thus $\tilde h\in\mathcal G_\sigma(A_{\alpha,d})$. Proposition~\ref{prop:transfer_variation_ball_sphere}(i) now yields $h=\mathcal T\tilde h\in\mathcal V_\sigma(A_{\alpha,d})$. Taking $M_0:=A_{\alpha,d}$ therefore proves the supercritical case.
	
	We next consider the critical regime, where $\alpha=\alpha_c$ and $\alpha_c$ is not an integer. By Proposition~\ref{prop:variation_of_gm}(ii), there exists $A_{\alpha,d}>0$ such that $\gamma(g_m)\le A_{\alpha,d}\sqrt{\log m} \text{ for all } m\ge 2$. Now, fix $M>0$ large enough so that $\exp\left(M^2/(4A_{\alpha,d}^2)\right)\ge 2$, and define $ m(M):= \left\lfloor \exp\left(M^2/(4A_{\alpha,d}^2)\right) \right\rfloor$. Then $m(M)\ge2$ and $\log m(M)\le M^2/(4A_{\alpha,d}^2)$. The choice of $m(M)$ gives $\gamma(f_{m(M)})\le A_{\alpha,d}\sqrt{\log m(M)}\le M/2$, so $f_{m(M)}\in\mathcal V_\sigma(M)$. Note also that for all sufficiently large $M$ one has $m(M)\ge \exp\!\left(M^2/(4A_{\alpha,d}^2)\right)/2$, so the approximation bound gives
	\[
	\|h-f_{m(M)}\|_{W^{q,\infty}(\mathbb B^d)}
	\le
	B_{\alpha,q,d}\,m(M)^{-(\alpha - q)}
	\le
	B_{\alpha,q,d}\,2^{(\alpha - q)}
	\exp\!\left(-\frac{\alpha - q}{4A_{\alpha,d}^2}M^2\right).
	\]
	Therefore there exist constants $c,C>0$, depending only on $d,\alpha,q$, such that
	\[
	\inf_{f\in\mathcal V_\sigma(M)}
	\|h-f\|_{W^{q,\infty}(\mathbb B^d)}
	\le
	C\exp\bigl(-c\,M^2\bigr)
	\]
	for all sufficiently large $M$. Since the constants are uniform in $h\in\mathcal H^\alpha$, taking the supremum over $h$ proves the critical case.
	
	Finally, assume that $\alpha<\alpha_c$. By Proposition~\ref{prop:variation_of_gm}(iii), there exists $A_{\alpha,d}>0$ such that $\gamma(g_m)\le A_{\alpha,d}\,m^{\alpha_c-\alpha} \text{ for all } m\ge 1$. Now, fix again $M>0$ large enough so that $\left(M/(2A_{\alpha,d})\right)^{1/(\alpha_c-\alpha)}\ge 2$, and define $m(M):=\left\lfloor \left(M/(2A_{\alpha,d})\right)^{1/(\alpha_c-\alpha)}\right\rfloor$. Then $m(M)\ge1$, and $\gamma(f_{m(M)})\le A_{\alpha,d}m(M)^{\alpha_c-\alpha}\le M/2$, so $f_{m(M)}\in\mathcal V_\sigma(M)$. Since $m(M)\ge x/2$ whenever $x\ge 2$ and $m(M)=\lfloor x\rfloor$, we also have $2m(M)\ge\left({M}/({2A_{\alpha,d}})\right)^{1/(\alpha_c-\alpha)}$ for all sufficiently large $M$. Therefore
	\[
	\|h-f_{m(M)}\|_{W^{q,\infty}(\mathbb B^d)}
	\le
	B_{\alpha,q,d}\,m(M)^{-(\alpha - q)}
	\le
	C\,M^{-\,\frac{\alpha - q}{\alpha_c-\alpha}}
	\]
	with a constant $C>0$ depending only on $d,\alpha,q$. Taking the infimum over $f\in\mathcal V_\sigma(M)$ and then the supremum over $h\in\mathcal H^\alpha$ proves the subcritical case.
\end{proof}

\subsection{From variation space to finite-width networks}
\label{subsec:finite_width_networks}

In this subsection we pass from the infinite-width approximation results for the ReQU variation space established in Theorem~\ref{thm:ReQU_variation_Wrinfty} to approximation guarantees for actual finite-width shallow networks. The key point is that the same discrete measure that approximates the representing measure of a target function in the variation space automatically approximates all derivatives up to order two, so that the sampling argument can be formulated in the Banach space $W^{2,\infty}(\mathbb B^d)$ instead of only $L^\infty(\mathbb B^d)$.

\begin{definition}[Finite-width ReQU networks]
	\label{def:finite_width_ReQU}
	Fix $d\in\N$ and let $\sigma:t\mapsto\max\{0,t\}^2$ be the ReQU activation. Let $N\in\N$ and $M>0$. We consider the subclass of fully connected networks from Definition~\ref{def:fcnn_hypothesis_space} given by $L=1$, $\|\mathbf{m}\|_\infty\le N$ and $\|c\|_\infty\le M$, with a fixed inner-layer bound $B_0>0$ that depends only on $d$. In other words, we take $\Sigma^\sigma_{\mathbf{m}_{1:L},M}(B_0)$ with $L=1$ and $m_1\le N$ as the class of shallow ReQU networks with width at most $N$ and outer coefficients bounded by $M$. Each such network can be written in the form
	\[
	f_N(x)=a_0 + \sum_{j=1}^{m_1} a_j\,\sigma(w_j^\top x + b_j),
	\]
	where $m_1\le N$, the inner parameters satisfy $\|w_j\|_\infty\le B_0$ and $|b_j|\le B_0$, and the outer coefficients satisfy $|a_j|\le M$ for $0\le j\le m_1$. The constant $B_0$ is chosen so that, for every $v\in\mathbb S^d$, the feature map used in Definition~\ref{def:spherical_variation_space} is realized by some $(w,b)$ with $\|w\|_\infty\le B_0$ and $|b|\le B_0$.
\end{definition}

By Definitions~\ref{def:ReQU_activation_variation} and \ref{def:spherical_variation_space}, every function $f\in\mathcal V_\sigma(M)$ admits a representation
\[
f(x)
=
\int_{\mathbb S^d} \sigma\bigl((x^\top,1)\cdot v\bigr)\,d\mu(v),
\text{ for all } x\in\mathbb B^d,
\]
for some signed Radon measure $\mu$ on $\mathbb S^d$ with $\|\mu\|_{\mathrm{TV}}\le M$. Replacing $\mu$ by a discrete measure $\mu_N = \sum_{j=1}^{m_1} a_j\,\delta_{v_j}$ with at most $m_1\le N$ atoms produces a one-hidden-layer network of the form in Definition~\ref{def:finite_width_ReQU}. The total variation of $\mu_N$ is $\|\mu_N\|_{\mathrm{TV}} = \sum_{j=1}^{m_1}|a_j|$, so that $\|\mu_N\|_{\mathrm{TV}}\le M$ implies $|a_j|\le M$ for every $j$ and therefore $f_N\in\Sigma^\sigma_{\mathbf{m}_{1:L},M}(B_0)$.

We next state the result that allows us to approximate a given function in the ReQU variation space by a finite-width one-hidden-layer network with control in the $W^{2,\infty}$-norm.

\begin{lemma}
	\label{lemma:sampling_Wrinfty}
	Fix $d\in\N$ and the ReQU activation $\sigma:t\mapsto\max\{0,t\}^2$. There exists a constant $C_d>0$ such that the following holds: for every $M>0$ and every $f\in\mathcal V_\sigma(M)$ in the sense of Definition~\ref{def:ReQU_activation_variation}, there exists a shallow ReQU network $f_N$ of the form described in Definition~\ref{def:finite_width_ReQU}, with width at most $N$ and outer coefficient bound at most $C_d M$, such that
	\[
	\|f-f_N\|_{W^{2,\infty}(\mathbb B^d)}
	\le
	C_d\,M\,N^{-\theta_d}, \text{ where } 
	\theta_d = \frac{1}{2} + \frac{1}{2d}.
	\]
\end{lemma}

\begin{proof}
	This follows from \citep[Theorem 3]{siegel2025optimal} with $k=m=2$ and domain $\mathbb B^d$. Their dictionary $\mathbb P_2^d$ consists of ReQU ridge functions, and its variation ball agrees, up to a constant depending only on $d$, with the ball $\mathcal V_\sigma(M)$ in Definition~\ref{def:ReQU_activation_variation}. Their theorem thus gives the claimed $W^{2,\infty}(\mathbb B^d)$ rate using at most $N$ ridge functions whose coefficients have total absolute value at most $C_dM$, and whose inner parameters are controlled by the fixed bound $B_0$ in Definition~\ref{def:finite_width_ReQU}. This implies immediately the desired result.
\end{proof}

%We now combine Lemma~\ref{lemma:sampling_Wrinfty} with the infinite-width approximation result of Theorem~\ref{thm:ReQU_variation_Wrinfty}. We use Theorem~\ref{thm:ReQU_variation_Wrinfty} with $r'=2$ to approximate a target function $h\in\mathcal H^\alpha$ by an element $f$ of the variation ball $\mathcal V_\sigma(M)$ with a controllable $W^{2,\infty}$ error, and then apply Lemma~\ref{lemma:sampling_Wrinfty} to approximate $f$ by a finite-width network $f_N$. Optimizing over the intermediate parameter $M$ yields approximation rates in $W^{2,\infty}$ for the finite-width model that match the infinite-width rates in the subcritical regime and saturate at the Barron-type rate in the critical and supercritical regimes.

Having established approximation error bounds of smooth functions by the variation space in Theorem~\ref{thm:ReQU_variation_Wrinfty}, and approximation bounds of variation space elements by finite-width networks in Lemma~\ref{lemma:sampling_Wrinfty}, all that is left is combining these results to prove our final bound. We simply recall a well-known fact to relate Sobolev and Hölder spaces before proceeding:

\begin{lemma}[{\citep[Section 5.6.3]{evans2022partial}}]
	\label{lemma:Winfty_to_Cr-1,1}
	Let $r\in\N$ with $r\ge 1$ and let $u\in W^{r,\infty}(\R^d)$. Then $u|_{\mathbb B^d}$ has a representative in $C^{r-1,1}(\mathbb B^d)$ and there exists a constant $C_{d,r}$, such that
	\[
	\|u\|_{C^{r-1,1}(\mathbb B^d)}
	\le
	C_{d,r}\,\|u\|_{W^{r,\infty}(\R^d)}.
	\]
\end{lemma}

\begin{proof}[Proof of Theorem~\ref{thm:approx_requ}]
	Let $u\in W^{r,\infty}(\R^d)$ and set $h:=u|_{\mathbb B^d}$. The case $u=0$ is immediate. By Lemma~\ref{lemma:Winfty_to_Cr-1,1}, there is a constant $A=C_{d,r}\|u\|_{W^{r,\infty}(\R^d)}$ such that $h/A\in\mathcal H^r$. We first apply Theorem~\ref{thm:ReQU_variation_Wrinfty} to $h/A$, then Lemma~\ref{lemma:sampling_Wrinfty} to the resulting element of the variation space. Multiplying the sampled network by $A$ restores the scale of $h$ and multiplies its outer coefficients by $A$.
	
	Suppose $r<\alpha_c:=(d+5)/2$. For every sufficiently large variation radius $M$, Theorem~\ref{thm:ReQU_variation_Wrinfty} with $q=2$ gives $f\in\mathcal V_\sigma(M)$ such that $\|h/A-f\|_{W^{2,\infty}(\mathbb B^d)}\le C_{d,r}M^{-(r-2)/(\alpha_c-r)}$. Applying Lemma~\ref{lemma:sampling_Wrinfty} to $f$ with width $N^d$ gives a network $f_N$ with outer coefficients bounded by $C_dM$. Thus
	\[
	\|h-Af_N\|_{W^{2,\infty}(\mathbb B^d)}
	\le C_{d,r}A\bigl(M^{-(r-2)/(\alpha_c-r)}+MN^{-d\theta_d}\bigr).
	\]
	Since $d\theta_d=(d+1)/2$, the choice $M$ of order $N^{\alpha_c-r}$ makes both terms of order $N^{-(r-2)}$. The outer coefficients of $Af_N$ are then bounded by $C_{d,r}A N^{\alpha_c-r}$.
	
	If $r\ge\alpha_c$, Theorem~\ref{thm:ReQU_variation_Wrinfty} gives $h/A\in\mathcal V_\sigma(M_0)$ for a constant $M_0$ depending only on $d,r$. This includes equality at the critical index because $r$ is an integer. Lemma~\ref{lemma:sampling_Wrinfty}, again with width $N^d$, gives $\|h-Af_N\|_{W^{2,\infty}(\mathbb B^d)}\le C_{d,r}A N^{-(d+1)/2}$ and outer coefficients bounded by $C_{d,r}A$.
	
	Combining the two cases gives $\gamma(r)=\min\{r-2,(d+1)/2\}$ and the coefficient bound $M_N(u)$ in the theorem. The inner parameters satisfy the fixed bound $B_0$, and zero coefficients may be added if exactly $N^d$ units are needed. This proves \eqref{eqn:approx_requ_1}.
\end{proof}

\section{VC dimension bounds}
\label{sec:vcdim_bounds}

In this section, we prove the VC dimension bounds stated in Theorem~\ref{thm:vcdim_bounds}. We first establish these bounds for standard FNNs and their derivatives, and will extend them to boundary-enforced networks in Section~\ref{sec:boundary_effect} at the end.

\subsection{ReQU FNNs}

We aim to derive bounds on the VC dimension of the hypothesis spaces defined by $\alpha$-th order partial derivatives of fully connected neural networks (FNNs) with ReQU activations, where $\alpha \in \mathbb{N}^d$ and $\|\alpha\|\leq 2$. This problem is challenging because the compositional structure of FNNs leads to highly intricate derivative expressions, making it difficult to apply standard VC dimension bounding techniques. Although \citep{yang2023nearly} obtained optimal VC dimension and pseudodimension bounds for FNNs with ReLU activations, and \citep{lei2025solving} established VC dimension bounds for higher-order derivatives of deep convolutional ReLU networks composed with a shallow ReLU$^k$ network, these results do not directly extend to the ReQU setting. We therefore develop in this section novel bounds tailored specifically for FNNs with ReQU activations and their derivatives.

To estimate the VC dimension of FNNs with piecewise polynomial activation, we first need these essential Lemmas:

\begin{lemma}[\citep{bartlett1999neural}, Theorem 8.3]
	\label{lemma:vc_dim_poly}
	Let $p_1, \ldots, p_M$ be polynomials in $W\le M$ variables with degree at most $D$. Define the sign function as $\sgn(x):= \ind_{x>0}$, and let
	\[K := \left|\left\{(\sgn(p_1(\theta)),\ldots, \sgn(p_M(\theta)))\mid \theta\in \R^W\right\}\right|\]
	be the number of sign vectors that can be realized by $p_1,\ldots,p_M$. We have $K\le 2(2eMD/W)^W$.
\end{lemma}

\begin{lemma}[\citep{bartlett2019nearly}, Lemma 18]
	\label{lemma:vc_dim_poly_2} Suppose that $2^m\le 2^t(mq/w)^w$ for some $q\ge16$ and $m\ge w\ge t\ge 0$. Then, $m\le t + w\log_2(2q\log_2 q)$.
\end{lemma}

Before starting our analysis, we begin by recalling the bound on the VC dimension of standard ReQU FNNs spaces, which follows directly from \citep[Theorem~7]{bartlett2019nearly}:

\begin{proposition}
	Let $\sigma:x\mapsto \max\{0,x\}^2$ be the ReQU activation function, and denote by $\mathcal{F} = \Sigma^\sigma_{\mathbf{m}_{1:L}}(\infty)$ the hypothesis space of ReQU FNNs with no restriction on their weights, as per Definition \ref{def:fcnn_hypothesis_space}. We have the estimate
	\[\VCdim(\mathcal{F}) = \bigO\left(\size L\big[L + \log_2(L\|\bm\|_\infty) \big]\right),\]
	where $\size = \sum_{i=1}^L m_i(m_{i-1}+1) + m_L + 1$ is the number of parameters in the architecture.
\end{proposition}

\subsubsection{First order derivatives}

For convenience, denote $\sigma_2 :x\mapsto \max\{0,x\}^2$ the ReQU activation, and $\sigma_1:x\mapsto 2\max\{0,x\}$ its derivative. For vector arguments, $\sigma_2$ acts componentwise, while $\sigma_1$ denotes the diagonal matrix of the corresponding scalar derivatives. Recall that an element $\phi\in \Sigma^{\sigma_2}_{\mathbf{m}_{1:L}}(\infty)$ can be written as 
\[\phi:\bx\mapsto \bW_{L+1}\sigma_2(\bW_L\sigma_2(\cdots \sigma_2(\bW_1 \bx + \bb_1) \cdots) +\bb_L) + b_{L+1}.\]
For any $1\le i\le d$, the partial derivative with respect to the $i$-th coordinate can thus be written as
\begin{equation}
	\label{eqn:first_requ}
	\begin{split}
		\partial_i\phi:\bx\mapsto \bW_{L+1}\sigma_1(\bW_L\sigma_2(\cdots \sigma_2(\bW_1 \bx + \bb_1) \cdots) +\bb_L)\\
		\cdot \bW_L\sigma_1(\cdots \sigma_2(\bW_1 \bx + \bb_1) \cdots)\\ 
		\cdots \bW_2\sigma_1(\bW_1\bx + \bb_1)(\bW_1)_i,
	\end{split}
\end{equation}
where $(\bW)_i$ denotes the $i$-th column of $\bW$. Note that for fixed $\bx\in\R^d$, $\partial_i\phi$ can be seen as a function of the parameters (weights and biases), in which case we can denote it $\partial_i \phi(\bx, \cdot)$.  

Following an approach similar to \citep{bartlett2019nearly,yang2023nearly}, our goal is to find a partition $\partition$ of the parameter space $\R^\size$ such that, for fixed $x_1,\ldots,x_M$, and any $P\in\partition$, the functions $\partial_i\phi(x_1,\cdot),\ldots, \partial_i\phi(x_M,\cdot)$ are polynomials on $P$. Indeed if we denote
\[K = \left|\left\{(\sgn(\partial_i\phi(x_1,\theta)),\ldots, \sgn(\partial_i\phi(x_M,\theta)))\mid \theta\in \R^\size\right\}\right|,\]
we have
\[K \le \sum_{P\in \partition} \left|\left\{(\sgn(\partial_i\phi(x_1,\theta)),\ldots, \sgn(\partial_i\phi(x_M,\theta)))\mid \theta\in P \right\}\right|.\]
Therefore, by applying Lemmas \ref{lemma:vc_dim_poly} and \ref{lemma:vc_dim_poly_2}, we can deduce from the cardinality of the partition $\partition$ a bound on the VC dimension. Note that it suffices to consider $M\ge(L+1)\size$, since smaller $M$ already satisfies the desired bound. Set $m_{L+1}=1$, so that $\sum_{k=1}^{L+1}m_k(m_{k-1}+1)=\size$ includes the output weights and bias. We begin by defining a family of functions with respect to the parameter $\btheta\in\R^\size$:

\begin{equation}
	\label{eqn:func_family_1}
	\begin{aligned}
		\Func_0&:=\{(\bW_1)_i,\bW_1\bx+\bb_1\},\\
		\Func_1&:=\{(\bW_1)_i,\bW_2\sigma_1(\bW_1\bx+\bb_1),
		\bW_2\sigma_2(\bW_1\bx+\bb_1)+\bb_2\},\\
		\Func_L&:=\{(\bW_1)_i,\bW_2\sigma_1(\bW_1\bx+\bb_1),\ldots,
		\bW_{L+1}\sigma_1(\bW_L\sigma_2(\cdots)+\bb_L)\}.
	\end{aligned}
\end{equation}

For $0\le\ell<L$, label the elements of $\Func_\ell$ in their displayed order as $f_0^{(\ell)},\ldots,f_{\ell+1}^{(\ell)}$. They consist of $(\bW_1)_i$, the next $\ell$ matrix factors in \eqref{eqn:first_requ} counted from the input, and the input to layer $\ell+1$ before activation. The final family $\Func_L$ contains only the $L+1$ matrix factors $f_0^{(L)},\ldots,f_L^{(L)}$. When $L=1$, the intermediate line for $\Func_1$ is omitted.

We will build the partition recursively through $\partition_0,\ldots,\partition_L$, successively refining it over each layer. These refinements will satisfy the following properties:
\begin{enumerate}[font=\bfseries]
	\item\label{itm:one} $|\partition_0|=1$, and for $1\le\ell\le L$,
	\begin{equation}
		\label{eqn:partition_ineq_1}
		\frac{|\partition_\ell|}{|\partition_{\ell-1}|}
		\le2\left(\frac{2eMm_\ell(2^\ell-1)}{\sum_{k=1}^\ell m_k(m_{k-1}+1)}\right)^{\sum_{k=1}^\ell m_k(m_{k-1}+1)}.
	\end{equation}
	\item\label{itm:two} For $0\le\ell<L$ and $P\in\partition_\ell$, each element of $\Func_\ell$, restricted to $P$, is componentwise polynomial in $\sum_{k=1}^{\ell+1}m_k(m_{k-1}+1)$ parameters, of degree at most $2^{\ell+1}-1$.
	\item\label{itm:three} For $P\in\partition_L$ and $0\le k\le L$, the factor $f_k^{(L)}$, restricted to $P$, is componentwise polynomial in $\sum_{h=1}^{k+1}m_h(m_{h-1}+1)$ parameters, of degree at most $2^k$.
\end{enumerate}

We begin with $\partition_0:=\{\R^\size\}$. Both $(\bW_1)_i$ and $\bW_1\bx+\bb_1$ are affine in the first-layer parameters, so items~\ref{itm:one} and \ref{itm:two} hold.

For the induction step, assume that $\partition_0,\ldots,\partition_{\ell-1}$ have been defined. For $1\le h\le m_\ell$, $1\le j\le M$ and $P\in\partition_{\ell-1}$, denote by $p_{h,\bx_j,P}(\btheta)$ the input to unit $h$ of layer $\ell$ at $\bx_j$, restricted to $P$. By item~\ref{itm:two}, these are $Mm_\ell$ polynomials in $\sum_{k=1}^\ell m_k(m_{k-1}+1)$ parameters, of degree at most $2^\ell-1$. Lemma~\ref{lemma:vc_dim_poly} bounds their number of sign patterns by the right-hand side of \eqref{eqn:partition_ineq_1}. Intersecting each sign region with $P$, and repeating over $P\in\partition_{\ell-1}$, gives $\partition_\ell$ and proves item~\ref{itm:one}.

We now check item~\ref{itm:two}. The retained factors remain polynomial because $\partition_\ell$ refines $\partition_{\ell-1}$. The remaining input $f_\ell^{(\ell-1)}$ has fixed sign componentwise on each new region. Applying $\sigma_1$ and multiplying by $\bW_{\ell+1}$ gives the new matrix factor $f_\ell^{(\ell)}$ with degree at most $2^\ell$. Applying $\sigma_2$, multiplying by $\bW_{\ell+1}$ and adding $\bb_{\ell+1}$ gives the next input $f_{\ell+1}^{(\ell)}$ with degree at most $2^{\ell+1}-1$. Both involve parameters only through layer $\ell+1$, proving item~\ref{itm:two}. At the final step, only the matrix factor is added. Each earlier matrix factor retains the degree bound from the step at which it was added, which proves item~\ref{itm:three}.

We are now ready to bound the VC dimension of ReQU network derivatives. By \eqref{eqn:first_requ}, $\partial_i\phi$ is the product $f_L^{(L)}\cdots f_0^{(L)}$. Item~\ref{itm:three} therefore gives degree at most $\sum_{k=0}^L2^k=2^{L+1}-1$ on every $P\in\partition_L$. Denote by $K_L$ the maximum number of derivative sign patterns on one such region. Lemma~\ref{lemma:vc_dim_poly} gives
\[
K_L\le2\left(\frac{2eM(2^{L+1}-1)}{\sum_{k=1}^{L+1}m_k(m_{k-1}+1)}\right)^{\sum_{k=1}^{L+1}m_k(m_{k-1}+1)}.
\]
On the other hand, recursively applying \eqref{eqn:partition_ineq_1} bounds $|\partition_L|$. Denoting the sum of the derivative sign-pattern counts over its regions by $\Pi_{\partial_i\F}(M)$, and using $m_{L+1}=1$, we obtain
\begin{equation}
	\label{eqn:requ_derivative_growth}
	\Pi_{\partial_i\F}(M)
	\le\prod_{\ell=1}^{L+1}2\left(\frac{2eMm_\ell(2^\ell-1)}{\sum_{k=1}^\ell m_k(m_{k-1}+1)}\right)^{\sum_{k=1}^\ell m_k(m_{k-1}+1)}.
\end{equation}
Set $w=\sum_{\ell=1}^{L+1}\sum_{k=1}^\ell m_k(m_{k-1}+1)$ and $q=2e\sum_{\ell=1}^{L+1}m_\ell(2^\ell-1)$. The weighted AM-GM inequality bounds the preceding product by $2^{L+1}(Mq/w)^w$. If $M$ points are shattered, then $2^M\le\Pi_{\partial_i\F}(M)$. Lemma~\ref{lemma:vc_dim_poly_2} thus gives $M\le L+1+w\log_2(2q\log_2q)$. Since $w\le(L+1)\size$ and $q\le2^{L+3}e\|\bm\|_\infty$, we conclude that $\VCdim(\partial_i\F)=\bigO\bigl(\size L[L+\log_2(2\|\bm\|_\infty)]\bigr)$. The same partitions give the same bound for the network output, whose degree is also at most $2^{L+1}-1$.

\subsubsection{Second order derivatives}

Again for convenience, denote by $\sigma_0(t):=2\ind_{t>0}$ the weak second derivative of $\sigma_2$, taking value zero at $t=0$. For vector arguments it denotes the diagonal matrix of these values, as for $\sigma_1$. Recall from \eqref{eqn:first_requ} that $\partial_i\phi=f_L^{(L)}\cdots f_0^{(L)}$. Denote $g_\ell:=\partial_j f_\ell^{(L)}$ for $1\le\ell\le L$. Since $\partial_jf_0^{(L)}=0$, the product rule gives
\begin{equation}
	\label{eqn:second_requ}
	\partial_{ij}\phi
	=\sum_{\ell=1}^L f_L^{(L)}\cdots f_{\ell+1}^{(L)}g_\ell f_{\ell-1}^{(L)}\cdots f_0^{(L)},
\end{equation}
where an empty product is omitted. As before, let $\Func_k^{(\ell)}=\Func_k$ for $k<L$. We define $\Func_L^{(\ell)}$ by replacing its factor $f_\ell^{(L)}$ by $g_\ell$, keeping the other factors in their original order. Each summand in \eqref{eqn:second_requ} is then the product of the factors in $\Func_L^{(\ell)}$.

Reusing the same partitions as in the first-derivative proof, we have by the chain rule
\[
g_\ell=\bW_{\ell+1}\sigma_0\bigl(f_\ell^{(\ell-1)}\bigr)
\operatorname{diag}\bigl(\partial_jf_\ell^{(\ell-1)}\bigr).
\]
To estimate its degree, note that the first-layer input and its $x_j$ derivative are affine in the parameters. At each later layer, the input derivative is the preceding derivative multiplied by $\sigma_1$ of the preceding input and by the next weight matrix. Thus the input to layer $\ell$ and its first input derivatives have degree at most $2^\ell-1$, by the same induction as before. On each region, $\sigma_0$ is constant, so $g_\ell$ has degree at most $2^\ell$ and depends on parameters only through layer $\ell+1$.

Replacing $f_\ell^{(L)}$ by $g_\ell$ therefore preserves the degree bounds in item~\ref{itm:three}. Every summand, and hence $\partial_{ij}\phi$, has degree at most $2^{L+1}-1$. Repeating the sign count \eqref{eqn:requ_derivative_growth} gives $\VCdim(\partial_{ij}\F)=\bigO\bigl(\size L[L+\log_2(2\|\bm\|_\infty)]\bigr)$. By Lemma~\ref{lem:rho_invariance}, the same bounds apply to boundary-enforced ReQU networks, completing the ReQU proof.\qed

\subsection{tanh FNNs}

We now turn our attention to the case of FNNs with tanh activation. The hyperbolic tangent not being a piecewise polynomial, there is no way to extend the analysis in the last section to this case. We thus begin by introducing the class of so-called \textit{Pfaffian functions} \citep{khovanskiui1991fewnomials}, together with related results which will be crucial to carry our analysis.

\subsubsection{VC dimension bounds for Pfaffian functions}
\label{sec:pfaffian}

\begin{definition}[Pfaffian Functions]
	\label{def:pfaffian_func}
	A \emph{Pfaffian chain} of order $\ell \ge 1$ and degree $\gamma \ge 1$, in an open domain $U \subseteq \R^d$, is a sequence of real-valued analytic functions $f_1, f_2, \dots, f_\ell$
	over $U$, satisfying the differential equations
	\begin{equation}
		\label{eqn:pfaffian_eq_diff}
		\frac{\partial f_j(\bx) }{\partial x_i}= g_{ij}\Bigl(\bx, f_1(\bx), \dots, f_j(\bx)\Bigr), \quad 1 \le j \le \ell,
	\end{equation}
	where the functions $g_{ij}(\bx, y_1, \dots, y_j)$ are polynomials in $\mathbf{x}\in U$ and $y_1,\dots,y_j \in \R$ of degree at most $\gamma$. A function
	\[
	f:\bx\mapsto P\Bigl(\mathbf{x}, f_1(\bx), \dots, f_\ell(\bx)\Bigr),
	\]
	where $P(\bx, y_1, \dots, y_\ell)$ is a polynomial of degree not exceeding $\beta$, is called a Pfaffian function of format $(\gamma, \beta, \ell)$. We might alternatively say that $f$ is Pfaffian of order $\ell$ and degree $(\gamma,\beta)$.  
\end{definition}

Pfaffian functions constitute an extremely broad class which includes most analytic functions one may encounter in practical applications \citep{khovanskiui1991fewnomials}. For instance, sigmoidal activation functions such as $\arctan$, the logistic sigmoid $\operatorname{sig}$, and the hyperbolic tangent $\tanh$ are Pfaffian functions, with respective formats  $\format(\arctan)=(3,1,2)$, $\format(\operatorname{sig})=(2,1,1)$, and $\format(\tanh)=(2,1,1)$.

Pfaffian functions were first introduced in \citep{khovanskiui1991fewnomials} to generalize Bézout's theorem to bound the number of solutions to systems of polynomial equations. The theory of Pfaffian function has been applied to bounding the VC dimension of FNNs using Pfaffian activation functions first in the seminal work \citep{karpinski1997polynomial}, and more recently for other architectures such as Graph Neural Networks and Recurrent Neural Networks \citep{dinverno2025vc,scarselli2018vapnik}.

Following the framework proposed in \citep{karpinski1997polynomial} we see a classifier as an object described by a logical formula constructed from Pfaffian equations: concretely, for some integer  $\Bar{s}\ge1$, we let $ \phi_1, \dots, \phi_{\Bar{s}} $ be a collection of $ C^{\infty} $ functions from $ \R^{d+p} $ to $ \mathbb{R} $, and we suppose that for each $\bx\in\mathbb{R}^{d}, \btheta\in\mathbb{R}^{p},$ there is a logical expression $\Phi(\bx,\btheta)$ built using the operators \textit{and} and \textit{or} from simpler conditions (atomic predicates) of the form
$ \phi_i(\bx,\btheta)=0 $ and $ \phi_i(\bx,\btheta)>0$, which describes our classifier. In the case of a tanh FNN, the decision rule is fully specified by the logical formula $\phi_1(\bx,\btheta)\ge0$, where we take $\Bar{s}=1$ and $\phi_1$ is taken as the network's output. With this formalism, the prediction is then given by $1$ if the formula is true and $0$ otherwise.

We then define the hypothesis space $\F_\Phi:=\{\bx\mapsto\Phi(\bx,\btheta)\mid\btheta\in\R^p\}$ in the usual manner, and likewise its VC dimension $\VCdim(\F_\Phi)$ as the maximum possible cardinality of a set shattered by $\F_\Phi$. As it turns out, $\VCdim(\F_\Phi)$ can be bounded by studying the topological properties of the inverse image—in the parameter space—of the functions $\phi_i$ \citep{karpinski1997polynomial}. More precisely, for integers $1\le u,z\leq p$, and $\bx_1, \dots, \bx_z$ arbitrary points in $\mathbb{R}^{d}$, define the mapping $\mathbf{T}:\mathbb{R}^{p}\to\mathbb{R}^{u}, u\leq p$,
by
\begin{equation}
	\label{eqn:TPaff}
	\mathbf{T}(\btheta)=[\Bar{\phi}_1(\btheta),\dots,\Bar{\phi}_{u}(\btheta)],
\end{equation}
where each $\Bar{\phi}_r(\btheta)$ is obtained by evaluating one of the functions $\phi_i$ at one of the points $\bx_j$. That is, for each $1\le r\le u$, there exist indices $i$ and $j$ such that $\Bar{\phi}_r(\btheta)=\phi_i(\bx_j,\btheta)$. Then, the following theorem holds:

\begin{theorem}[\citep{karpinski1997polynomial}, Theorem 2]
	\label{thm:karpinsky97}
	Assume that there exists a positive integer $B$ which uniformly bounds the number of connected components of the inverse image $\mathbf{T}^{-1}([\epsilon_1,\dots,\epsilon_{u}])$, for any $1\le u,z\le p$, any choice of points $\bx_1,\ldots\bx_z$, and any $[\epsilon_1,\dots,\epsilon_{u}]$ which is a regular value of $\mathbf{T}$. Then, the following VC dimension bound on $\F_\Phi$ holds:
	\[\VCdim(\F_\Phi)\leq 2\log B+p\Bigl(16+2\log\Bar{s}\Bigr).
	\]
\end{theorem}

Theorem~\ref{thm:karpinsky97} thus provides a bound on the VC dimension of $\F_\Phi$ in terms of the number $p$ of parameters, the total number $\Bar{s}$ of atomic functions $\phi_i$, and the constant $B$ bounding the number of connected components of $\mathbf{T}^{-1}([\epsilon_1,\dots,\epsilon_{u}])$. For this result to be useful however, one crucially needs a concrete estimate on the quantity $B$. This is where the following bound for systems of Pfaffian equations comes into play:

\begin{theorem}[\citep{gabrielov2004complexity}, Corollary 3.3]
	\label{thm:gabrielov}
	Consider a system of equations 
	\[
	\Bar{q}_1(\btheta)=0, \ldots, \Bar{q}_k(\btheta)=0,
	\]
	where each $\Bar{q}_i$, for $1 \leq i \leq k$, is a Pfaffian function defined on $\mathbb{R}^{p}$. Assume that these functions share a common Pfaffian chain of length $\ell$ and have maximum degrees $(\gamma, \beta)$. Then, the number of connected components of the set $ \{\btheta \in \R^p \mid \Bar{q}_1(\btheta)=0, \dots, \Bar{q}_k(\btheta)=0\}$	is bounded by
	\[
	2^{\frac{\ell(\ell-1)}{2}+1} (\gamma + 2\beta-1)^{p-1}\Bigl((2p-1)(\gamma+\beta)-2p+2\Bigr)^{\ell}.
	\]
\end{theorem}

%Theorem \ref{thm:gabrielov} is in fact a Corollary of an earlier result from \citep{khovanskiui1991fewnomials}. 
Combining Theorems \ref{thm:karpinsky97} and \ref{thm:gabrielov}, together with knowledge on the Pfaffian format of the $\tanh$ function, we can deduce the following VC dimension bounds for tanh FNNs:
\begin{proposition}[\citep{karpinski1997polynomial}, Theorem 7]
	\label{prop:karpinski_vcdim}
	Let $\sigma:x\mapsto \tanh(x)$ be the tanh activation function, and denote by $\F = \Sigma^\sigma_{\mathbf{m}_{1:L}}(\infty)$ the hypothesis space of tanh FNNs with no restriction on their weights, as per Definition \ref{def:fcnn_hypothesis_space}. We have the estimate
	\[\VCdim(\F) = \bigO\left(L^2\|\bm\|_\infty^2\size^2\right),\]
	where $\size = \sum_{i=1}^L m_i(m_{i-1}+1) + m_L + 1$ is the number of parameters in the architecture.
\end{proposition}

Because it will be a helpful reference for later, we give a proof of Proposition \ref{prop:karpinski_vcdim} below, which is merely a modern and slightly more fleshed-out presentation of the original argument in \citep{karpinski1997polynomial}.

\begin{proof}
	Let $\phi\in\F$. We start by building a Pfaffian chain to which it belongs. To do so, we proceed iteratively, layer by layer: in the first layer we have $m_1$ neurons, so for $1\le i\le m_1$ let
	\[
	g_i^{(1)}(\bx,\btheta):=(\bW_1\bx)_i+(\bb_1)_i,
	\text{ and }
	h_i^{(1)}(\bx,\btheta):=\sigma(g_i^{(1)}(\bx,\btheta)),
	\]
	where $(\bW_1\bx)_i$ and $(\bb_1)_i$ denote the $i$-th coordinates of $\bW_1\bx$ and $\bb_1$. Now, for the next step, let $2\le \ell \le L$ and assume that we have already defined $g_i^{(k)}$ and $h_i^{(k)}$ for all $1\le i\le m_k$ and $1\le k\le \ell-1$, and let for all $1\le i\le m_\ell$:
	\begin{equation}
		\label{eqn:tanh_recurrence}
		g_i^{(\ell)}(\bx,\btheta) := \sum_{j=1}^{m_{\ell - 1}} (\bW_\ell)_{ij}\  h_j^{(\ell-1)}(\bx,\btheta) +  (\bb_\ell)_i, \quad h_i^{(\ell)}(\bx,\btheta):= \sigma(g_i^{(\ell)}(\bx,\btheta)).
	\end{equation}
	Fix $\phi\in\F$ and consider it as a real-analytic function on $U:=\R^{d+p}$ with variables $(\bx,\btheta)$.
	Define the ordered list of functions
	\[
	\mathcal{C}(\bx,\btheta)
	:=
	\Bigl(g_i^{(\ell)}(\bx,\btheta),\,h_i^{(\ell)}(\bx,\btheta)\Bigr)_{1\le \ell\le L,\ 1\le i\le m_\ell},
	\]
	where $h_i^{(\ell)}=\sigma(g_i^{(\ell)})$ and the $g_i^{(\ell)}$ are given by \eqref{eqn:tanh_recurrence}. Its length is
	\[
	\ell_0 := 2\sum_{\ell=1}^L m_\ell \le 2L\|\bm\|_\infty.
	\]
	
	We claim that $\mathcal C$ is a Pfaffian chain of degree $\gamma=3L$. Indeed, let $v$ be any coordinate among $(x_1,\ldots,x_d,\theta_1,\ldots,\theta_p)$. For $\ell=1$, $g_i^{(1)}=\sum_j(W_1)_{ij}x_j+(b_1)_i$ is a polynomial of degree at most two in these variables, because each input term is a weight times an input coordinate. Hence $\partial_v g_i^{(1)}$ has degree at most one.
	Moreover,
	\[
	\partial_v h_i^{(1)}=(1-(h_i^{(1)})^2)\,\partial_v g_i^{(1)},
	\]
	so $\partial_v h_i^{(1)}$ is a polynomial in $(\bx,\btheta,h_i^{(1)})$ of degree at most $3$. Now, assume inductively that for all layers $1\le k \le \ell-1$ the derivatives $\partial_v h_j^{(k)}$ are polynomials in $(\bx,\btheta)$ and earlier chain elements of degree at most $3k$.
	Since
	\[
	g_i^{(\ell)}=\sum_{j=1}^{m_{\ell-1}} (W_\ell)_{ij}h_j^{(\ell-1)}+(b_\ell)_i,
	\]
	we have
	\[
	\partial_v g_i^{(\ell)}
	=
	\sum_{j=1}^{m_{\ell-1}} \bigl(\partial_v (W_\ell)_{ij}\bigr) h_j^{(\ell-1)}
	+
	\sum_{j=1}^{m_{\ell-1}} (W_\ell)_{ij}\,\partial_v h_j^{(\ell-1)}
	+\partial_v(b_\ell)_i.
	\]
	Here $\partial_v(W_\ell)_{ij}$ and $\partial_v(b_\ell)_i$ are either $0$ or $1$.
	Thus $\partial_v g_i^{(\ell)}$ is a polynomial in $(\bx,\btheta)$ and earlier chain elements of degree at most $1+3(\ell-1)=3\ell-2$.
	Finally,
	\[
	\partial_v h_i^{(\ell)}=(1-(h_i^{(\ell)})^2)\,\partial_v g_i^{(\ell)},
	\]
	so $\partial_v h_i^{(\ell)}$ is polynomial in $(\bx,\btheta)$ and chain elements up to $h_i^{(\ell)}$ of degree at most $(3\ell-2)+2=3\ell$.
	This shows that the chain is of degree $\gamma=3L$, as claimed. Since $\phi=\sum_{i=1}^{m_L}(\bW_{L+1})_ih_i^{(L)}+b_{L+1}$ is a degree-two polynomial in the parameters and chain elements, $\phi$ is Pfaffian of format $(\gamma,2,\ell_0)$.

	We now bound, uniformly over $1\le u,z\le p$ (with $p=\size$), the number of connected components of
	$\mathbf{T}^{-1}([\epsilon_1,\dots,\epsilon_u])$ for regular values, where $\mathbf{T}$ is defined in \eqref{eqn:TPaff}: Fix $z$ points $\bx_1,\dots,\bx_z$ and define $\bar\phi_r(\btheta)=\phi(\bx_{j_r},\btheta)$ for $r=1,\dots,u$. As we've just established, $(\bx,\btheta)\mapsto \phi(\bx,\btheta)$ is Pfaffian on $\R^{d+p}$ with format $(\gamma,2,\ell_0)$.
	Freezing $\bx=\bx_j$ preserves the Pfaffian structure, so each $\bar\phi_r:\R^p\to\R$ is Pfaffian of format $(\gamma,2,\ell_0)$. Concatenating the $z$ chains corresponding to the $z$ evaluation points yields a common Pfaffian chain for
	$\bar\phi_1,\dots,\bar\phi_u$ of length
	\[
	\ell := z\,\ell_0 \le 2z\sum_{\ell'=1}^L m_{\ell'} \le 2pL\|\bm\|_\infty.
	\]
	Since translating by constants does not change the format, we may apply Theorem~\ref{thm:gabrielov} to the system
	$\bar\phi_1(\btheta)-\epsilon_1=0,\dots,\bar\phi_u(\btheta)-\epsilon_u=0$ to obtain a bound
	\[
	B(z)
	\le
	2^{\frac{\ell(\ell-1)}{2}+1}\,(\gamma+3)^{p-1}\,
	\Bigl((2p-1)(\gamma+2)-2p+2\Bigr)^{\ell}.
	\]
	The right-hand side is nondecreasing in $\ell$, hence in $z$, and it does not depend on the particular choice of points nor on the regular value.
	Therefore, taking $z=p$ yields a single constant $B:=B(p)$ that satisfies the uniformity requirement in Theorem~\ref{thm:karpinsky97} for all $1\le u,z\le p$.
	Applying Theorem~\ref{thm:karpinsky97} with $\bar s=1$ concludes the proof.
	\end{proof}

\subsubsection{Higher order derivatives}

We now extend the above argument to mixed derivatives of arbitrary order. Let $r\ge 1$ and $\alpha\in\N^d$ be a multi-index with $\|\alpha\|_1\le r$. Recall the layerwise quantities $g_i^{(\ell)}(\bx,\btheta)$ and $h_i^{(\ell)}(\bx,\btheta)=\sigma(g_i^{(\ell)}(\bx,\btheta))$ defined by \eqref{eqn:tanh_recurrence}, where $\sigma=\tanh$. We begin by a simple lemma showing that Pfaffian functions are closed under differentiation with explicit control on their derivatives' format.

\begin{lemma}
	\label{lem:pfaff_diff}
	Let $f$ be Pfaffian on an open set $U\subseteq\R^n$ with format $(\gamma,\beta,\ell)$, given by a Pfaffian chain $f_1,\dots,f_\ell$ of degree $\gamma$ and a polynomial representation
	$f(\bx)=P(\bx,f_1(\bx),\dots,f_\ell(\bx))$ with $\deg P\le \beta$.
	Then for each coordinate $x_i$, the partial derivative $\partial_{x_i} f$ is Pfaffian with respect to the same chain and has format $(\gamma,\gamma+\beta-1,\ell)$.
\end{lemma}

\begin{proof}
	Follows by differentiating $P(\bx,f_1(\bx),\dots,f_\ell(\bx))$ and noting that each $\partial_{x_i} f_j$ is a polynomial in $(\bx,f_1,\dots,f_j)$ of degree at most $\gamma$.
\end{proof}

Lemma~\ref{lem:pfaff_diff} together with the analysis developed in the preceding section yields the following proposition:

\begin{proposition}[Input-derivatives of tanh FNNs are Pfaffian]
	\label{prop:higher_order_tanh_pfaff}
	Fix $r\ge 0$ and a multi-index $\alpha\in\N^d$ with $\|\alpha\|_1\le r$. Let $\phi\in\F$ and view $(\bx,\btheta)\mapsto \phi(\bx,\btheta)$ as an analytic function on $\R^{d+p}$.
	Then $(\bx,\btheta)\mapsto \partial_\bx^\alpha \phi(\bx,\btheta)$ is Pfaffian on $\R^{d+p}$ with respect to the same chain $\mathcal{C}$, of order $\ell_0\le 2L\|\bm\|_\infty$ and degrees $(\gamma,\beta_r)$, where $\gamma:=3L$ and $\beta_r:=2+r(\gamma-1)$.
\end{proposition}

\begin{proof}
	By repeating the argument used to prove Proposition~\ref{prop:karpinski_vcdim}, $\phi$ is Pfaffian of format $(\gamma,2,\ell_0)$.
	Applying Lemma~\ref{lem:pfaff_diff} iteratively $\|\alpha\|_1$ times with respect to the input variables then yields the claim.
\end{proof}

We are now ready to prove the claimed VC-dimension bound: fix $\alpha\in\N^d$ with $\|\alpha\|_1\le r$ and consider $\partial_\bx^\alpha\F:=\{\bx\mapsto \partial_\bx^\alpha\phi(\bx,\btheta):\btheta\in\R^p\}$.
We repeat the proof of Proposition~\ref{prop:karpinski_vcdim} with $\bar\phi_j(\btheta):=(\partial_\bx^\alpha\phi)(\bx_j,\btheta)$. By Proposition~\ref{prop:higher_order_tanh_pfaff}, each $\bar\phi_j$ is Pfaffian with degrees $(\gamma,\beta_r)$, and after concatenation over $z\le p$ sample points we obtain a common chain of length $\ell\le 2pL\|\bm\|_\infty$. Applying Theorems~\ref{thm:gabrielov} and \ref{thm:karpinsky97} thus yields
\[
\VCdim(\partial_\bx^\alpha\F)=\bigO_r\!\left(L^2\|\bm\|_\infty^2\size^2\right),
\]
where the dependence on $r$ only enters through $\beta_r$ inside logarithmic factors. By Lemma~\ref{lem:rho_invariance}, the same bound holds for the boundary-enforced tanh networks as well, as desired.\qed

It remains to derive the pseudodimension bounds used in the statistical error analysis:

\begin{proof}[Proof of Lemma~\ref{lem:pdim_from_vcdim}]
	By Lemma~\ref{lemma:pseudodim_vcdim}, we must count sign patterns for $\partial^\alpha h(x_i)-y_i$ at fixed pairs $(x_i,y_i)$, with $\|\alpha\|_1\le2$. For ReQU, Lemma~\ref{lem:rho_invariance} and the the partitions $\partition_L$ used in \eqref{eqn:requ_derivative_growth} make $\partial^\alpha h(x_i)$ polynomial in the network parameters, of degree at most $2^{L+1}-1$. Subtracting the fixed threshold $y_i$ leaves this degree unchanged. The sign count \eqref{eqn:requ_derivative_growth} and the same weighted AM-GM argument therefore give $\Pdim(\partial^\alpha\F)=\bigO\bigl(\size L[L+\log_2(2\|\bm\|_\infty)]\bigr)$ for $\|\alpha\|_1\le2$. For tanh, at each sample point, the derivatives share the chain $\mathcal C$ of Proposition~\ref{prop:higher_order_tanh_pfaff}, and subtracting $y_i$ changes only the constant term of their polynomial expressions in that chain. The component count used for Theorem~\ref{thm:vcdim_bounds} therefore applies in both cases.
\end{proof}

\subsection{Effect of the fixed boundary functions}
\label{sec:boundary_effect}

To conclude, we now extend these bounds to boundary-enforced networks.

\begin{lemma}
	\label{lem:rho_invariance}
	Let $r\ge2$ be an integer and let $g,\rho\in C^r(\bar\Omega)$ be independent of the network parameters. The VC-dimension bounds \eqref{eqn:vcdim_requ} and \eqref{eqn:vcdim_tanh} for derivatives of ReQU and tanh FNNs remain valid after replacing each network $\phi$ by $g+\rho\phi$. The same holds for the pseudodimension bounds \eqref{eqn:pseudodim_2}.
\end{lemma}

\begin{proof}
	At any fixed point $x$, the Leibniz rule gives $\partial^\alpha(g+\rho\phi)=\partial^\alpha g+\sum_{\beta\le\alpha}\binom{\alpha}{\beta}\partial^{\alpha-\beta}\rho\,\partial^\beta\phi$. The derivatives of $g$ and $\rho$ are constants independent of $\btheta$. For ReQU, all network derivatives in this sum are polynomial on the same regions $\partition_L$. Their linear combination has the same degree bound, so the sign count \eqref{eqn:requ_derivative_growth} still applies. For tanh, these derivatives use the common chain of Proposition~\ref{prop:higher_order_tanh_pfaff}. Their linear combination preserves its degree bounds, so Theorems~\ref{thm:gabrielov} and \ref{thm:karpinsky97} give the same VC-dimension estimate. Subtracting fixed thresholds changes only constant terms in both arguments, giving the stated pseudodimension bounds.
\end{proof}

%\section*{Acknowledgments}
%The authors thank the anonymous reviewers for their valuable suggestions. This work is supported in part by funds from the National Science Foundation (NSF: \# 1636933 and \# 1920920).

\section*{Funding}
N. Tepakbong was supported by the Hong Kong PhD Fellowship Scheme, which also funded a research visit to The University of Sydney where part of this work was carried out. J. Fan is partially supported by the Research Grants Council of Hong Kong [Project No. HKBU12302923 and HKBU12303024], and Guangdong and Hong Kong Universities “1+1+1” Joint Research Collaboration Scheme 2025A0505000007. X. Zhou is supported by the General Research Funds from the Research Grants Council of the Hong Kong Special Administrative Region, China (Project No. 11308323, 11304525).

\bibliographystyle{plainnat}
\bibliography{mean_escape_time_pinn.bib}

\end{document}